\let\oldvec\vec

\documentclass[invmat,numbook,final]{svjour}

\let\newvec\vec
\let\vec\oldvec
\usepackage{amsmath}
\usepackage{calc}
\usepackage[T1]{fontenc}
\usepackage[full]{textcomp}

\usepackage[utopia]{mathdesign}

\usepackage{mathtools}
\DeclareSymbolFont{cmcal}{OMS}{cmsy}{m}{n}
\SetSymbolFont{cmcal}{bold}{OMS}{cmsy}{b}{n}
\DeclareSymbolFontAlphabet{\mathcal}{cmcal}

\let\vec\newvec

\usepackage{pifont}

\usepackage[all,cmtip]{xy}
\usepackage{xcolor}
\usepackage{enumitem} 
\usepackage[parfill]{parskip}

\usepackage{mathtools}
\usepackage[colorlinks=true]{hyperref} 

\spnewtheorem{notation}[definition]{Notation}{\bfseries}{\itshape}

\usepackage[super]{nth}

\usepackage{etoolbox}
\tracingpatches
\makeatletter
\patchcmd{\@citeo}{\hskip0.1em}{\kern0.1em}{}{}
\patchcmd{\@citex}{\hskip0.1em}{\kern0.1em}{}{}
\makeatother

\newcommand{\eq}[1]{eq.\ (\ref{#1})}

\newcommand{\eqn}[2]{
\begin{equation}\label{#1}
#2
\end{equation}
}

\newcommand{\eqnalign}[2]{
	\begin{equation}
	\begin{aligned}\label{#1}
	#2
	\end{aligned}
	\end{equation}
}

\makeatletter
\newcommand*{\Xbar}[1]{}%
\DeclareRobustCommand*{\Xbar}[1]{%
  \mathpalette\@Xbar{#1}%
}
\newcommand*{\@Xbar}[2]{%
  \sbox0{$#1\mathrm{#2}\m@th$}%
  \sbox2{$#1#2\m@th$}%
  \rlap{%
    \hbox to\wd2{%
      \hfill
      $\overline{%
        \vrule width 0pt height\ht0 %
        \kern\wd0 %
      }$%
    }%
  }%
  \copy2 %
}
\newcommand*{\Frozenbar}[1]{}%
\DeclareRobustCommand*{\Frozenbar}[1]{%
  \mathpalette\@Frozenbar{#1}%
}
\newcommand*{\@Frozenbar}[2]{%
  \sbox0{$#1\mathrm{W}\m@th$}%
  \sbox2{$#1#2\m@th$}%
  \rlap{%
    \hbox to\wd2{%
      \hfill
      $\overline{%
        \vrule width 0pt height\ht0 %
        \kern\wd0 %
      }$%
    }%
  }%
  \copy2 %
}
\makeatother

\newcommand{\eqalign}[1]{
	\begin{align*}#1 \end{align*}
	}

\usepackage{isomath}  
\SetMathAlphabet{\mathbf}{normal}{OML}{mdput}{b}{n}

\def\mb#1{{\mathbf{#1}}}
\def\bm#1{{\boldsymbol{#1}}}

\usepackage{tocbibind}

\def\a{\alpha}
\def\b{\beta}
\def\c{\chi}
\def\d{\delta}  
 \def\ep{\epsilon}
  
\def\g{\gamma}  \def\G{\Gamma}

\def\l{\lambda}  
\def\m{\mu}

\def\r{\rho}

\def\o{\omega}  \def\O{\Omega}
\def\p{\psi}  
\def\s{\sigma}  \def\vs{\varsigma} 
  
\def\t{\tau}
\def\w{\varphi}

\def\ups{\upsilon}

\def\CE{{\cal E}}

\def\CG{{\cal G}}

\def\CP{{\cal P}}

\def\Z{\mathbb{Z}}

\def\Q{\mathbb{Q}}
\def\I{\mathbb{I}}

\def\mI{\mathfrak{I}}

\def\mq{\mathfrak{q}}

\def\mP{\mathfrak{P}}
\def\mp{\mathfrak{p}}

\def\mr{\mathfrak{r}}
\def\ms{\mathfrak{s}}

\def\ml{\mathfrak{l}}

\def\mG{\mathfrak{G}}

\def\sN{\mathscr{N}}

\def\sQ{\mathscr{Q}}

\def\sS{\mathscr{S}}

\newcommand{\Fr}[2]{\dfrac{#1}{#2}}

\def\rd{\partial}
\def\pr{\prime}
\def\ppr{{\prime\!\prime}}

\DeclareMathOperator{\End}{End}
\DeclareMathOperator{\Aut}{Aut}
\DeclareMathOperator{\Ker}{Ker}

\DeclareMathOperator{\Hom}{Hom}

\usepackage{graphicx}
\makeatletter
\DeclareRobustCommand*\uell{\mathpalette\@uell\relax}
\newcommand*\@uell[2]{
  \setbox0=\hbox{$#1\ell$}
  \setbox1=\hbox{\rotatebox{10}{$#1\ell$}}
  \dimen0=\wd0 \advance\dimen0 by -\wd1 \divide\dimen0 by 2
  \mathord{\lower 0.1ex \hbox{\kern\dimen0\unhbox1\kern\dimen0}}
}

\makeatletter
\def\widebreve{\mathpalette\wide@breve}
\def\wide@breve#1#2{\sbox\z@{$#1#2$}%
     \mathop{\vbox{\m@th\ialign{##\crcr
\kern0.06em\brevefill#1{0.6\wd\z@}\crcr\noalign{\nointerlineskip}%
                    $\hss#1#2\hss$\crcr}}}\limits}
\def\brevefill#1#2{$\m@th\sbox\tw@{$#1($}%
  \hss\resizebox{#2}{\wd\tw@}{\rotatebox[origin=c]{90}{\upshape(}}\hss$}
\makeatletter

\def\fieldk{\Bbbk}

\def\cstar{\hbox{\scriptsize\ding{75}}}

\newcommand{\functor}[1]{\mathbf{#1}}

\newcommand{\gdcat}[1]{\underline{\category{#1}}\,}
\newcommand{\category}[1]{\emph{\textbf{#1}}\,}

\newcommand{\HOM}{\hbox{$\mathop{\bm{\Hom}}$}}

\DeclareMathOperator{\THom}{THom}
\newcommand{\THOM}{\hbox{$\mathop{\bm{\THom}}$}}

\newcommand\isoto{\xrightarrow{
   \,\smash{\raisebox{-0.65ex}{\ensuremath{\scriptstyle\sim}}}\,}}

\usepackage{pifont}

\def\cp{{\mathrel\vartriangle}}

\def\dga{{\category{dgA}(\Bbbk)}\!}

\def\ccdgc{{\category{ccdgC}(\Bbbk)}}
\def\hccdgc{{\mathit{ho}\category{ccdgC}(\Bbbk)}\!}

\def\ccdgh{{\category{ccdgH}(\Bbbk)}}
\def\hccdgh{{\mathit{ho}\category{ccdgH}(\Bbbk)}\!}

\DeclareMathAlphabet{\mathpzc}{OT1}{pzc}{m}{it}
\def\ide{\mathpzc{e}}

\begin{document}

\title{Representable presheaves of groups on the homotopy category of  cocommutative 
dg-coalgebras and  Tannakian reconstruction}

\titlerunning{Representable Presheaf of Groups and Tannakian Reconstruction}
\author{
Jaehyeok Lee\thanks{jhlee1228lee@postech.ac.kr} \and
Jae-Suk Park\thanks{jaesuk@postech.ac.kr}
}

\institute{
Department of Mathematics, POSTECH, Pohang 37673, Republic of Korea
}

\maketitle
\begin{abstract}
Motivated by rational homotopy theory,
we study a representable presheaf of groups $\bm{\mP}$
on the homotopy category  of cocommutative differential graded coalgebras,
its Lie algebraic counterpart and its linear representations.
We prove a Tannaka type reconstruction theorem
that $\bm{\mP}$ can be recovered from the dg-category of its linear representations
along with the forgetful dg-functor to  the underlying dg-category  of chain complexes. 
\end{abstract}
{MSC (2010): 20G05,16T05;14L15,55P62.}

{\footnotesize
\tableofcontents
}

\section{Introduction}
Throughout this paper $\Bbbk$ is a fixed ground field of characteristic zero,
and every differential graded (dg) object has homological $\Z$-grading, unless otherwise specified.

\subsection{A selected history and motivation}
There is a well-known picture connecting cocommutative Hopf algebras,  abstract groups 
and Lie algebras.\footnote{We refer to \cite{Cartier}
for a more extensive history.}
We obtain a  group  from a cocommutative Hopf  algebra as the group of group-like elements
and a cocommutative Hopf algebra
from an abstract group $\G$ as its group ring $\fieldk \G$, which case  $\G$ 
is isomorphic to the group of group-like elements.
The category $\category{Rep}(\G)$ of linear representations of $\G$  is  isomorphic to the category of left modules over $\Bbbk \G$ 
as tensor categories. 
Cocommutative Hopf algebras have
a similar correspondence with Lie algebras via the Lie algebra of primitive elements 
and the universal enveloping algebra. 
A group ring $\fieldk \G$ can be completed by  the powers of its augmentation ideal to a complete  Hopf algebra $\widehat{\fieldk \G}$.
For a complete Hopf algebra the group of group-like
elements is determined by the Lie algebra of primitive elements, and vice versa \cite{Quillen}.

This classic picture can be naturally recast
in the context of a representable presheaf of groups
$\category{P}: \mathring{\category{ccC}}(\Bbbk)\rightsquigarrow \category{Grp}$
on the category $\category{ccC}(\Bbbk)$ of cocommutative coalgebras.\footnote{
A formal group is a representable presheaf of groups on ${\category{ccC}}(\Bbbk)$
satisfying certain  conditions \cite{Dieud}.
}
A representing object of $\category{P}$ is a cocommutative Hopf algebra $H$,
where  the group $\category{P}(C)$ for each cocommutative coalgebra $C$
is isomorphic to the group formed by the set 
of all coalgebra maps from $C$ to $H$ with the convolution product as the composition.
In particular the group of group-like elements in $H$ is isomorphic to the group
$\category{P}(\Bbbk^\vee)$, where $\Bbbk^\vee$ is the ground field 
as a cocommutative coalgebra.
Conversely an abstract group $\G$ determines a presheaf  $\category{P}_{\Bbbk\G}$ of groups on 
${\category{ccC}}(\Bbbk)$ represented by $\Bbbk\G$, and we have
$\category{P}_{\Bbbk\G}(\Bbbk^\vee)\cong \G$.  We can also form the tensor
category of linear representations of $\category{P}$, which is isomorphic to the tensor
category of left modules over $H$, and reconstruct $\category{P}$ 
from the forgetful fiber functor.
The usual description of the Lie algebra of primitive elements in $H$
can also be recast more functorially by associating a presheaf
of Lie algebras $\category{TP}: \mathring{\category{ccC}}(\Bbbk)\rightsquigarrow \category{Lie}(\Bbbk)$
on ${\category{ccC}}(\Bbbk)$ to each  representable presheaf of groups $\category{P}$ such that  $(a)$
the Lie algebra $\category{TP}(\Bbbk^\vee)$ is isomorphic to the Lie algebra of primitive elements in $H$ representing $\category{P}$;
$(b)$ we have a natural isomorphism $\category{TP}\cong \category{TP}^\pr$
whenever the presheaves of groups $\category{P}$ and $\category{P}^\pr$ are isomorphic;
and $(c)$ we have a pair of natural isomorphisms $\grave{\category{TP}}\cong \grave{\category{P}}:  \mathring{\category{ccC}}(\Bbbk)\rightsquigarrow \category{Set}$,
between the  presheaves underlying ${\category{TP}}$ and $\category{P}$,
whenever the presheaf of groups ${\category{P}}$  is pro-represented by a complete Hopf algebra.

The purpose of this paper is to expand the above picture one-step further by studying 
a representable  presheaf of groups 
$\bm{\mP}:\mathit{ho}\mathring{\category{ccdgC}}(\fieldk)\rightsquigarrow \category{Grp}$
on the homotopy category $\mathit{ho}\category{ccdgC}(\fieldk)$ of cocommutative dg coalgebras (ccdg-coalgebras),
its Lie algebraic counterpart,  their linear representations and a Tannaka type reconstruction theorem.

In general,  a representable presheaf  of groups on a category 
can be regarded as  a group object in the category.
A fundamental example is 
the presheaf of groups  $\bm{\Pi}_{X_*}: \mathring{\mathit{ho}\category{Top}}_* \rightsquigarrow \category{Grp}$ 
on the homotopy category $\mathit{ho}\category{Top}_*$ 
of pointed topological spaces represented by the based loop space $\O X_*$ 
of a pointed space $X_*$, where the group $\bm{\Pi}_{X_*}(Y_*)$ for every
based space $Y_*$ is the group formed by the set  $[Y_*, \O X_*]$ of homotopy types of all 
base point  preserving continuous  maps to $\O X_*$
so that we have $\bm{\Pi}_{X_*}(S^n_*)\cong\pi_{n+1}(X_*)$ for $n\geq 0$. 
As it seems that  a full understanding of $\bm{\Pi}_{X_*}$ is out of reach, 
we may follow the ideas of rational homotopy theory of Quillen \cite{Quillen} 
and Sullivan \cite{Sullivan}
to replace $\mathit{ho}\category{Top}_*$ with the rational homotopy category 
$\mathit{ho}_{\mathbb{Q}}\category{Top}_*$ and consider suitable full subcategories.
For example,  Quillen has considered  the full subcategory 
$\mathit{ho}_{\mathbb{Q}}\category{Top}(2)_*$ of $1$-connected pointed spaces
and constructed a full embedding  
$\bm{\sQ}:\mathit{ho}_{\mathbb{Q}}\category{Top}(2)_*\rightsquigarrow \mathit{ho}\category{ccdgC}(\mathbb{Q})_*$
to the homotopy category $\mathit{ho}\category{ccdgC}(\mathbb{Q})_*$ 
of coaugmented ccdg-coalgebras over $\mathbb{Q}$. 
This gives us the motivation to develop a general theory of representable  presheaves
of groups on the homotopy category $\mathit{ho}\category{ccdgC}(\fieldk)$.

The categorical dual to a  representable  presheaf of groups
$\bm{\mP}:\mathit{ho}\mathring{\category{ccdgC}}(\fieldk)\rightsquigarrow \category{Grp}$
on $\mathit{ho}{\category{ccdgC}}(\fieldk)$ is a representable functor
$\bm{\mG}:\mathit{ho}{\category{cdgA}}(\fieldk)\rightsquigarrow \category{Grp}$
from the homotopy category $\mathit{ho}{\category{cdgA}}(\fieldk)$ 
of commutative dg-algebras (cdg-algebra) with cohomological grading.
This is the dg-version of an affine group scheme $G$ over $\Bbbk$, which we call  an affine group dg-scheme.
The study of affine group schemes and their linear representation is a classic subject in algebraic geometry,
which has led to the theory of Tannakian categories \cite{Rivano,Deligne90}.  A neutral Tannakian category is equivalent to
the category of finite dimensional linear representations of an affine group scheme along with the forgetful functor to
the category of underlying finite dimensional vector spaces.  
We expect to have  similar constructions for affine group dg-schemes, 
which is the main subject of a sequel to this paper \cite{JLP}.
Rational homotopy theory  gives us  additional motivation for our study, since Sullivan has constructed  
a \emph{contravariant} full embedding  
$\bm{\sS}:\mathit{ho}_{\mathbb{Q}}\mathring{\category{Top}}(1)^{f\!n}_*\rightsquigarrow 
\mathit{ho}\category{cdgA}(\mathbb{Q})_*$
of  the rational homotopy category ${\mathit{ho}_{\mathbb{Q}}\category{Top}}(1)^{f\!n}_*$ 
of $0$-connected \emph{nilpotent} pointed  spaces of \emph{finite types}
into the homotopy category $\mathit{ho}\category{cdgA}(\mathbb{Q})_*$ 
of augmented cdg-algebras over $\mathbb{Q}$ with the cohomological grading \cite{Sullivan}.

\subsection{Results}

Let $\bm{\mP}:\mathring{\mathit{ho}\category{ccdgC}}(\fieldk)\rightsquigarrow \category{Grp}$ 
be a representable presheaf of groups
on the homotopy category $\mathit{ho}\category{ccdgC}(\Bbbk)$ 
of ccdg-coalgebras over $\Bbbk$.  

A representing object of $\bm{\mP}$
is a cocommutative dg Hopf algebra (ccdg-Hopf algebra)  $\O$,  and we 
use the notation $\bm{\mP}_{\O}$ for it.
For each ccdg-coalgebra $C$
the group $\bm{\mP}_{\O}(C)$  is the group formed
by the set of homotopy types of all morphisms $g: C\rightarrow\O$ of ccdg-coalgebras,
and a homotopy equivalence $f:C\rightarrow C^\pr$ of ccdg-coalgebras
induces an isomorphism 
$\bm{\mP}_{\O}(C^\pr) \isoto\bm{\mP}_{\O}(C)$ of groups.
The category of representable presheaves of groups on  
$\mathit{ho}\category{ccdgC}(\fieldk)$
is equivalent to the homotopy category $\mathit{ho}\category{ccdgH}(\fieldk)$ 
of ccdg-Hopf algebras. [\emph{Theorem \ref{reppresheaf}}]

The Lie theoretic counterpart of  $\bm{\mP}_{\!\O}$ is a  presheaf  
$\category{T}\bm{\mP}_{\!\O}:
\mathring{\mathit{ho}\category{ccdgC}}(\fieldk)\rightsquigarrow \category{Lie}(\Bbbk)$ 
of Lie algebras  over $\Bbbk$ 
on the homotopy category $\mathit{ho}\category{ccdgC}(\fieldk)$, defined so that
we have a natural isomorphism 
$\category{T}\bm{\mP}_{\!\O}\cong \category{T}\bm{\mP}_{\!\O^\pr}$ 
whenever we have a natural isomorphism $\bm{\mP}_{\!\O}\cong \bm{\mP}_{\!\O^\pr}$ or, equivalently, 
whenever $\O$ and $\O^\pr$ are homotopy equivalent as ccdg-Hopf algebras. [\emph{Theorem \ref{Liereppresheaf}}]

If $\O$ is concentrated in degree zero, 
the group $\bm{\mP}_{\!\O}(\Bbbk^\vee)$ is isomorphic to the group of group-like elements in $\O$
and
the Lie algebra 
$\category{T}\bm{\mP}_{\!\O}(\Bbbk^\vee)$ is isomorphic to the Lie algebra of primitive elements in $\O$.

A complete ccdg-Hopf algebra is the dg-version of Quillen's complete cocommutative Hopf algebra.
If $\O$ is a complete ccdg-Hopf algebra,
we construct a natural isomorphism $\xymatrix{\grave{\bm{T\!\mP}}_{\O}\ar@/^0.3pc/@{=>}[r]
& \ar@/^0.3pc/@{=>}[l] \grave{\bm{\mP}}_{\O}}:
\mathring{\mathit{ho}\category{ccdgC}}(\fieldk)\rightsquigarrow \category{Set}$
between the underlying presheaves,
so that the representable presheaf  $\bm{\mP}_{\O}$ of groups can be recovered from the presheaf  
$\bm{T\!\mP}_{\O}$ of Lie algebras by  the Baker-Campbell-Hausdorff formula. [\emph{Theorem \ref{dgquillen}}]

We use the chain model for dg-categories---categories enriched in the category of chain complexes
over $\Bbbk$.  
We define a linear representation of $\bm{\mP}_\O$ via a linear representation of
the associated presheaf of groups 
$\bm{\CP}_{\!\!\O}:\mathring{\category{ccdgC}}(\fieldk)\rightsquigarrow \category{Grp}$ 
on the category $\category{ccdgC}(\fieldk)$ of ccdg-coalgebras, which is represented by $\O$
and induces $\bm{\mP}_\O$ on the homotopy category $\mathit{ho}\category{ccdgC}(\fieldk)$.
The linear representations   of $\bm{\CP}_{\!\!\O}$ form a dg-tensor category
$\gdcat{Rep}(\bm{\CP}_{\!\!\O})$, which is isomorphic 
to the dg-tensor category $\gdcat{dgMod}_L(\O)$ of 
left dg-modules over $\O$.

We reconstruct $\bm{\mP}_\O$ via  the forgetful dg-functor
$
\bm{\o}:\gdcat{dgMod}_L(\O)\rightsquigarrow \gdcat{Ch}(\Bbbk)
$
to the underlying dg-category $\gdcat{Ch}(\Bbbk)$ of chain complexes
as follows.
\begin{itemize}
\item
We consider a dg-tensor functor 
$C\otimes: \gdcat{Ch}(\Bbbk)\rightsquigarrow\gdcat{dgComod}^{\mathit{cofr}}_L(C)$
for each ccdg-coalgebra $C$, 
which sends a chain complex to the cofree left dg-comodule  over $C$ cogenerated 
by the chain complex. Composing it with $\bm{\o}$, we have a dg-tensor functor  
$C\otimes \bm{\o}: \gdcat{dgMod}_L(\O)\rightsquigarrow\gdcat{dgComod}^{\mathit{cofr}}_L(C)$.
Then, it is trivial that the  set  $\mathsf{End}\big(C\!\otimes\!\bm{\o}\big)$ of natural endomorphisms 
of the functor $C\!\otimes\! \bm{\o}$ is  a dg-algebra over $\Bbbk$.

\item
By considering the subset $\mathsf{Z_0End}^\otimes\big(C\!\otimes\! \bm{\o}\big)$ of 
$\mathsf{End}\big(C\!\otimes\! \bm{\o}\big)$ 
consisting  of tensorial natural endomorphisms belonging to the kernel
of the natural differential, we construct a presheaf of \emph{groups}
$\bm{\CP}^\otimes_{\!\!\bm{\o}}: \mathring{\category{ccdgC}(\Bbbk)}\rightsquigarrow 
\category{Grp}$  on the
category $\category{ccdgC}(\Bbbk)$ of ccdg-coalgebras and prove that it is 
represented by the ccdg-Hopf algebra $\O$.

\item
After introducing a notion of homotopy types of elements in
$\mathsf{Z_0End}^\otimes\big(C\otimes \bm{\o}\big)$ 
analogous to the homotopy types of morphisms
of ccdg-coalgebras, we show that the presheaf of groups $\bm{\CP}^\otimes_{\!\!\bm{\o}}$ 
on $\category{ccdgC}(\Bbbk)$ induces a presheaf of groups
$\bm{\mP}^\otimes_{\!\!\bm{\o}}$  on the homotopy category $\mathit{ho}\category{ccdgC}(\Bbbk)$.
We, then, construct a natural isomorphism 
$$
\xymatrix{ \bm{\mP}^\otimes_{\!\!\bm{\o}}\ar@/^0.3pc/@{=>}[r]^{}  
& \ar@/^0.3pc/@{=>}[l]^{} \bm{\mP}_\O}:
\mathring{\mathit{ho}\category{ccdgC}}(\Bbbk)\rightsquigarrow\category{Grp}
$$
of presheaves of groups on the homotopy category ${\mathit{ho}\category{ccdgC}}(\Bbbk)$,
which is our Tannakian reconstruction. [\emph{Theorem \ref{homainth}}]
\end{itemize}
 
A typical example of a ccdg-Hopf algebra $\O$ is the complete tensor dg-Hopf algebra
generated by a $\Z$-graded vector space, via the cobar construction \cite{Adams} of a
ccdg-coalgebra or a $C_\infty$-coalgebra. 
If $X_*$ is a pointed $2$-connected and pointed space, the cobar construction of the ccdg-coalgebra
$\sQ(X_*)$, the rational Quillen model for $X_*$,  is a rational Quillen model for the based loop space $\O X_*$.
Let $\Q\bm{\Pi}_{X_*}: \mathbb{Q}\mathit{ho}\category{Top}(1)_*\rightsquigarrow\category{Grp}$ be
a presheaf of groups on $\mathbb{Q}\mathit{ho}\category{Top}(1)_*$ represented by $\O X_*$.
Then we have an isomorphism $\Q\bm{\Pi}_{X_*}(Y_*) \cong \category{P}_{\O(\sQ(X_*)}\big(\sQ(Y_*)\big)$ of groups
for every $1$-connected pointed space $Y_*$ since Quillen's functor $\sQ$ is a full embedding. 
 

\begin{acknowledgement}
The work of JL was supported by NRF(National Research Foundation of Korea) 
Grant funded by the Korean Government(NRF-2016-Global Ph.D. Fellowship Program).
JSP is grateful to Cheolhyun Cho, Gabriel Drummond Cole and Dennis Sullivan for useful comments.
\end{acknowledgement}

\section{Notation and basic notions}

Unadorned tensor product $\otimes$ is over the ground field $\fieldk$.
By an element  of a $\Z$-graded vector space 
we shall usually mean a homogeneous element $x$ whose degree will be denoted $|x|$. 
Let $V=\bigoplus_{n\in \Z}V_n$ and $W=\bigoplus_{n\in \Z}W_n$ be $\Z$-graded vector spaces.
Then $V\otimes W = \bigoplus_{n\in \Z}(V \otimes W)_n$,  
where $(V \otimes W)_n= \bigoplus_{i+j=n\in \Z}V_i \otimes W_j$, 
is also a  $\Z$-graded vector  space.
Denoted by  $\Hom(V,W)= \bigoplus_{n\in \Z}\Hom(V,W)_n$ is the $\Z$-graded vector 
space of $\fieldk$-linear maps from $V$ to $W$, where $\Hom(V,W)_n$ 
is the space of $\fieldk$-linear maps increasing the degrees  by $n$.
A chain complex $\big(V, \rd_V\big)$ is often denoted by $V$ for simplicity.
The ground field $\fieldk$ is a chain complex with the zero differential.
Let $V$ and $W$ be chain complexes. 
Then $V\otimes W$ 
and $\Hom(V, W)$ are also chain complexes with the following differentials
\eqn{tshmdiff}{
\begin{cases}
\rd_{V\otimes  W} = \rd_V\otimes \I_W +\I_V\otimes \rd_W, &
\cr
\rd_{V\!,W} f= \rd_W\circ f -(-1)^{|f|}f\circ \rd_V, & \forall f\in \Hom(V,W)_{|f|}
.
\end{cases}
}
A \emph{chain map} $f:\big(V, \rd_V\big)\rightarrow \big(W, \rd_W\big)$ is an $f \in \Hom(V,W)_0$
satisfying $\rd_{\mathit{V\!,W}} f =\rd_V\circ f -f\circ \rd_W =0$.  
Two chain maps $f$ and $\tilde f$ are \emph{homotopic}, denoted by $f \sim \tilde f$,
or \emph{have the same homotopy type}, denoted by $[f]=[\tilde f]$, if there is a chain homotopy  $\l \in \Hom(V,W)_{1}$
such that $\tilde f - f = \rd_{\mathit{V\!,W}} \l$.

The set of morphisms from an object $C$ to another object $C^\pr$ in a category 
$\category{C}$ is denoted by $\HOM_{\category{C}}(C, C^\pr)$. 
We denote  the set of natural transformations of functors 
$\category{F}\Rightarrow \category{G}:\category{C}\rightsquigarrow \category{D}$
by $\mathsf{Nat}(\category{F}, \category{G})$.
For any functor $\category{F}:\category{C}\rightsquigarrow \category{D}$,
where $\category{D}$ is small, we use the notation 
$\grave{\category{F}}: \category{C}\rightsquigarrow \category{Set}$ for the
underlying set valued functor obtained by composing it with the forgetful functor 
$\functor{Forget}:\category{D}\rightsquigarrow \category{Set}$.
A \emph{presheaf of groups on  a category 
$\category{C}$} is a functor $\category{F}:\mathring{\category{C}}\rightsquigarrow \category{Grp}$
from the opposite category $\mathring{\category{C}}$ of  $\category{C}$ to the category $\category{Grp}$ of groups
and is called \emph{representable} if $\grave{\category{F}}$ is representable.

The canonical isomorphisms $\fieldk\otimes V\cong V$  and $V\otimes \fieldk \cong V$
will be denote by $\imath_V: \fieldk\otimes V\rightarrow  V$ and $\imath^{-1}_V:V\rightarrow \fieldk\otimes V$
as well as  $\jmath_V: V\otimes\fieldk\rightarrow V$ and $\jmath^{-1}_V:V\rightarrow   V\otimes \fieldk$.

A \emph{dg-algebra} on  a $\Z$-graded vector space $A$ is 
a tuple $A=\big(A, u_A, m_A, \rd_A\big)$, which is both a chain complex $\big(A, \rd_A\big)$
and a $\Z$-graded associative algebra $\big(A, u_A, m_A\big)$ 
such that both the unit $u_A: \fieldk \rightarrow  A$
and the product $m_A:A\otimes A\rightarrow A$ are chain maps:
\eqnalign{dgalgebra}{
\begin{cases}
\rd_A\circ u_A=0
,\cr
m_A\circ \rd_{A\otimes A} = \rd_{A}\circ m_A
,
\end{cases}
\quad
\begin{cases}
m_A\circ (u_A\otimes \I_A)=\imath_A\cong  m_A\circ (\I_A\otimes u_A)=\jmath_A,
\cr
m_A\circ (m_A\otimes \I_A)= m_A\circ (\I_A\circ m_A)
.
\end{cases}
}

A morphism $f:A \rightarrow A^\pr$ of dg-algebras is simultaneously a chain map, $f\circ \rd_A =\rd_{A^\pr}\circ f$, and an unital algebra
map,  $f\circ u_A= u_{A^\pr}$ and  $f\circ m_A = m_{A^\pr}\circ(f\otimes f)$.
The dg-algebras  form a category, denoted by $\category{dgA}(\fieldk)$,
where the composition of morphisms is the composition as linear maps.

A \emph{dg-coalgebra} on $\Z$-graded vector space $C$  is 
a tuple $C=\big(C, \ep_C, \cp_C, \rd_C\big)$, which is both a chain complex $\big(C, \rd_C\big)$
and a $\Z$-graded  coassociative coalgebra $\big(C, \ep_C, \cp_C\big)$
such that both the counit $\ep_C: C\rightarrow \fieldk$
and the coproduct $\cp_C:C\rightarrow  C\otimes C$ are chain maps:
\eqnalign{dgcoalgebra}{
\begin{cases}
\ep_C\circ \rd_C=0
,\cr
\cp_C\circ \rd_{C} = \rd_{C\otimes C}\circ\cp_C
,
\end{cases}
\quad
\begin{cases}
(\ep_C\otimes \I_C)\circ \cp_C=\imath^{-1}_C\cong  (\I_C\otimes \ep_C)\circ \cp_C=\jmath^{-1}_C,
\cr
(\cp_C\otimes \I_C)\circ \cp_C=(\I_C\otimes \cp_C)\circ \cp_C
.
\end{cases}
}

A morphism $f:C \rightarrow C^\pr$ of dg-coalgebras is simultaneously a chain map, 
$f\circ \rd_C =\rd_{C^\pr}\circ f$, and a counital coalgebra
map, $\ep_{C^\pr}\circ f=\ep_{C}$ and $\cp_{C^\pr}\circ f =(f\otimes f)\circ \cp_C$.
The dg-coalgebras  form a category, denoted by $\category{dgC}(\fieldk)$,
where the composition of morphisms is the composition as linear maps.

Every dg-coalgebra $C$ in this paper is \emph{cocommutative} that 
$\cp_C = \t\circ \cp_C$, where  $\t(x\otimes y) = (-1)^{|x||y|}y\otimes x$, $\forall x, y \in C$.
The full subcategory of cocommutative dg-coalgebras (ccdg-coalgebras) of $\category{dgC}(\fieldk)$ 
is denoted by $\category{ccdgC}(\fieldk)$, and we use the prefix "cc" for cocommutative.

Remark that the ground field $\fieldk$ is 
an algebra $\fieldk=(\fieldk, u_\fieldk, m_\fieldk)$, where  $u_\fieldk =\I_\fieldk$
and $m_\fieldk(a\otimes b)=a\cdot b$, 
and  a coalgebra $\fieldk^\vee=(\fieldk, \ep_\fieldk, \cp_\fieldk)$ 
with $\ep_\fieldk =\I_\fieldk$ and $\cp_\fieldk(1)= 1\otimes 1$.
A  \emph{ccdg-bialgebra} $\O=\big(\O, u_\O, m_\O, \ep_\O, \cp_\O, \rd_\O\big)$
is simultaneously  a dg-algebra $\big(\O, u_\O, m_\O, \rd_\O\big)$ and  a ccdg-coalgebra 
$\big(\O, \ep_\O, \cp_\O, \rd_\O\big)$ such that
both the counit $\ep_\O:\O\rightarrow \fieldk$ and the coproduct $\cp_\O:\O\rightarrow \O\otimes \O$
are morphisms of dg-algebras:
\eqnalign{bialgebra}{
\begin{cases}
\ep_\O\circ u_\O = u_\fieldk
,\cr
\cp_\O\circ u_\O=  (u_\O\otimes u_\O)\circ \cp_\fieldk
,
\end{cases}
\quad
\begin{cases}
\ep_\O\circ m_\O=m_\fieldk\circ (\ep_\O\otimes \ep_\O)
,\cr
\cp_\O\circ m_\O=m_{\O\otimes \O}\circ (\cp_\O\otimes \cp_\O)
,
\end{cases}
}
or, equivalently, both the unit $u_\O:\fieldk\rightarrow \O$ and the product $m_\O:\O\otimes \O \rightarrow \O$
are morphisms of ccdg-coalgebras---remind that 
$m_{\O\otimes \O}\circ (\cp_\O\otimes \cp_\O)=
(m_\O\otimes m_\O)\circ (\I_\O\otimes \t\otimes \I_\O)\circ (\cp_\O\otimes \cp_\O)
=(m_\O\otimes m_\O)\circ \cp_{\O\otimes \O}$.

A {\it ccdg-Hopf algebra}  $\O$ is a ccdg-bialgebra with an antipode $\vs_\O$, 
which is a linear map $\vs_\O:\O\rightarrow \O$ of degree zero satisfying
the following axiom:
\eqn{antipodeaxiom}{
m_\O\circ(\vs_\O\otimes \I_\O)\circ\cp_\O = m_\O\circ (\I_\O\otimes \vs_\O)\circ\cp_\O = u_\O\circ \ep_\O.
}
Then, $\vs_\O$  is automatically a chain map.
Also, antipode $\vs_\O$ of ccdg-bialgebra is unique if exists and both an anti-algebra map and coalgebra map:
\eqnalign{antialgandcoalg}{
\begin{cases}
\vs_\O\circ u_\O=u_\O
,\cr 
\vs_\O \circ m_\O = m_\O\circ(\vs_\O\otimes \vs_\O)\circ \t
,
\end{cases}
\qquad
\begin{cases}
\ep_\O\circ \vs_\O=\ep_\O
,\cr
\cp_\O \circ \vs_\O = (\vs_\O\otimes \vs_\O)\circ\cp_\O
.
\end{cases}
}
A morphism $f:\O\rightarrow \O^\pr$ of  ccdg-Hopf algebras  is simultaneously a morphism 
of dg-algebras and a morphism
of ccdg-coalgebras---it is, then, automatic that  $f$ commutes with the antipodes.
The ccdg-Hopf algebras  form a category $\category{ccdgH}(\fieldk)$,
where the composition of morphisms is the composition as linear maps.

We can also form the homotopy category of ccdg-Hopf algebras, for which we introduce the notion of homotopy type
of  ccdg-Hopf algebra morphisms. 
\begin{definition}
A homotopy pair on $\HOM_{\ccdgh}(\O, \O^\pr)$
is a pair of $1$-parameter families $\big(f(t), \xi(t)\big)\in \Hom(\O, \O^\pr)_0[t]\oplus\Hom(\O, \O^\pr)_{1}[t]$,
which is parametrized by time variable $t$ with polynomial dependences
and
satisfies the \emph{homotopy flow equation} $\Fr{d}{dt}f(t)=\rd_{\O, \O^\pr} \xi(t)$ generated by $\xi(t)$,
subject to the following two types of conditions:
\begin{itemize}
\item \emph{infinitesimal algebra map}:  $f(0) \in  \HOM_{\dga}(\O, \O^\pr)$ and 
$$
\xi(t)\circ u_\O=0, \qquad
\xi(t) \circ m_\O =m_{\O^\pr}\circ \big(f(t)\otimes \xi(t) +\xi(t)\otimes f(t)\big)
.
$$
\item\emph{infinitesimal coalgebra map}:  $f(0) \in  \HOM_{\ccdgc}(\O, \O^\pr)$
and
$$
\ep_{\O^\pr}\circ\xi(t)=0,\qquad
\cp_{\O^\pr}\circ\xi(t)=\big(f(t)\otimes \xi(t) +\xi(t)\otimes f(t)\big)\circ \cp_\O
.
$$
\end{itemize}
\end{definition}

Let $\big(f(t), \xi(t)\big)$ be a homotopy pair on $\HOM_{\ccdgh}(\O, \O^\pr)$. 
By the homotopy flow equation,  ${f}(t)$ is uniquely determined by $\xi(t)$ modulo an initial condition $f(0)$ such that
${f}(t) = {f}(0) +\rd_{\O, \O^\pr}\int^t_0\xi(s)\mathit{ds}$,  and we can check that $f(t)$ is a family of morphisms of ccdg-Hopf algebras.
In this case, we say \emph{${f}(1)$ is homotopic to ${f}(0)$ by the homotopy $\int^1_0\xi(t)\mathit{dt}$}
and denote ${f}(0)\sim {f}(1)$, which is clearly an equivalence relation.
In other words,  two morphisms $f$ and $\tilde f$  of ccdg-Hopf algebras are homotopic, $f\sim \tilde f$,
if there is a homotopy pair connecting them (by the time $1$ map).  
Then,  we also say that $f$ and $\tilde f$ \emph{have the same homotopy type}, 
denoted by $[f]=[\tilde f]$.
For any diagram 
$\xymatrix{\O\ar@/^/[r]^-{f}\ar@/_/[r]_-{\tilde f}&\O^\pr\ar@/^/[r]^-{f^\pr}\ar@/_/[r]_-{\tilde f^\pr}&\O^\ppr}$ 
in the category $\category{ccdgH}(\fieldk)$, where $f\sim \tilde f$ and $f^\pr\sim \tilde f^\pr$, 
it is straightforward to check that   $f^\pr\circ f \sim \tilde f^\pr\circ \tilde f \in \HOM_{\ccdgh}(\O, \O^\pr)$
and the homotopy type of $f^\pr\circ f$ depends  only on the homotopy types of $f$ and $f^\pr$, so that we have
the well-defined  associative composition  $[f^\pr]\circ_h [f] :=[f^\pr\circ f]$ of homotopy types.
A morphism  $\xymatrix{\O\ar[r]^f & \O^\pr}$ of ccdg-Hopf algebras is a \emph{homotopy equivalence} if there
is a morphism  $\xymatrix{\O & \ar[l]_{h} \O^\pr}$ of  ccdg-Hopf algebras from the opposite direction
such that $h\circ f \sim \I_\O$ and $f\circ h \sim \I_{\O^\pr}$.

The homotopy category $\mathit{ho}\category{ccdgH}(\fieldk)$ of ccdg-Hopf algebras is defined such
that the objects are ccdg-Hopf algebras and morphisms are homotopy types of morphisms of ccdg-Hopf algebras.  
Note that a homotopy equivalence of ccdg-Hopf algebras is an isomorphism in the homotopy category 
$\mathit{ho}\category{ccdgH}(\fieldk)$. 

We define a homotopy pair on te morphisms of ccdg-coalgebras  as the case of ccdg-Hopf algebras  
but without imposing the condition for infinitesimal algebra map.
Then,  we have corresponding notions for homotopy types of morphisms of ccdg-coalgebras and 
a homotopy equivalence of ccdg-coalgebras.
Thus we can form the homotopy category 
$\mathit{ho}\category{ccdgC}(\fieldk)$ of ccdg-coalgebras, 
whose morphisms are homotopy types of morphisms of ccdg-coalgebras.

Remark that our notion of homotopy category is the \emph{naive} one---all 
based on  chain complex over a field $\fieldk$ with explicitly defined homotopies of morphisms.\footnote{
Our definition of the homotopy category of ccdg-coalgebras is equivalent to Quillen's definition in  \cite{Quillen}
if we consider the full subcategory of $2$-reduced ccdg-coalgebras.}

A dg-category $\gdcat{C}$   over $\Bbbk$
is a category enriched in the category $\category{Ch}(\fieldk)$ 
of \emph{chain} complexes over $\fieldk$---we refer to \cite{Keller} for a review of dg-categories. 
A dg-category shall be distinguished from an ordinary category by putting an "underline".
Besides from using the chain model we follows \cite{Simpson} for the notion of dg-tensor categories.
The chain complex of morphisms from object $X$ to object $Y$ in a dg-category $\gdcat{C}$
is denoted by $\HOM_{\gdcat{C}}(X,Y)$ with the differential $\rd_{\HOM_{\gdcat{C}}(X,Y)}$.
A morphism $f \in \HOM_{\gdcat{C}}(X,Y)$ between two objects $X$ and $Y$ in $\gdcat{C}$  
is an \emph{isomorphism} if  
$f\in \HOM_{\gdcat{C}}(X,Y)_0$  and satisfies $\rd_{\HOM_{\gdcat{C}}(X,Y)}f=0$, with its inverse
$g \in \HOM_{\gdcat{C}}(Y,X)_0$ satisfying $\rd_{\HOM_{\gdcat{C}}(Y,X)}g=0$.

A dg-functor $\category{F}:\gdcat{C} \rightsquigarrow \gdcat{D}$ 
is a functor which induces chain maps
$\HOM_{\gdcat{C}}(X,Y)\to\HOM_{\gdcat{D}}(\functor{F}(X),\functor{F}(Y))$
for every pair $(X,Y)$ of objects in $\gdcat{C}$.
The set ${\mathsf{Nat}}(\category{F}, \category{G})$ of natural transformations of dg-functors 
is a chain complex $\big({\mathsf{Nat}}(\category{F}, \category{G}), \bm{\d}\big)$,
where
\begin{itemize}
\item
its degree $n$ element is a collection of  morphisms 
$\eta=\{\eta_X:\functor{F}(X)\to\functor{G}(X)|X\in \text{Ob}(\gdcat{C})\}$
of degree $n$,
where $\eta_X$ is called the \emph{component of $\eta$ at $X$}, 
with the supercommuting naturalness condition, i.e. $\functor{G}(f)\circ\eta_X=(-1)^{mn}\eta_Y\circ\functor{F}(f)$ 
for every morphism $f:X\to Y$ of degree $m$.

\item
for every $\eta \in {\mathsf{Nat}}(\category{F}, \category{G})$ of degree $n$ 
we have $\d\eta \in {\mathsf{Nat}}(\category{F}, \category{G})$ of degree $n-1$,
whose component at $X$ is defined by $\big(\d\eta\big)_X:=\rd_{\HOM_{\gdcat{C}}\big(\functor{F}(X),\functor{G}(X)\big)}\eta_X$,
and $\d\circ \d=0$.
\end{itemize}
The dg-functors from $\gdcat{C}$ to $\gdcat{D}$ form a dg-category, with morphisms 
as the above natural transformations. 
In particular, the set $\mathsf{End}(\functor{F})$ of natural endomorphisms  has a canonical structure of dg-algebra.
A natural  transformation $\eta$ 
from a dg-functor $\functor{F}$ to another dg-functor
$\functor{G}$ is often indicated by a diagram $\eta: \functor{F}\Rightarrow \functor{G}: \gdcat{C}\rightsquigarrow \gdcat{D}$.
A natural transformation $\eta$ is an \emph{(natural) isomorphism} 
if the component morphism  $\eta_X:\functor{F}(X)\rightarrow \functor{G}(X)$ is an isomorphism in
$ \gdcat{D}$ for every object $X$ of  $\gdcat{C}$.

The notion of tensor categories \cite{Rivano,Deligne90} has a natural generalization to dg-tensor categories.
For a dg-category $\gdcat{C}$ we have a new dg-category $\gdcat{C}\boxtimes \gdcat{C}$, whose objects
are pairs denoted by $X\boxtimes Y$ and whose Hom complexes are the tensor products of
Hom complexes of  $\gdcat{C}$, i.e., $\HOM_{\gdcat{C}\boxtimes \gdcat{C}}(X\boxtimes Y, X^\pr\boxtimes Y^\pr)
=\HOM_{\gdcat{C}}(X, X^\pr)\otimes \HOM_{\gdcat{C}}(Y, Y^\pr)$ with the natural composition operation and 
differentials.  Then  we have a natural equivalence of dg-categories  
$(\gdcat{C}\boxtimes \gdcat{C})\boxtimes \gdcat{C} \cong \gdcat{C}\boxtimes (\gdcat{C}\boxtimes \gdcat{C})$.
A dg-category
$\gdcat{C}$ is a \emph{dg-tensor category} if 
we have dg-functor $\bm{\otimes}: \gdcat{C}\boxtimes \gdcat{C}\rightsquigarrow \gdcat{C}$
and a unit object $\bm{1}_{\gdcat{C}}$ satisfying the associativity, the commutativity and the unit axioms.
(See pp $40$-$41$ in  \cite{Simpson} for the details.)

The fundamental example of dg-tensor 
category over $\Bbbk$ is the dg-category $\gdcat{Ch}(\Bbbk)$ of chain complexes, whose 
set of morphisms $\HOM_{\gdcat{Ch}(\Bbbk)}(V,W)$
from a chain complex $V$ to a chain complex $W$
is the Hom complex $\Hom(V,W)$ with the differential $\rd_{\HOM_{\gdcat{Ch}(\Bbbk)}(V,W)}=\rd_{V\!,W}$.  
The dg-functor $\bm{\otimes}: \gdcat{Ch}(\Bbbk)\boxtimes \gdcat{Ch}(\Bbbk)\rightsquigarrow \gdcat{Ch}(\Bbbk)$
sends $(V, \rd_V)\boxtimes(W, \rd_W)$ to the  chain complex $(V\otimes W, \rd_{V\otimes W})$ 
and an unit object is the ground field $\Bbbk$ as a chain complex $(\Bbbk,0)$,
where all coherence isomorphisms  are the obvious ones.

A \emph{dg-tensor functor} $\functor{F}:\gdcat{C}\rightsquigarrow\gdcat{D}$ 
between dg-tensor categories is a dg-functor 
satisfying
$\functor{F}(X\bm{\otimes} Y)\cong\functor{F}(X)\bm{\otimes}\functor{F}(Y)$
and $\functor{F}(\bm{1}_{\gdcat{C}})\cong
\bm{1}_{\gdcat{D}}$.
A \emph{tensor natural transformation} $\eta:\functor{F}\Rightarrow\functor{G}$ 
between dg-tensor functors is a natural transformation  of degree $0$
satisfying $\eta_{X\bm{\otimes} Y}\cong\eta_X\otimes\eta_Y$ 
and $\eta_{\bm{1}_{\gdcat{C}}}\cong\I_{\bm{1}_{\gdcat{D}}}$.

We use the  notation $[\a]$ for the homotopy type of a morphism 
$\a$ as well as for the homology class of  a cycle $\a$, depending on the context.

\section{Representable presheaves of groups and presheaves of Lie algebras}

In this section, we develop basic theory of a
representable presheaf of groups  
$\bm{\mP}:\mathring{\mathit{ho}\category{ccdgC}}(\fieldk)\rightsquigarrow \category{Grp}$
on the homotopy category ${\mathit{ho}\category{ccdgC}}(\fieldk)$ and its Lie algebraic counterpart.

In Sect.\ 3.1,  we check that a representing object of $\bm{\mP}$
is a ccdg-Hopf algebra  $\O$,  and we use the notation $\bm{\mP}_{\O}$ for it.
The group $\bm{\mP}_{\O}(C)$ for each ccdg-coalgebra $C$ is the group formed
by the set of homotopy types of all morphisms $g: C\rightarrow\O$ of ccdg-coalgebras.
Moreover, a homotopy equivalence $f:C\rightarrow C^\pr$ of ccdg-coalgebras
induces an isomorphism  $\bm{\mP}_{\O}(C^\pr) \isoto\bm{\mP}_{\O}(C)$ of groups.
The category of representable presheaves of groups on  $\mathit{ho}\category{ccdgC}(\fieldk)$
is equivalent to the homotopy category $\mathit{ho}\category{ccdgH}(\fieldk)$ of ccdg-Hopf algebras.
We observe that the group  $\bm{\mP}_{\O}(\Bbbk^\vee)$, where $\Bbbk^\vee$ is the dual coalgebra on $\Bbbk$,
acts naturally on $\bm{\mP}_{\O}(C)$ so that the underlying presheaf 
$\grave{\bm{\mP}}_{\O}:\mathring{\mathit{ho}\category{ccdgC}}(\fieldk)\rightsquigarrow \category{Set}$
is $\bm{\mP}_{\O}(\Bbbk^\vee)$-set valued. If $\O$ is concentrated in degree zero, the group
$\bm{\mP}_{\!\O}(\Bbbk^\vee)$ is isomorphic to the group of group-like elements in $\O$.

In Sect.\ 3.2, we define  the Lie theoretic counterpart to  $\bm{\mP}_{\!\O}$  by 
a  presheaf of Lie algebras 
$\category{T}\bm{\mP}_{\!\O}:\mathring{\mathit{ho}\category{ccdgC}}(\fieldk)\rightsquigarrow \category{Lie}(\Bbbk)$ 
on the homotopy category $\mathit{ho}\category{ccdgC}(\fieldk)$.
The Lie algebra $\category{T}\bm{\mP}_{\O}(C)$ for each ccdg-coalgebra $C$ is the Lie algebra formed
by the set of homotopy types of all  \emph{infinitesimal} morphisms 
$\ups: C\rightarrow\O$ of ccdg-coalgebras about the
identity element of the group $\bm{\mP}_{\O}(C)$.
We have a natural isomorphism 
$\category{T}\bm{\mP}_{\!\O}\cong \category{T}\bm{\mP}_{\!\O^\pr}$ 
whenever there is a natural isomorphism $\bm{\mP}_{\!\O}\cong \bm{\mP}_{\!\O^\pr}$ or, equivalently, 
whenever $\O$ and $\O^\pr$ are homotopy equivalent as ccdg-Hopf algebras. 
If $\O$ is concentrated in degree zero, the Lie algebra 
$\category{T}\bm{\mP}_{\!\O}(\Bbbk^\vee)$ is isomorphic to the Lie algebra of primitive elements in $\O$.

In Sect.\ 3.3, we consider a pro-representable presheaf of group $\bm{\mP}_{\O}$, whose representing object 
$\O$ is  a complete ccdg-Hopf algebra $\O=\widehat{\O}$.
Complete ccdg-Hopf algebra is the dg-version of Quillen's complete (cocommutative) Hopf algebra.
We construct a natural isomorphism 
$
\xymatrixcolsep{3pc}
\xymatrix{ \grave{\bm{\mP}}_{{\O}}\ar@/^0.3pc/@{=>}[r]^{\bm{\ln}}  
& \ar@/^0.3pc/@{=>}[l]^{\bm{\exp}} \grave{\bm{T\mP}}_{{\O}}}
:\mathring{\mathit{ho}\category{ccdgC}}(\Bbbk)\rightsquigarrow \category{Set}
$
between the underlying presheaves,
so that the representable presheaf of groups $\bm{\mP}_{\O}$  can be recovered from the presheaf of Lie algebras
$\bm{T\!\mP}_{\O}$ using  the Baker-Campbell-Hausdorff formula.

\subsection{Representable presheaf  of groups 
$\bm{\mP}_{\O}:\mathring{\mathit{ho}\category{ccdgC}}(\fieldk)\rightsquigarrow \category{Grp}$}

The purpose of this subsection is to prove the following:

\begin{theorem}[Definition]\label{reppresheaf}
For each ccdg-Hopf algebra $\O$,
we have 
a representable presheaf  of groups  
$\bm{\mP}_{\!\Omega}:\mathring{\mathit{ho}\category{ccdgC}}(\fieldk)\rightsquigarrow \category{Grp}$
on the homotopy category $\mathit{ho}\category{ccdgC}(\fieldk)$ of ccdg-coalgebras 
over $\Bbbk$,
sending

\begin{itemize}
\item
a  ccdg-coalgebra $C=\big(C, \ep_C,\cp_C, \rd_C\big)$ to the group
$$
\bm{\mP}_{\!\Omega}(C):=\Big(\HOM_{\hccdgc} \big(C,\O\big), \ide_{C\!,\O}, \ast_{C\!,\O} \Big),
$$
with the group operation 
$[g_1]\ast_{C\!,\O}  [g_2] 
:=\big[m_\O\circ (g_1\otimes g_2) \circ\cp_C\big]$,  
the identity element $\ide_{C\!,\O} =\big[ u_\O\circ \ep_C\big]$,  
and  the inverse $[g]^{-1}:=\big[\vs_\O\circ g\big]$ of $[g]$,
where $g_i \in \HOM_{\ccdgc} \big(C,\O\big)$ is an arbitrary representative 
of the homotopy type $[g_i]\in \HOM_{\hccdgc} \big(C,\O\big)$.

\item 
a morphism $[f] \in \HOM_{\hccdgc}\big(C, C^\pr\big)$ in the homotopy category of ccdg-coalgebras 
to a homomorphism
$\bm{\mP}_{\O}([f]): \bm{\mP}_{\O}\big(C^\pr\big)\rightarrow \bm{\mP}_{\O}\big(C\big)$ of groups
defined by, $\forall [g^\pr] \in \HOM_{\hccdgc} \big(C^\pr,\O\big)$,
$$
\bm{\mP}_{\Omega}\big([f]\big)\big([g^\pr]\big):=\big[g^\pr\circ f\big],
$$
where  $f\in \HOM_{\ccdgc}\big(C, C^\pr\big)$
and $g^\pr \in \HOM_{\ccdgc} \big(C^\pr,\O\big)$
are arbitrary representatives of the homotopy types $[f]$ and $[g^\pr]$, respectively,
\end{itemize}

such that
$\bm{\mP}_{\!\Omega}([f])$ is an isomorphism of groups 
whenever $f:C \rightarrow C^\pr$ is a homotopy equivalence of ccdg-coalgebras.

Representing object of a representable  presheaf  of groups  
$\bm{\mP}:\mathring{\mathit{ho}\category{ccdgC}}(\fieldk)\rightsquigarrow \category{Grp}$
on  $\mathit{ho}\category{ccdgC}(\fieldk)$ is a ccdg-Hopf algebra $\O$ such that $\bm{\mP}\cong \bm{\mP}_\O$.

For each morphism  $[\p] \in \HOM_{\hccdgh}\big(\O,\O^\pr\big)$ in the homotopy category of ccdg-Hopf algebras
we have a natural transformation 
$\sN_{[\p]}: \bm{\mP}_{\!\Omega}\Longrightarrow \bm{\mP}_{\!\Omega^\pr}:
\mathring{\mathit{ho}\category{ccdgC}}(\fieldk)\rightsquigarrow \category{Grp}$ 
defined  such that  for every ccdg-coalgebra $C$
and $[g] \in \HOM_{\hccdgc}\big(C, \O\big)$ we have
$$
\sN_{[\p]}^{C}\big([g]\big):= \big[\psi\circ g\big],
$$
where 
$\p \in \HOM_{\ccdgh}\big(\Omega, \Omega^\pr\big)$ 
and $g \in \HOM_{\ccdgc}\big(C, \O\big)$ 
are  arbitrary representatives of the homotopy types $[\p]$
and $[g]$, respectively.

We have an one-to-one correspondence
$\HOM_{\hccdgh}\big(\O,\O^\pr\big)\cong  \mathsf{Nat}\big(\bm{\mP}_{\!\Omega},\bm{\mP}_{\!\Omega^\pr}\big)$
such that $\sN_{[\p]}$ is a natural isomorphism whenever $\p:\O \rightarrow \O^\pr$ is 
a homotopy equivalence of ccdg-Hopf algebras.

\end{theorem}

We divide the proof into few pieces.
The main technical point is that
the presheaf of groups $\bm{\mP}_\O$ is defined via 
the associated presheaf of groups 
$\bm{\CP}_{\!\!\O}:\mathring{\category{ccdgC}}(\fieldk)\rightsquigarrow \category{Grp}$ 
on the category $\category{ccdgC}(\fieldk)$ of ccdg-coalgebras, which is represented by $\O$
and induces $\bm{\mP}_\O$ on the homotopy category $\mathit{ho}\category{ccdgC}(\fieldk)$.

\begin{lemma}[Definition]\label{presheafone}
For every ccdg-Hopf algebra $\O$
we have a presheaf of dg-algebras 
$\bm{\CE}_{\!\!\O}:\mathring{\category{ccdgC}}(\Bbbk) \rightsquigarrow \category{dgA}(\Bbbk)$ over $\Bbbk$ 
on ${\category{ccdgC}}(\Bbbk)$,
sending
each ccdg-coalgebra $C$ to the convolution dg-algebra
$\bm{\CE}_{\!\O}(C):=\big(\Hom\big(C, \O\big), u_\O\circ\ep_C, \star_{C\!,\O}, \rd_{C\!,\O}\big)$, 
where 
\begin{itemize}
\item  $u_\O\circ {\ep}_C:\xymatrix{C\ar[r]^{\ep_C} &\fieldk \ar[r]^{u_\O} & \O}$ is the unit, and
\item
$\a_1\star_{C\!,\O} \a_2 
:=m_\O\circ (\a_1\otimes \a_2)\circ \cp_C
:\xymatrix{C\ar[r]^-{\cp_C}& C\otimes C\ar[r]^-{\a_1\otimes \a_2} &\O\otimes \O \ar[r]^-{m_\O} &\O}$,
$\a_1, \a_2 \in  \Hom(C,\O)$, is the convolution product,
\end{itemize}
 and
each morphism $f:C\rightarrow C^\pr$ of ccdg-coalgebras
to a morphism $\bm{\CE}_{\!\!\O}(f): \bm{\CE}_{\!\!\O}(C^\pr)\rightarrow \bm{\CE}_{\!\!\O}(C)$
of dg-algebras defined by $\bm{\CE}_{\!\O}(f)(\a^\pr):= \a^\pr\circ f$ for all $\a^\pr \in \Hom\big(C^\pr, \O\big)$.
\end{lemma}

\begin{proof}
It is a standard fact  that $\bm{\CE}_{\!\O}(C)$ is a dg-algebra:
The convolution product $\star_{C\!,\O}$ is associative due to
the associativity of $m_\O$ and the coassociativity of $\cp_C$,
and $u_\O\circ\ep_C$ is the identity element for the convolution product
due to the counit axiom of $C\xrightarrow{\ep_C}\Bbbk$ and the unit axiom of $\Bbbk\xrightarrow{u_\O}\O$.
We have $\rd_{C\!,\O} (u_\O\circ\ep_C)=0$ by $\ep_C\circ\rd_C =\rd_\O\circ u_\O=0$
and $\rd_{C\!,\O}$ is a derivation of the convolution product since $\rd_C$ is a coderivation
of $\cp_C$ and $\rd_\O$ is a derivation of $m_\O$.

We check that $\bm{\CE}_{\!\O}(f)$ is a morphism of
dg-algebras as follows: $\forall \a^\pr,\a_1^\pr, \a_2^\pr \in \Hom\big(C^\pr, \O\big)$,
\eqalign{
\bm{\CE}_{\!\!\O}(f)(u_\O\!\circ\!\ep_{C^\pr}) = & u_\O\!\circ\!\ep_{C^\pr}\circ f = u_\O\!\circ\!\ep_C
,\cr
\bm{\CE}_{\!\!\O}(f)\left(\rd_{C^\pr\!,\O}\a^\pr\right) = & \rd_{\O}\circ \a^\pr\circ f -(-1)^{|\a^\pr|}\a^\pr\circ \rd_{C^\pr}\circ f 
=\rd_{\O}\circ \a^\pr\circ f - (-1)^{|\a^\pr|}\a^\pr \circ f \circ \rd_{C} 
\cr
=& \rd_{C^\pr\!,\O}\left(\bm{\CE}_{\!\!\O}(f)(\a^\pr)\right)
,\cr
\bm{\CE}_{\!\!\O}(f)\left(\a_1^\pr\star_{C^\pr\!,\O}\a_2^\pr\right) 
= & m_\O\circ (\a_1^\pr\otimes \a_2^\pr)\circ \cp_{C^\pr}\circ f
= m_\O\circ (\a_1^\pr\circ f \otimes \a_2^\pr\circ f)\circ \cp_{C} 
\cr
=&\bm{\CE}_{\!\!\O}(f)\left(\a_1^\pr\right)\star_{C^\pr\!,\O}\bm{\CE}_{\!\!\O}(f)\left(\a_2^\pr\right),
}
where we have use the defining properties of   $f$  being a morphism of ccdg-coalgebras. 
The functoriality of $\bm{\CE}_{\!\O}$ is obvious.
\qed
\end{proof}

\begin{lemma}\label{presheaftwo} 
For every ccdg-Hopf algebra $\O$
we have a representable presheaf of groups  
$\bm{\CP}_{\!\!\O}:\mathring{\category{ccdgC}}(\Bbbk) \rightsquigarrow \category{Grp}$ 
on ${\category{ccdgC}}(\Bbbk)$,
sending
\begin{itemize}
\item
each ccdg-coalgebra $C$ to a group 
$\bm{\CP}_{\!\O}(C):=\big(\HOM_{\ccdgc}\big(C, \O\big), u_\O\circ\ep_C, \star_{C\!,\O}\big)$,
where  the inverse of $g\in \HOM_{\ccdgc}\big(C, \O\big)$ is $g^{-1}:=\vs_\O\circ g$, and 
\item
each morphism $f:C\rightarrow C^\pr$ of ccdg-coalgebras
to a morphism $\bm{\CP}_{\!\!\O}(f): \bm{\CP}_{\!\!\O}(C^\pr)\rightarrow \bm{\CP}_{\!\!\O}(C)$
of groups defined by $\bm{\CP}_{\!\O}(f)(g^\pr):= g^\pr\circ f$ for all $g^\pr \in \HOM_{\ccdgc}\big(C^\pr, \O\big)$.
\end{itemize}
\end{lemma}

\begin{proof}
1. We show that $\bm{\CP}_{\!\O}(C)$ is a group for every ccdg-coalgebra $C$ as follows.
We remind that
$$
 \HOM_{\ccdgc}\big(C, \O\big):=\Big\{ g \in \Hom\big(C, \O\big)_0\Big|
 \rd_{C\!,\O}g =0,\;
 {\ep}_\O\circ g= {\ep}_C,\;
 \cp_\O\circ g=(g\otimes g)\circ \cp_C
\Big\},
$$
and check routinely that

\begin{itemize}
\item $u_\O\circ \ep_C\in\HOM_{\ccdgc}\big(C, \O\big)$, 
since $C\xrightarrow{\ep_C}\Bbbk$ and $\Bbbk\xrightarrow{u_\O}\O$ are morphisms of ccdg-coalgebras;

\item $g_1\star_{C\!,\O}g_2\in\HOM_{\ccdgc}\big(C, \O\big)$
whenever $g_1,g_2  \in \HOM_{\ccdgc}\big(C, \O\big)$, since we have
\eqalign{
\rd_{C\!,\O}(g_1\star_{C\!,\O}g_2) &=(\rd_{C\!,\O} g_1)\star_{C\!,\O}g_2 +g_1\star_{C\!,\O} \rd_{C\!,\O}g_2=0
,\cr
\ep_\O\circ (g_1\star_{C\!,\O}g_2) 
&= \ep_\O\circ m_\O\circ( g_1\otimes g_2)\otimes \cp_C 
=m_\Bbbk\circ( \ep_\O\circ g_1\otimes \ep_\O\circ g_2)\circ \cp_C
\cr
&
=m_\Bbbk\circ( \ep_C\otimes \ep_C)\circ \cp_C 
= m_\Bbbk\circ( \ep_C\otimes \I_\Bbbk)\circ \jmath^{-1}_C
=\ep_C,
}
and by the cocommutativity of $\cp_C$, we have
\eqalign{
\cp_\O\circ \Big(g_1\star_{C\!,\O} g_2\Big)
=
&
\cp_\O\circ m_\O\circ (g_1\otimes g_2)\circ \cp_C
\cr
=
&
m_{\O\otimes\O}\circ (\cp_\O\otimes \cp_\O)\circ (g_1\otimes g_2)\circ \cp_C
\cr
=
&
( m_\O\otimes m_\O)\circ (\I_\O\otimes \t\otimes \I_\O)
\circ (g_1\otimes g_1\otimes g_2\otimes g_2)\circ(\cp_C\otimes \cp_C)\circ \cp_C
\cr
=
&( m_\O\otimes m_\O)\circ (g_1\otimes g_2\otimes g_1\otimes g_2)
\circ (\I_C\otimes \t\otimes \I_C)\circ(\cp_C\otimes \cp_C)\circ \cp_C
\cr
=
&
( m_\O\otimes m_\O)\circ (g_1\otimes g_2\otimes g_1\otimes g_2)\circ (\cp_C\otimes \cp_C)\circ \cp_C
\cr
=&
\Big( (g_1\star_{C\!,\O}g_2)\otimes (g_1\star_{C\!,\O}g_2)\Big)\circ \cp_C
.
}

\item $g^{-1}:=\vs_\O\circ g\in \HOM_{\ccdgc}\big(C, \O\big)$ whenever $g\in \HOM_{\ccdgc}\big(C, \O\big)$,
since $\vs_\O:\O\to \O$ is a morphism of ccdg-coalgebras.

\end{itemize}
Then, by Lemma \ref{presheafone}, 
$\bm{\CP}_{\!\O}(C):=\big(\HOM_{\ccdgc}\big(C, \O\big), u_\O\circ\ep_C, \star_{C\!,\O}\big)$ is a monoid.
On the other hand, by the antipode axiom  \eq{antipodeaxiom} we have, $\forall g \in \HOM_{\ccdgc}\big(C, \O\big)$,
\eqalign{
g\star_{C\!,\O} g^{-1}
:= m_\O\circ (g\otimes \vs_\O\circ g)\circ\cp_C =m_\O\circ (\I_\O\otimes \vs_\O)\circ\cp_\O\circ g
= u_\O\circ \ep_\O\circ g = u_\O\circ \ep_C,
} 
and, similarly, $g^{-1}\star_{C\!,\O} g=u_\O\circ \ep_C$. Hence,  $\bm{\CP}_{\!\O}(C)$ is actually a group.

2.  We have $\bm{\CP}_{\!\O}(f)(g^\pr):=g^\pr\circ f \in \HOM_{\ccdgc}\big(C, \O\big)$ whenever 
$g^\pr\in \HOM_{\ccdgc}\big(C^\pr, \O\big)$ since $C\xrightarrow{f}C^\pr$ is a morphism of ccdg-coalgebras. 
Then, by Lemma \ref{presheafone},  we have
\begin{itemize}
\item
$\bm{\CP}_{\!\O}(f):\bm{\CP}_{\!\O}(C^\pr)\rightarrow\bm{\CP}_{\!\O}(C)$ is a group homomorphism;

\item $\bm{\CP}_{\!\O}(f^\pr\circ f)
= \bm{\CP}_{\!\O}(f)\circ \bm{\CP}_{\!\O}(f^\pr)$, $\forall f^\pr  \in \HOM_{\ccdgc}\big(C^\pr, C^\ppr\big)$.
\end{itemize}
Therefore
$\bm{\CP}_{\!\O}:\mathring{\category{ccdgC}}(\Bbbk)\rightsquigarrow \category{Grp}$ 
is  a representable presheaf of groups
on the category $\category{ccdgC}(\Bbbk)$ of ccdg-coalgebras.  
\qed
\end{proof}

\begin{remark} $\bm{\CP}_{\!\O}:\mathring{\category{ccdgC}}(\Bbbk)\rightsquigarrow \category{Grp}$ 
is a presheaf of groups even if  
$\O$ is not a strictly cocommutative 
dg-Hopf algebra but, then,  it is not a representable presheaf on ${\category{ccdgC}}(\Bbbk)$.
\end{remark}

\begin{lemma}\label{presheaffour}
Suppose $\bm{\CP}$ is a representable presheaf of groups 
$\mathring{\category{ccdgC}}(\Bbbk) \rightsquigarrow \category{Grp}$ on $\ccdgc$. 
Then $\bm{\CP}\cong\bm{\CP}_{\!\!\O}$ for some ccdg-Hopf algebra $\O$.
\end{lemma}

\begin{proof}
Since $\grave{\bm{\CP}}:\mathring{\category{ccdgC}}(\Bbbk) \rightsquigarrow \category{Set}$ is representable, 
we have $\grave{\bm{\CP}}\cong\HOM_\ccdgc(-,\O)$ for some ccdg-coalgebra $\O$. 
We shall show  that $\O$ carries a ccdg-Hopf algebra structure. 
We can restate the condition that $\grave{\bm{\CP}}$ factors through $\category{Grp}$\ as follows:
\begin{itemize}
\item For each ccdg-coalgebra $C$, there is  a  group structure on $\grave{\bm{\CP}}(C)$. 
In other words, there are three structure functions
$\m^C:\grave{\bm{\CP}}(C)\times\grave{\bm{\CP}}(C)\to\grave{\bm{\CP}}(C)$,
$e^C:\{*\}\to\grave{\bm{\CP}}(C)$ and $i^C:\grave{\bm{\CP}}(C)\to\grave{\bm{\CP}}(C)$ satisfying the group axioms.
\item For each morphism $f:C\to C^\pr$ of ccdg-coalgebras, 
the function $\grave{\bm{\CP}}(f):\grave{\bm{\CP}}(C^\pr)\to\grave{\bm{\CP}}(C)$ is a group homomorphism.
\end{itemize}
This is equivalent to the existence of natural transformations 
$\m:\grave{\bm{\CP}}\times\grave{\bm{\CP}}\to\grave{\bm{\CP}}$,
$e:\underline{\{*\}}\to\grave{\bm{\CP}}$, 
and
$i:\grave{\bm{\CP}}\to\grave{\bm{\CP}}$ satisfying the group axioms.
Here, $\underline{\{*\}}$ is a functor $\mathring{\category{ccdgC}}(\Bbbk)\to\category{Set}$ 
sending every ccdg-coalgebra $C$ to a one-point set $\{*\}$.

Let $\O\otimes \O$ be
the ccdg-coalgebra obtained by the tensor product of the ccdg-coalgebra $\O$.
We claim that there are  natural isomorphisms of presheaves
\eqn{isopre}{
\grave{\bm{\CP}}\times\grave{\bm{\CP}}\cong\HOM_\ccdgc(-,\O\otimes\O),
\qquad
\underline{\{*\}}\cong\HOM_\ccdgc(-,\Bbbk^\vee).
}
Then, by the Yoneda lemma, the natural transformations $\m$, $e$, and $i$
are completely determined by morphisms $m_\O:\O\otimes\O\to \O$, $u_\O:\Bbbk^\vee\to\O$ 
and $\vs_\O:\O\to \O$ of ccdg-coalgebras, respectively.
Applying the Yoneda lemma again,  a plain calculation shows that
\begin{enumerate}
\item $\m\circ(\m\times\I_{\grave{\bm{\CP}}})=\m\circ(\I_{\grave{\bm{\CP}}}\times \m)$ implies the associativity of $m_\O$.
\item $\m\circ(e\times\I_{\grave{\bm{\CP}}})=\m\circ(\I_{\grave{\bm{\CP}}}\times e)=\I_{\grave{\bm{\CP}}}$ 
implies the unit axiom of $u_\O$.
\item inverse element axiom of $i$ implies the antipode axiom of $\vs_\O$. 
\end{enumerate}
Therefore $\O$ has a ccdg-Hopf algebra structure $(\O,u_\O,m_\O,\ep_\O,\cp_\O,\vs_\O)$.

Now we check the claimed isomorphisms in \eq{isopre}.
Note that $\Bbbk^\vee$ is a terminal object in the category $\ccdgc$,
since  any morphism $C\to \Bbbk^\vee$  of ccdg-coalgebras has to be the counit $\ep_C$  by the counit axiom. 
Let  $C\otimes C^\pr$ be the ccdg-coalgebra obtained by the tensor product
of ccdg-coalgebras $C$ and $C^\pr$. Then we have the following projection morphisms of ccdg-coalgebras: 
\[
\pi_C:=
\xymatrix{
C\otimes C^\pr\ar[r]^-{\I_C\otimes\ep_{C^\pr}}&C\otimes \Bbbk\ar[r]^-{\jmath_C}&C
,
}
\quad\quad
\pi_{C^\pr}:=
\xymatrix{
C\otimes C^\pr\ar[r]^-{\ep_C\otimes\I_{C^\pr}}&\Bbbk\otimes C^\pr\ar[r]^-{\imath_{C^\pr}}&C^\pr
.
}
\]
For each ccdg-coalgebra $T$ we consider the function
$$
\HOM_\ccdgc(T,C\otimes C^\pr)\rightarrow\HOM_\ccdgc(T,C)\times \HOM_\ccdgc(T,C^\pr)
$$
defined by  $h\mapsto (\pi_C\circ h,\pi_{C^\pr}\circ h)$. We show that the function
is a bijection for every $T$ by constructing its inverse. 
Given morphisms $f:T\to C$, $g:T\to C^\pr$ of ccdg-coalgebras, define 
$\left<f, g\right>:=(f\otimes g)\circ\cp_T:T\to C\otimes C^\pr$. 
Then it is a morphism of ccdg-coalgebras since $\cp_T$ is cocommutative.
It is obvious that $\left<\pi_C\circ h,\pi_{C^\pr}\circ h\right>=h$, 
$\pi_C\circ\left<f, g\right>=f$ and $\pi_{C^\pr}\circ\left<f, g\right>=g$. 
Moreover, the above bijection is natural in $T\in\category{ccdgC}(\Bbbk)$.
Therefore we have constructed the claimed isomorphisms in  \eq{isopre}.

Let $x_1,x_2$ be the elements in $\grave{\bm{\CP}}(C)$ that correspond
to morphisms $g_1,g_2:C\to \O$ of ccdg-coalgebras via the bijection 
$\grave{\bm{\CP}}(C)\cong\grave{\bm{\CP}}_\O(C)$. 
Then, by the Yoneda lemma, $\m^C(x_1,x_2)\in\grave{\bm{\CP}}(C)$ corresponds to 
$m_\O\circ\left<g_1,g_2\right>=g_1\star_{C\!,\O}g_2\in\grave{\bm{\CP}}_\O(C)$, and
$e^C\in\grave{\bm{\CP}}(C)$ corresponds to $u_\O\circ \ep_C\in\grave{\bm{\CP}}_\O(C)$.
This shows that the bijection $\grave{\bm{\CP}}(C)\cong\grave{\bm{\CP}}_\O(C)$ is an isomorphism of groups.
Thus we have a natural isomorphism $\bm{\CP}\cong\bm{\CP}_{\!\!\O}$ of  the presheaves of groups 
on ${\category{ccdgC}}(\Bbbk)$.
\qed
\end{proof}

\begin{lemma}\label{presheafthree}
For every morphism $\p:\O \rightarrow \O^\pr$ of ccdg-Hopf algebras we have a natural transformation 
$\sN_{\p}: \bm{\CP}_\O \Longrightarrow \bm{\CP}_{\O^\pr}: 
\mathring{\category{ccdgC}}(\Bbbk) \rightsquigarrow \category{Grp}$
defined such that
its component $\sN^C_{\psi}: \bm{\CP}_{\!\Omega}(C)\longrightarrow\bm{\CP}_{\!\Omega^\pr}(C)$ 
group homomorphism
at each ccdg-coalgebra $C$ is  
$\sN_{\p}^C(\a) :=\p\circ\a$ for all $\a \in \HOM_{\ccdgc}\big(C, \O\big)$. 
We also have 
$\HOM_{\ccdgh}\big(\O, \O^\pr\big)\cong \mathsf{Nat}\big(\bm{\CP}_{\O},\bm{\CP}_{\O^\pr}\big)$.
\end{lemma}

\begin{proof}
For every $C$,
$\sN_{\p}^C$ is a group homomorphism since, $\forall g_1, g_2 \in \HOM_{\ccdgc}\big(C, \O\big)$,
\eqalign{
\sN_{\p}^C(g_1\star_{C\!,\O}g_2)&= \p\circ m_\O\circ (g_1\otimes g_2)\circ \cp_C
= m_{\O^\pr}\circ (\p\circ g_1\otimes \p\circ g_2)\circ \cp_C
\cr
&=
\sN_{\p}^C(g_1)\star_{C,\O^\pr}\sN_{\p}^C(g_2)
,\cr
\sN_{\p}^C(u_\O\circ \ep_C)&=\p\circ u_\O\circ \ep_C = u_{\O^\pr}\circ \ep_C.
}
For every morphism $f:C\rightarrow C^\pr$ of ccdg-coalgebras we have,
$\forall g^\pr \in \HOM_{\ccdgc}\big(C^\pr, \O\big)$,
\eqalign{
\sN_{\psi}^C\!\circ\! \bm{\CP}_{\O}(f)(g^\pr)= {\psi}\circ (g^\pr\!\circ f)=({\psi}\circ g^\pr)\circ\! f=
\bm{\CP}_{\O^\pr}(f)\!\circ\! \sN_{\psi}^{C^\pr}(g^\pr).
}
Therefore $\sN_{\psi} \in \mathsf{Nat}\big(\bm{\CP}_{\O}, \bm{\CP}_{\O^\pr}\big)$ whenever 
${\psi} \in \HOM_{\ccdgh}\big(\O, \O^\pr\big)$. Combined with the Yoneda lemma, we have 
$\mathsf{Nat}\big(\bm{\CP}_{\!\O}, \bm{\CP}_{\!\O^\pr}\big)\cong \HOM_{\ccdgh}\big(\O, \O^\pr\big)$.
\qed
\end{proof}

\begin{lemma}\label{cohprt}
Assume that
\begin{itemize}

\item $\big({g}(t),\chi(t)\big)$,  $\big({g}_1(t),\chi_1(t)\big)$ and  $\big({g}_2(t),\chi_2(t)\big)$
are homotopy pairs  in $\HOM_{\ccdgc}(C,\O)$;

\item $\big(f(t),\l(t)\big)$ is a homotopy pair in $\HOM_{\ccdgc}(C,C^\pr)$
and $\big({g}^\pr(t),\l^\pr(t)\big)$  is a homotopy pair in $\HOM_{\ccdgc}(C^\pr,\O)$;

\item
$\big(\p(t),\xi(t)\big)$ is a homotopy pair in $\HOM_{\ccdgh}(\O,\O^\pr)$.
\end{itemize}

Then we have following homotopy pairs
\begin{enumerate}[label=({\alph*})]

\item
$\big({g}_1(t)\star_{C\!,\O}{g}_2(t), \chi_1(t)\star_{C\!,\O}{g}_2(t) + {g}_1(t)\star_{C\!,\O}\chi_2(t) \big)$
on $\HOM_{\ccdgc}(C,\O)$.

\item
$\big(\vs_\O\circ {g}(t), \vs_\O\circ\chi(t)\big)$
on $\HOM_{\ccdgc}(C,\O)$.

\item
$\big(g^\pr(t)\circ{f}(t), g^\pr(t)\circ \l(t) +\chi^\pr(t)\circ {f}(t)\big)$ 
on $\HOM_{\ccdgc}(C,\O)$.

\item
$\big({\p}(t)\circ g(t), \p(t)\circ \chi(t)+ {\xi}(t)\circ g(t)\big)$ 
on $\HOM_{\ccdgc}(C,\O^\pr)$.
\end{enumerate}
\end{lemma}

\begin{proof}
These can be checked by routine computations, which are omitted for the sake of space.
\qed
\end{proof}

Now we are ready for the proof of  Theorem \ref{reppresheaf}.

\begin{proof}[Theorem \ref{reppresheaf}]
After Lemmas \ref{presheafone}, \ref{presheaftwo},  \ref{presheaffour}, \ref{presheafthree}
and \ref{cohprt},  we just need to check  few things to finish the proof.

1.  We check that the group $\bm{\mP}_{\!\O}(C)$ is well defined for every ccdg-coalgebra $C$.
\begin{itemize}
\item We have $g_1\star_{C\!,\O} g_2 \sim \tilde g_1\star_{C\!,\O} \tilde g_2 \in  \Hom_{\ccdgc}\big(C, \O\big)$
whenever $g_1\sim \tilde g_1, g_2\sim \tilde g_2\in \Hom_{\ccdgc}\big(C, \O\big)$: this follows from
Lemma \ref{cohprt}$(a)$.
 
\item  We have $g^{-1}\sim \tilde{g}^{-1} \in  \Hom_{\ccdgc}\big(C, \O\big)$
whenever $g\sim \tilde g \in \Hom_{\ccdgc}\big(C, \O\big)$: this follows from
Lemma \ref{cohprt}$(b)$.
\end{itemize}
Also followed is that the homotopy type
$[g_1\star_{C\!,\O} g_2]$ of $g_1\star_{C\!,\O} g_2$ depends only on the homotopy types
$[g_1], [g_2]  \in \HOM_{\hccdgc}\big(C, \O\big)$ of $g_1$ and $g_2$.
Therefore the group  $\bm{\mP}_{\!\O}(C)$ is well-defined.

2. We check that the homomorphism 
$\bm{\mP}_{\!\O}([f]):\bm{\mP}_{\!\O}(C^\pr)\rightarrow \bm{\mP}_{\!\O}(C)$ of groups
is well defined for every $[f]\in \HOM_\hccdgc(C,C^\pr)$.  Let $f\sim \tilde f \in \HOM_\ccdgc(C,C^\pr)$ and
$g^\pr\sim\tilde g^\pr \in   \HOM_\ccdgc(C^\pr,\O)$. Then, by Lemma \ref{cohprt}$(c)$, we have
$g\circ f \sim g\circ \tilde f\sim \tilde g\circ f\sim \tilde g\circ \tilde f \in \HOM_\ccdgc(C, \O)$
so that $\bm{\mP}_{\!\O}([f])([g]):=[g\circ f]$ depends only on the homotopy types
$[f]$ and $[g^\pr]$ of arbitrary representatives $f \in \HOM_{\ccdgc}\big(C, C^\pr\big)$ and
$g^\pr \in   \HOM_\ccdgc(C^\pr,\O)$.  Therefore $\bm{\mP}_{\!\O}([f])$ is a well defined group homomorphism.
It is obvious that $\bm{\mP}_{\!\O}([f])$ is an isomorphism of groups
whenever $f: C \rightarrow C^\pr$ is a homotopy equivalence of ccdg-coalgebras.

3.  We check that the natural transformation 
$\sN_{[\p]}: \bm{\mP}_{\!\O}\Longrightarrow \bm{\mP}_{\!\O^\pr}:
\mathit{ho}\category{ccdgC}(\fieldk)\rightsquigarrow \category{Grp}$ is well-defined
for every  $[{\psi}]\in \HOM_{\hccdgh}\big(\O, \O^\pr\big)$.
Let $\p\sim \tilde \p \in \HOM_\ccdgh(\O,\O^\pr)$ and
$g, \tilde g \in   \HOM_\ccdgc(C,\O)$. Then, by Lemma \ref{cohprt}$(d)$, we have
$$
\p\circ g \sim \p\circ \tilde{g}\sim \tilde{\p}\circ g\sim \tilde{\p}\circ \tilde{g} \in \HOM_\ccdgc(C, \O^\pr)
$$
so that $\sN^C_{[\p]}([g]):=[\p\circ g]$ for every ccdg-coalgebra $C$ depends only on the homotopy types
$[\p]$ and $[g]$ of arbitrary representatives $\p \in \HOM_{\ccdgh}\big(\O, \O^\pr\big)$ and
$g \in   \HOM_\ccdgc(C,\O)$. 
Therefore the natural transformation 
$\sN_{[\p]}: \bm{\mP}_{\!\O}\Longrightarrow \bm{\mP}_{\!\O^\pr}$ is well-defined such that
$\sN_{[{\psi}]}\in \mathsf{Nat}\big(\bm{\mP}_{\!\O}, \bm{\mP}_{\!\O^\pr}\big)$ 
whenever $[{\psi}]\in \HOM_{\hccdgh}\big(\O, \O^\pr\big)$.
Combined with the Yoneda lemma, we have 
$\mathsf{Nat}\big(\bm{\mP}_{\!\O}, \bm{\mP}_{\!\O^\pr}\big)\cong \HOM_{\hccdgh}\big(\O, \O^\pr\big)$,
i.e., the category of representable presheaves of groups on $\mathit{ho}\category{ccdgC}(\Bbbk)$  over $\Bbbk$ is 
equivalent to the homotopy category $\mathit{ho}\category{ccdgH}(\Bbbk)$ 
of ccdg-Hopf algebras over $\Bbbk$.
It is obvious that 
$\sN^C_{[{\psi}]}: \bm{\mP}_{\!\O}(C)\longrightarrow\bm{\mP}_{\!\O^\pr}(C)$  is an isomorphism of groups 
for every ccdg-coalgebra $C$ 
whenever $\p:\O\rightarrow \O^\pr$ is a homotopy equivalence of ccdg-Hopf algebras. 
Therefore
$\sN_{[{\psi}]}$  is a natural isomorphism whenever $\p$ is  a homotopy equivalence of ccdg-Hopf algebras.
\qed
\end{proof}

\subsubsection{$\bm{\mP}(\Bbbk^\vee)$ action on the presheaf
$\grave{\bm{\mP}}:\mathring{\mathit{ho}\category{ccdgC}}(\fieldk)\rightsquigarrow \category{Set}$.}

Remind that $\fieldk$ has a structure $\Bbbk^\vee=\big(\Bbbk, {\ep}_\Bbbk, \cp_\Bbbk, 0\big)$ of ccdg-coalgebra,
which 
is a terminal object in the category  ${\mathit{ho}\category{ccdgC}}(\fieldk)$.
Therefore we have the group $\bm{\mP}_{\!\!\O}(\Bbbk^\vee)$, which plays a special role.
The counit ${\ep}_C:C \rightarrow \Bbbk$ of every ccdg-coalgebra $C$ is  a morphism
of ccdg-coalgebra  to $\Bbbk^\vee$, so that there is a canonical  homomorphism 
$\bm{\mP}_{\O}([{\ep}_C]): \bm{\mP}_{\O}(\Bbbk^\vee) \rightarrow \bm{\mP}_{\O}(C)$ of groups.
Consider  the 
representable presheaf $\grave{\bm{\mP}}_\O: \mathring{\mathit{ho}\category{ccdgC}}(\fieldk)\rightsquigarrow \category{Set}$ underlying  $\bm{\mP}_{\!\O}$. 

\begin{lemma}
$\grave{\bm{\mP}}_\O: \mathring{\mathit{ho}\category{ccdgC}}(\fieldk)\rightsquigarrow \category{Set}$ 
is a representable presheaf of 
$\bm{\mP}_{\O}(\Bbbk^\vee)$-sets  on the homotopy category of ccdg-coalgebras.
\end{lemma}
\begin{proof}
It is trivial to  check that $\grave{\bm{\mP}}_\O(C)$ is 
a $\bm{\mP}_{\O}(\Bbbk^\vee)$-set for every ccdg-coalgebra $C$ with the  action
$\xymatrix{\r_C:\bm{\mP}_{\O}(\Bbbk^\vee)\times\grave{\bm{\mP}}_\O(C)\ar[r] & \grave{\bm{\mP}}_\O(C)}$ 
defined by
$$
\r_C\big([g], [\a] \big):= [g\circ \ep_C]\ast_{C\!,\O} [\a]
=\big[m_\O\circ  \big((g\circ \ep_C)\otimes\a\big)\circ \cp_C\big],
$$
where $g \in \HOM_\ccdgc (\Bbbk^\vee, \O)$ and $\a \in \HOM_\ccdgc (C,\O)$ 
are arbitrary representative of $[g]$ and
$[\a]$ respectively, 
such that for every $[f] \in \HOM_\hccdgc (C, C^\pr)$  the following diagram commutes
$$
\xymatrixrowsep{1.5pc}
\xymatrixcolsep{3pc}
\xymatrix{
\ar[d]^-{\I\times \grave{\bm{\mP}}_\O([f])}
\bm{\mP}_{\O}(\Bbbk^\vee)\times\grave{\bm{\mP}}_\O\big(C^\pr\big)\ar[r]^-{\r_{\!C^\pr}}  &\grave{\bm{\mP}}_\O\big(C^\pr\big)\ar[d]^-{\grave{\bm{\mP}}_{\O}([f])}
\cr
\bm{\mP}_{\O}(\Bbbk^\vee)\times \grave{\bm{\mP}}_\O\big(C\big)\ar[r]^-{\r_{\!C}}  &\grave{\bm{\mP}}_\O\big(C\big)
,
}
\quad\hbox{i.e.,}\quad
\r_{\!C}\circ (\I\times \grave{\bm{\mP}}_\O([f]))= \grave{\bm{\mP}}_\O([f])\circ \r_{\!C^\pr}
.
$$
\qed
\end{proof}

\subsection{Presheaf of Lie algebras 
$\bm{T\mP}_{\O}: \mathring{\mathit{ho}\category{ccdgC}}(\Bbbk) \rightsquigarrow \category{Lie}(\Bbbk)$}

For a ccdg-Hopf algebra $\O$ we construct a presheaf 
$\bm{T\mP}_{\O}: \mathring{\mathit{ho}\category{ccdgC}}(\Bbbk) \rightsquigarrow \category{Lie}(\Bbbk)$ 
of Lie algebras on  the homotopy category of ccdg-coalgebras such that 
we have a natural isomorphism $\bm{T\mP}_{\O}\cong\bm{T\mP}_{\O^\pr}$ whenever  we have
a natural isomorphism 
$\bm{\mP}_{\O}\cong\bm{\mP}_{\O^\pr}$ or, equivalently, we have a homotopy equivalence $\O \cong \O^\pr$
of ccdg-Hopf algebras. 
The converse is not true in general.

Denoted by
$\THOM_{\ccdgc}(C, \O)$, for each ccdg-coalgebra $C$,  is the set of all \emph{tangential morphisms} of ccdg-coalgebras about 
the identity
$e_{C\!,\O}:=u_\O\circ \ep_C$: 
\eqn{tanccdgc}{
\THOM_{\ccdgc}(C, \O):=
\left\{{\ups} \in \Hom\big(C, \O\big)_0\left|
\begin{aligned}
& \rd_{C\!,\Omega}{\ups} =0
,\cr
&\ep_\O\circ {\ups} =0
,\cr
& \cp_\O\circ {\ups} = \big(e_{C,\!\O}\otimes {\ups} +{\ups}\otimes e_{C,\!\O}\big)\circ \cp_C
\end{aligned}
 \right.\!\!\!\!\!\!\!\!
\right\}.
}
A homotopy pair $\big({\ups}(t),\s(t)\big)$ on $\THOM_{\ccdgc}(C, \O)$  is 
a pair  of ${\ups}(t) \in \Hom\big(C, \O\big)_0[t]$
and $\l(t) \in \Hom\big(C, \O\big)_1[t]$ satisfying the homotopy flow equation 
$\Fr{d}{dt}{\ups}(t)= \rd_{C\!,\O}\s(t)$ generated by $\s(t)$ subject to the following
conditions
\eqn{tancchopa}{
\ups(0)\in \THOM_{\ccdgc}(C, \O)
,\qquad
\left\{
\begin{aligned}
\ep_\O\circ \l(t)&=0,\cr
\cp_\O\circ \l(t) &= \big(e_{C\!,\O}\otimes \s(t) +\s(t)\otimes e_{C\!,\O}\big)\circ \cp_C.
\end{aligned}\right.
}
It is straightforward to check that ${\ups}(t)$ is a family of tangential morphisms of ccdg-coalgebras about the identity.
We say ${\ups},\tilde {\ups} \in \THOM_{\ccdgc}(C, \O)$  are homotopic, ${\ups}\sim \tilde{\ups}$, 
or having the same homotopy type, $[{\ups}]=[\tilde{\ups}]$, if there is a homotopy flow connecting them.  
Denoted by
$\THOM_{\hccdgc}(C, \O)$  is the set of homotopy types of all tangential morphisms of ccdg-coalgebras about the identity.

\begin{theorem}[Definition]\label{Liereppresheaf}
For each ccdg-Hopf algebra $\O$
we have  a presheaf   of Lie algebras
$\bm{T\mP}_{\O}:\mathring{\mathit{ho}\category{ccdgC}}(\fieldk)\rightsquigarrow \category{Lie}(\Bbbk)$
on the homotopy category of ccdg-coalgebras,
sending

\begin{itemize}
\item
each  ccdg-coalgebra $C$ to the Lie algebra 
$$
\bm{T\mP}_{\O}(C):=\Big(\THOM_{\hccdgc}(C, \O), [-,-]_{\ast_{C\!,\O}} \Big)
$$
with the Lie bracket defined by, $\forall [\ups_1],[\ups_2] \in \THOM_{\hccdgc}(C, \O)$,
$$
\big[[{\ups}_1],[{\ups}_2]\big]_{\ast_{C\!,\O}} 
:=  \big[m_{\O}\circ ({\ups}_1\otimes {\ups}_2-{\ups}_2\otimes {\ups}_1\big)\circ \cp_C\big]=[\ups_1]\ast_{C\!,\O}[\ups_2]- 
[\ups_2]\ast_{C\!,\O}[\ups_1]
,
$$
where ${\ups}_1,{\ups}_2 \in \THOM_{\ccdgc}\big(C,\Omega\big)$ are arbitrary representatives
of the homotopy types $[{\ups}_1],[{\ups}_2]$, respectively.

\item 
each  morphism $[f] \in \HOM_{\hccdgc}\big(C, C^\pr\big)$ to
the morphism
$\bm{T\mP}_{\O}\big([f]\big):\bm{T\mP}_{\O}\big(C^\pr\big)\rightarrow \bm{T\mP}_{\O}\big(C\big)$ 
of Lie algebras
defined by, $\forall [{\ups}^\pr] \in   \THOM_{\hccdgc}\big(C^\pr,\Omega\big)$,
$$
\bm{T\mP}_{\O}([f])\big([{\ups}^\pr]\big):=\big[{\ups}^\pr\circ f\big]
,$$
where  $f$
and ${\ups}^\pr$
are arbitrary representatives of the homotopy types $[f]$ and $[{\ups}^\pr]$, respectively,
\end{itemize}

such that
$\bm{T\mP}_{\O}([f])$ is an isomorphism of Lie algebras
whenever $f:C\rightarrow C^\pr$ is a homotopy equivalence of ccdg-coalgebras.

For every morphism  $[\p] \in \HOM_{\hccdgh}\big(\O,\O^\pr\big)$ 
we have a natural transformation 
${T\!\!\sN}_{[\p]}: \bm{T\!\mP}_{\!\Omega}\Longrightarrow \bm{T\!\mP}_{\!\Omega^\pr}:
\mathring{\mathit{ho}\category{ccdgC}}(\fieldk)\rightsquigarrow \category{Lie}(\Bbbk)$,
whose component 
${T\!\!\sN}_{[\p]}^{C}: \bm{T\!\mP}_{\!\Omega}(C)\rightarrow  \bm{T\!\mP}_{\!\Omega^\pr}(C)$
at  each ccdg-coalgebra $C$  is defined by,  $\forall [\ups] \in \THOM_{\hccdgc}\big(C, \O\big)$,
$$
{T\!\!\sN}_{[\p]}^{C}\big([\ups]\big):= \big[\psi\circ \ups\big]
,
$$
where 
 $\p \in \HOM_{\ccdgh}\big(\Omega, \Omega^\pr\big)$ 
and $\ups \in \THOM_{\ccdgc}\big(C, \O\big)$ 
are  arbitrary representatives of the homotopy types $[\p]$
and $[\ups]$, respectively,
such that
${T\!\!\sN}_{[\p]}: \bm{T\mP}_{\!\Omega}\Longrightarrow \bm{T\!\mP}_{\!\Omega^\pr}:
\mathring{\mathit{ho}\category{ccdgC}}(\fieldk)\rightsquigarrow \category{Lie}(\Bbbk)$ 
is a natural isomorphism whenever $\p:\O\rightarrow \O^\pr$ is a homotopy equivalence of ccdg-Hopf algebras.

\end{theorem}

For every ccdg-Hopf algebra $\O$ we define a  presheaf of Lie algebras 
$\bm{T\!\CP}_{\O}:\mathring{\category{ccdgC}}(\fieldk)\rightsquigarrow \category{Lie}(\Bbbk)$
on $\mathring{\category{ccdgC}}(\fieldk)$, which will induces 
$\bm{T\mP}_{\O}:\mathring{\mathit{ho}\category{ccdgC}}(\fieldk)\rightsquigarrow \category{Lie}(\Bbbk)$.

\begin{lemma}\label{cLieone}
For each ccdg-Hopf algebra $\O$
we have  a presheaf   of Lie algebras
$\bm{T\!\CP}_{\O}:\mathring{\category{ccdgC}}(\fieldk)\rightsquigarrow \category{Lie}(\Bbbk)$,
sending
each  ccdg-coalgebra $C$ to the Lie algebra 
$$
\bm{T\!\CP}_{\O}(C):=\Big(\THOM_{\ccdgc}(C, \O), [-,-]_{\star_{C\!,\O}} \Big)
,
$$
where
$[\ups_1, \ups_2]_{\star_{C\!,\O}}: =  \ups_1{\star_{C\!,\O}}\ups_2 -\ups_2{\star_{C\!,\O}}\ups_1$ 
for all $\ups_1, \ups_2 \in \THOM_{\ccdgc}(C, \O)$,
and each morphism $f:C\rightarrow C^\pr$ of ccdg-coalgebras to
a Lie algebra homomorphism $\bm{T\!\CP}_{\O}(f):\bm{T\!\CP}_{\O}(C^\pr)\rightarrow \bm{T\!\CP}_{\O}(C)$
defined by $\bm{T\!\CP}_{\O}(f)(\ups^\pr) =\ups^\pr\circ f$, $\forall \ups^\pr \in \THOM_{\ccdgc}(C^\pr, \O)$.
\end{lemma}

\begin{proof}
1.  We show that $\bm{T\!\CP}_{\O}(C)$ is a Lie algebra as follows.  
It is suffice to check that  $[\ups_1, \ups_2]_{\star_{C\!,\O}}\in \THOM_{\ccdgc}(C, \O)$ 
for  every  $\ups_1, \ups_2 \in \THOM_{\ccdgc}(C, \O)$
since the convolution product ${\star_{C\!,\O}}$ is associative.  
We have
\eqalign{
\rd_{C\!,\O}[{\ups}_1,{\ups}_2]_\star=&[\rd_{C\!,\O}{\ups}_1,{\ups}_2]_\star +[{\ups}_1,\rd_{C\!,\O}{\ups}_2]_\star=0
,\cr
\ep_\O\circ [{\ups}_1,{\ups}_2]_\star =
&
\ep_\O\circ m_\O\circ \big([{\ups}_1, {\ups}_2]_\otimes\big)\circ \cp_C
= m_\O\circ (\ep_\O\otimes \ep_\O)\circ \big([{\ups}_1, {\ups}_2]_\otimes\big)\circ \cp_C
=0
,\cr
\cp_\O\circ [{\ups}_1,{\ups}_2]_\star =
&
\cp_\O\circ m_\O\circ \big([{\ups}_1, {\ups}_2]_\otimes\big)\circ \cp_C
\cr
=
&m_{\O\otimes\O}\circ\big(\cp_\O\otimes\cp_\O\big)({\ups}_1\otimes {\ups}_2 -
{\ups}_2\otimes {\ups}_1\big)\circ \cp_C
\cr
=
&m_{\O\otimes\O}\circ\big[{\ups}_1\otimes e+e\otimes {\ups}_1,{\ups}_2\otimes e+e\otimes{\ups}_2\big]_\otimes 
\circ\big(\cp_C\otimes\cp_C\big)\circ\cp_C
\cr
=
&\big(m_\O\otimes m_\O\big)
\circ\big([{\ups}_1,{\ups}_2]_\otimes \otimes e\otimes e +e\otimes e\otimes [{\ups}_1,{\ups}_2]_\otimes \big)
\circ \big(\cp_C\otimes\cp_C\big)\circ\cp_C
\cr
=
&\big([{\ups}_1,{\ups}_2]_\star\otimes e\star e + e\star e\otimes[{\ups}_1,{\ups}_2]_\star\big)\circ \cp_C
\cr
=&\big( [{\ups}_1,{\ups}_2]_\star\otimes e+e\otimes [{\ups}_1,{\ups}_2]_\star\big)\circ \cp_C
,
}
where we have used the short-hand notations ${\star_{C\!,\O}}=\star$, $u_\O\circ \ep_C=e$
and $[\ups_1,\ups_2]_\otimes = \ups_1\otimes \ups_2 -\ups_2\otimes \ups_1$.

2.  We show that $\bm{T\!\CP}_{\O}(f):\bm{T\!\CP}_{\O}(C^\pr)\rightarrow \bm{T\!\CP}_{\O}(C)$ 
is a Lie algebra homomorphism as
follows.   To begin with we need to check 
that $\bm{T\!\CP}_{\O}(f)(\ups^\pr) =\ups^\pr\circ f \in \THOM_{\ccdgc}(C,\O)$ 
whenever $\ups^\pr\in \THOM_{\ccdgc}(C^\pr,\O)$:
\eqalign{
\rd_{C\!,\O}(\ups^\pr\circ f)
= 
& \rd_\O\circ \ups^\pr\circ f - \ups^\pr\circ f\circ \rd_C 
=\big(\rd_\O\circ \ups^\pr- \ups^\pr\circ \rd_{C^\pr}\big)\circ f
=0
,\cr
\ep_\O\circ (\ups^\pr\circ f)= & (\ep_\O\circ \ups^\pr)\circ f=0
,\cr
\cp_\O\circ({\ups}^\pr\circ f) = 
& 
\big((u_\O\circ \ep_{C^\pr})\otimes {\ups}^\pr +{\ups}^\pr\otimes (u_\O\circ \ep_{C^\pr})\big)\circ \cp_{C^\pr}\circ f
\cr
=
& 
\big((u_\O\circ \ep_{C^\pr})\otimes {\ups}^\pr +{\ups}^\pr\otimes (u_\O\circ \ep_{C^\pr})\big)\circ (f\otimes f)\circ \cp_{C}
\cr
=
& 
\big((u_\O\circ \ep_{C})\otimes {\ups}^\pr\circ f +{\ups}^\pr\circ f\otimes (u_\O\circ \ep_{C})\big)\circ \cp_{C}
.
}
Then we check that $\bm{T\!\CP}_{\O}(f)$ is a homomorphism of Lie algebras:
\eqalign{
\bm{T\!\CP}_{\O}(f)\big([\ups_1^\pr,\ups_2^\pr]_{\star_{C^\pr\!,\O}}\big)
=
&m_\O\circ  [\ups^\pr_1, \ups^\pr_2]_\otimes \circ \cp_{C^\pr}\circ f
=
m_\O\circ  [\ups^\pr_1, \ups^\pr_2]_\otimes \circ (f\otimes f)\circ \cp_{C}
\cr
=
&m_\O\circ  [\ups^\pr_1\circ f , \ups^\pr_2\circ f]_\otimes \circ \cp_{C}
=\left[\bm{T\!\CP}_{\O}(f)(\ups_1^\pr), \bm{T\!\CP}_{\O}(f)(\ups_2^\pr)\right]_{\star_{C\!,\O}}
.
}

3. The functoriality of $\bm{T\!\CP}_{\O}$ is obvious.
\qed
\end{proof}

\begin{lemma}\label{cLietwo}
For every morphism  $\p:\O\rightarrow \O^\pr$ of ccdg-Hopf algebras 
we have a natural transformation 
${T\!\!\sN}_{\!\!\p}: \bm{T\!\CP}_{\!\O}\Longrightarrow 
\bm{T\!\CP}_{\!\O^\pr}: \mathring{\category{ccdgC}}(\fieldk)\rightsquigarrow \category{Lie}(\Bbbk)$,
whose component 
${T\!\!\sN}_{\!\!\p}^{C}: \bm{T\!\CP}_{\!\O}(C)\rightarrow  \bm{T\!\CP}_{\!\O^\pr}(C)$
at  each ccdg-coalgebra $C$ is defined by 
${T\!\!\sN}_{\!\!\p}^{C}(\ups):= \psi\circ \ups$,  $\forall \ups \in \THOM_{\ccdgc}\big(C, \O\big)$.
\end{lemma}

\begin{proof}
1.
We show that 
${T\!\!\sN}_{\p}^{C}: \bm{T\!\CP}_{\!\O}(C)\rightarrow  \bm{T\!\CP}_{\!\O^\pr}(C)$ is a Lie algebra homomorphism
for every ccdg-coalgebra $C$. 
We check that ${T\!\!\sN}_{\p}^{C}(\ups):= \psi\circ \ups \in \THOM_{\ccdgc}\big(C^\pr, \O\big)$ for all
$\ups \in \THOM_{\ccdgc}\big(C, \O\big)$:
\eqalign{
\rd_{C\!,\O^\pr}( \psi\circ \ups)= & \rd_{\O^\pr}\circ \psi\circ \ups -  \psi\circ \ups\circ \rd_{C} 
=\psi\circ \big(\rd_\O\circ \ups- \ups\circ \rd_{C}\big)=0
,\cr
\ep_{\O^\pr}\circ ( \psi\circ \ups)= & \ep_\O\circ \ups=0
,\cr
\cp_{\O^\pr}\circ( \psi\circ \ups) = 
&
(\psi\otimes\psi)\circ \cp_{\O}\circ \ups
=
(\psi\otimes\psi)\circ 
\big((u_\O\circ \ep_{C})\otimes {\ups} +{\ups}\otimes (u_\O\circ \ep_{C})\big)\circ\cp_C
\cr
=
& 
\big((u_{\O^\pr}\circ \ep_{C})\otimes \p\circ {\ups} + \p\circ{\ups}\otimes (u_{\O^\pr}\circ \ep_{C})\big)\circ\cp_C
.
}
Now we check that ${T\!\!\sN}_{\!\!\p}^{C}:\bm{T\!\CP}_{\!\O}(C)\rightarrow  \bm{T\!\CP}_{\!\O^\pr}(C)$ 
is a Lie algebra homomorphism:
for all $\ups_1,\ups_2 \in \THOM_{\ccdgc}\big(C, \O\big)$ we have
\eqalign{
{T\!\!\sN}_{\!\!\p}^{C}\big([\ups_1, \ups_2]_{\star_{C\!,\O}}\big)=
& \p\circ m_\O\circ [\ups_1,\ups_2]_\otimes \circ \cp_C
= m_{\O^\pr}\circ [\p\circ\ups_1,\p\circ\ups_2]_\otimes \circ \cp_C
\cr
=
&
\big[{T\!\!\sN}_{\!\!\p}^{C}(\ups_1),{T\!\!\sN}_{\p}^{C}(\ups_2)\big]_{\star_{C\!,\O^\pr}}
.
}

2.  We show that 
${T\!\!\sN}_{\!\!\p}: \bm{T\!\CP}_{\!\O}\Longrightarrow \bm{T\!\CP}_{\!\O^\pr}
: \mathring{\category{ccdgC}}(\fieldk)\rightsquigarrow \category{Lie}(\Bbbk)$ 
is a natural transformation:
For every morphism $f:C\rightarrow C^\pr$ of ccdg-coalgebras we have
${T\!\!\sN}_{\!\!\psi}^C\circ \bm{T\!\CP}_{\O}(f) = \bm{T\!\CP}_{\O^\pr}(f)\circ {T\!\!\sN}_{\!\!\psi}^{C^\pr}$
since for all $\ups^\pr \in \THOM_{\ccdgc}\big(C^\pr, \O\big)$ we obtain that
$$
{T\!\!\sN}_{\psi}^C\circ\bm{T\!\CP}_{\O}(f)(\ups^\pr)= {\psi}\circ (\ups^\pr\!\circ f)=({\psi}\circ \ups^\pr)\circ f=
\bm{T\!\CP}_{\O^\pr}(f)\circ {T\!\!\sN}_{\psi}^{C^\pr}(\ups^\pr)
.
$$
\qed
\end{proof}

\begin{lemma}\label{cLiethree}
Assume that
\begin{itemize}

\item $\big({\ups}(t),\s(t)\big)$, $\big({\ups}_1(t),\s_1(t)\big)$ and  $\big({\ups}_2(t),\s_2(t)\big)$
are homotopy pairs  on $\THOM_{\ccdgc}(C,\O)$;

\item $\big(f(t),\l(t)\big)$ is a homotopy pair on $\HOM_{\ccdgc}(C,C^\pr)$
and $\big({\ups}^\pr(t),\s^\pr(t)\big)$  is a homotopy pair on $\THOM_{\ccdgc}(C^\pr,\O)$;

\item
$\big(\p(t),\xi(t)\big)$ is a homotopy pair on $\HOM_{\ccdgh}(\O,\O^\pr)$.
\end{itemize}

Then we have the following homotopy pairs
\begin{enumerate}[label=({\alph*})]

\item
$\Big(\left[\big({\ups}_1(t), {\ups}_2(t)\right]_{\star_{C\!,\O}},  \left[\s_1(t), {\ups}_2(t)\right]_{\star_{C\!,\O}} \
+ \left[{\ups}_1(t),\s_2(t)\right]_{\star_{C\!,\O}} \Big)$
on $\THOM_{\ccdgc}(C,\O)$.

\item
$\Big(\ups^\pr(t)\circ{f}(t), \ups^\pr(t)\circ \l(t) +\s^\pr(t)\circ {f}(t)\Big)$ 
on $\THOM_{\ccdgc}(C,\O)$.

\item
$\Big({\p}(t)\circ \ups(t), \p(t)\circ \s(t)+ {\xi}(t)\circ \ups(t)\Big)$ 
 on $\HOM_{\ccdgc}(C,\O^\pr)$.
\end{enumerate}
\end{lemma}

\begin{proof}
These can be checked by routine computations, which are omitted for the sake of space.
\qed
\end{proof}

Now we are ready for the proof of  Theorem \ref{Liereppresheaf}

\begin{proof}[Theorem \ref{Liereppresheaf}]
Based on Lemmas \ref{cLieone}, \ref{cLietwo} and \ref{cLiethree},
it is trivial to
check that the Lie algebra $\bm{T\mP}_{\O}(C)$ is well-defined
and $\bm{T\!\mP}_{\O}([f]): \bm{T\!\mP}_{\O}(C^\pr)\rightarrow \bm{T\!\mP}_{\O}(C)$ 
is a well-defined Lie algebra homomorphism,
depending only on the homotopy type $[f]$ of $f$. 
Therefore $\bm{T\!\mP}_{\O}$ is a presheaf of Lie algebras
on the homotopy category $\mathit{ho}\category{ccdgC}(\Bbbk)$ of ccdg-coalgebras
as claimed,  where the functoriality of $\bm{T\!\mP}_{\O}$ is obvious.
It is also obvious that $\bm{T\!\mP}_{\O}([f])$ is an isomorphism of Lie algebras whenever $f:C\rightarrow C^\pr$ is
a homotopy equivalence of ccdg-coalgebras. 
It also trivial to check that the natural transformation 
${T\!\!\sN}_{\!\![\p]}: \bm{T\!\mP}_{\!\Omega}\Longrightarrow \bm{T\!\mP}_{\!\Omega^\pr}:
\mathring{\mathit{ho}\category{ccdgC}}(\fieldk)\rightsquigarrow \category{Lie}(\Bbbk)$ is well-defined such that
${T\!\!\sN}_{\!\![{\psi}]}\in \mathsf{Nat}\big(\bm{T\!\mP}_{\O}, \bm{T\!\mP}_{\O^\pr}\big)$ whenever 
$[{\psi}]\in \HOM_{\hccdgh}\big(\O, \O^\pr\big)$. Finally it is obvious that
${T\!\!\sN}_{\!\![{\psi}]}$  is a natural isomorphism whenever $\p$ is  a homotopy equivalence of ccdg-Hopf algebra.
\qed
\end{proof}

Remind that
the category of representable presheaves of
groups on  $\mathit{ho}\category{ccdgC}(\Bbbk)$ is equivalent 
to the homotopy category $\mathit{ho}\category{ccdgH}(\Bbbk)$
of ccdg-Hopf algebras---Theorem \ref{reppresheaf}.  Therefore,
the assignments
$\bm{\mP}_{\O}\mapsto \bm{T\!\mP}_{\O}$
and $\sN_{\![{\psi}]}\mapsto  {T\!\!\sN}_{\!\![{\psi}]}$
define a functor $\category{T}$ from the category of representable presheaves of
groups on  $\mathit{ho}\category{ccdgC}(\Bbbk)$  to the category of presheaves of
Lie algebras on  $\mathit{ho}\category{ccdgC}(\Bbbk)$ 
such that ${T\!\!\sN}_{\!\![{\psi}]}$ is a natural isomorphism
whenever $\sN_{\![{\psi}]}$ is a natural isomorphism.

The following is obvious by definitions.
\begin{corollary}
If $\O$ is concentrated in degree zero, 
the group $\bm{\mP}_{\O}(\Bbbk^\vee)$ is the group of group-like elements in $\O$,
and  the Lie algebra $\bm{T\!\mP}_{\O}(\Bbbk^\vee)$ is the Lie algebra of primitive elements in $\O$.
\end{corollary}

\subsection{Complete ccdg-Hopf algebras and an isomorphism
$\grave{\mb{T}\!\bm{\mP}}_{\O}\cong \grave{\bm{\mP}}_{\O}$ of presheaves}

Quillen has introduced the notion of complete cocommutative Hopf algebra 
to  provide the Hopf algebra framework for the Malcev completion  and groups 
defined by the Baker-Campbell-Hausdorff formula
as well as for the rationalized homotopy groups of a pointed space \cite{Quillen}.  
For example,  the $\Bbbk$-rational group ring $\Bbbk\G$ of an abstract group $\G$ 
can completed by  the powers of its augmentation ideal $J$ to a complete  
Hopf algebra $\widehat{\fieldk \G}=\varprojlim\; \Bbbk\G /J^{n+1}$.
Then the set of primitive elements  form a Lie algebra 
$\mathbb{L}\big(\widehat{\fieldk \G}\big)$, which has a bijective correspondence 
$\xymatrix{\mathbb{L}\big(\widehat{\Bbbk\G}\big)\ar@/^0.2pc/[r]^-{\exp} 
&\ar@/^0.2pc/[l]^{\ln} \mathbb{G}\big(\widehat{\Bbbk \G}\big)}$
with the group $\mathbb{G}\big(\widehat{\Bbbk \G}\big)$ of group-like elements.\footnote{
The Malcev completion of $\G$ is isomorphic to $\mathbb{G}\big(\widehat{\mathbb{Q} \G}\big)$.
If the rational Abelianization $\G^{\mathit{ab}}\otimes_{\mathbb{Z}}\mathbb{Q}$  of $\G$ 
is finite dimensional one can attach
a pro-unipotent affine group scheme $\G^{\mathit{uni}}:\category{cA}(\Bbbk)\rightsquigarrow \category{Grp}$,
where $\category{cA}(\Bbbk)$ denotes the category of commutative algebras,
which is pro-represented by
the  linearly compact dual commutative Hopf algebra to  $\widehat{\fieldk \G}$. 
For $\G =\pi_1(X_*)$, where $X_*$ is
a $0$-connected pointed finite type space, $\G^{\mathit{un}}=\pi^{\mathit{un}}_1(X_*)$ is called the pro-unipotent
fundamental group (scheme) of the space, 
which was a starting point of the rational homotopy theory according to Sullivan---his paper
\cite{Sullivan} start with an algebraic model of unipotent local system (private communication).
}
Quillen's construction of complete Hopf algebras can be easily generalized to dg setting.

Consider a ccdg-Hopf algebra  $\O=\big(\O, u_\O, m_\O, \ep_\O, \cp_\O, \vs_\O, \rd_\O\big)$.

We introduce some notations.
Denote by  $m^{(n)}_\O:\O^{\otimes n}\rightarrow \O$, $n\geq 1$, 
is the $n$-fold iterated product generated by $m_\O:\O\otimes \O \rightarrow \O$
such that $m^{(1)}_\O=\I_\O$ and 
$m^{(n+1)}_\O = m_\O\circ (m^{(n)}_\O\otimes \I_\O)\equiv  m^{(n)}_\O\circ (\I^{n-1}_\O\otimes m_\O\big)$.
Denote by $m^{(n)}_{\O\otimes \O}:(\O\otimes\O)^{\otimes n}\rightarrow \O\otimes \O$, $n\geq 1$, 
is the $n$-fold iterated coproduct generated by 
$m_{\O\otimes \O}=(m_\O\otimes m_\O)\circ (\I_\O\otimes \t \otimes \I_\O)
:(\O\otimes \O)\otimes (\O\otimes \O) \rightarrow \O\otimes \O$.   Then we have
\eqn{hsqa}{
\cp_\O\circ m_\O^{(n)} = m^{(n)}_{\O\otimes \O}\circ (\cp_\O\otimes \ldots\otimes \cp_\O).
}

The counit $\ep_\O:\O\rightarrow \Bbbk$ is a canonical augmentation of $\O$,
since we have ${\ep}_\O\circ u_\O=\I_\fieldk$,  ${\ep}_\O\circ m_\O =m_\Bbbk\circ \big({\ep}_\O\otimes{\ep}_\O\big)$
and ${\ep}_\O\circ\rd_\O =0$.
Let $\mI=\Ker {\ep}_\O$ be the augmentation ideal. Then, we have 
a splitting $\O = \fieldk\cdot u_\O(1) \oplus \mI$ and a decreasing filtration 
\eqn{jfilter}{
\O =\mI^0 \supset \mI^1 \supset \mI^2\supset  \mI^3\supset \cdots,
}
where $\mI^n:=m^{(n)}_\O(\mI\otimes\cdots\otimes \mI)$.
It is straightforward to check that every structure of ccdg-Hopf algebra $\O$ is compatible with the above filtration.
We can endow trivial filtration on $\Bbbk$ and both $u_\O:\Bbbk\rightarrow \O$ and  $\ep_\O:\O\rightarrow \Bbbk$
are filtration preserving map.  
It is obvious that $m_\O:\mI^i\otimes \mI^j\rightarrow \mI^{i+j}$.
From $\rd_\O\circ u_\O=\ep_\O\circ \rd_\O=0$, 
we have $\rd_\O:\mI \rightarrow \mI$ and $\rd_\O: \mI^n \subset \mI^n$ as $\rd_\O$ is a derivation of $m_\O$.
From $\ep_\O\circ \vs_\O=\ep_\O$ we have $\vs_\O:\mI \rightarrow \mI$
and $\vs_\O : \mI^n \rightarrow \mI^n$ since $\vs_\O$ is anti-homomorphism of $m_\O$. 
Finally, we have $\cp_\O:  \mI^n \rightarrow\bigoplus_{i+j=n}^n  \mI^i\otimes \mI^j$, which can be checked as follows:
we have $\cp_\O: \mI^0 \rightarrow \mI^0\otimes \mI^0$ by definition and we can deduce that
$\cp_\O: \mI^1 \rightarrow \mI^0\otimes \mI^1\oplus \mI^1\otimes \mI^0$ from the counit property,
which is combined with the identity \eq{hsqa} to show the rest.
Then, $\hat{\O}=\varprojlim \left( \O \big\slash \mI^{n+1}\right)$
is a complete augmented dg-algebra with the decreasing filtration 
$$
\hat{\O} =F^0(\hat{\O}) \supset F^1(\hat{\O}) \supset F^2(\hat{\O})\supset \cdots,
$$
where $F^n(\hat{\O}) =\Ker\big(\hat{\O}\rightarrow \O\slash \mI^{n+1}\big)$.
We also have the extended coproduct 
$$
\hat\cp_\O: \hat{\O}
\rightarrow\hat{\O}\hat{\otimes} \hat{\O} = 
\varprojlim\left(\hat{\O}{\otimes}\hat{\O}\Big\slash F^{n+1}\big(\hat{\O}{\otimes} \hat{\O}\big)\right),
$$
where $F^{n}\big(\hat{\O}{\otimes} \hat{\O}\big)=\bigoplus_{i+j=n}F^{i}\big(\hat{\O}\big){\otimes} F^j\big(\hat{\O}\big)$,
such that $\hat{\O}=\big(\hat{\O},\hat u_\O, \hat m_\O, \hat \ep_\O, \hat \cp_\O, \hat \vs_\O, \hat\rd_\O\big)$ 
is a  ccdg-Hopf algebra. 

Let $\O$ be a complete ccdg-Hopf algebra and 
$\bm{\mP}_{{\O}}:\mathring{\mathit{ho}\category{ccdgC}}(\fieldk)\rightsquigarrow \category{Grp}$
be the presheaf of groups pro-represented by $\O$. Let 
$\bm{T\!\mP}_{{\O}}:\mathring{\mathit{ho}\category{ccdgC}}(\fieldk)\rightsquigarrow \category{Lie}(\Bbbk)$
be the associated presheaf of Lie algebras.
Let 
\eqalign{
\grave{\bm{\mP}}_{\O}
&=\HOM_{\ccdgc}(-,\O):\mathring{\mathit{ho}\category{ccdgC}}(\Bbbk)\rightsquigarrow \category{Set}
,\cr
\grave{\category{T}\!\bm{\mP}}_\O
&=\THOM_{\ccdgc}(-,\O): \mathring{\mathit{ho}\category{ccdgC}}(\Bbbk)\rightsquigarrow \category{Set}
,
}
be the underlying presheaves.
The remaining part of this subsection is devoted to the proof of the following.

Consider a ccdg-coalgebra  $C=\big(C, \ep_C,\cp_C, \rd_C\big)$.
Denote by  $\cp^{(n)}_C:C\rightarrow C^{\otimes n}$, $n\geq 1$, is the $n$-fold 
iterated coproduct generated by the coproduct $\cp_C:C \rightarrow C\otimes C$,
where $\cp^{(1)}_C=\I_C$ and $\cp^{(n+1)}_C 
= (\cp^{(n)}_C\otimes \I_C)\circ \cp_C \equiv (\I^{n-1}_C\otimes \cp_C)\circ \cp^{(n)}_C$.

\begin{theorem}\label{dgquillen}
For every complete ccdg-Hopf algebra  ${\O}$ we have  natural isomorphism 
$(\bm{\exp}, \bm{\ln})$ of  the presheaves
$
\xymatrixcolsep{3pc}
\xymatrix{ \grave{\bm{T\mP}}_{{\O}}\ar@/^0.3pc/@{=>}[r]^{\bm{\exp}}  
& \ar@/^0.3pc/@{=>}[l]^{\bm{\ln}} \grave{\bm{\mP}}_{{\O}}}
:\mathring{\mathit{ho}\category{ccdgC}}(\Bbbk)\rightsquigarrow \category{Set}
$,
whose component
$(\bm{\exp}^C, \bm{\ln}^C)$
at each ccdg-coalgebra $C$ are defined as follows:
 $\forall [g]\in  \HOM_{\hccdgc}\big(C,{\O}\big)$ 
and $\forall [\ups]\in  \THOM_{\hccdgc}\big(C,{\O}\big)$,
\eqalign{
\bm{\exp}^{C}\big([\ups]\big) 
&:= [\ep_C\circ u_{{\O}}]
+\sum_{n=1}^\infty\Fr{1}{n!}
\left[m^{(n)}_{{\O}}\circ \big({\ups}\hat{\otimes} \ldots \hat{\otimes}{\ups}\big)\circ \cp^{(n)}_C\right]
,\cr
\bm{\ln}^{C}\big([g]\big) &:=- \sum_{n=1}^\infty\Fr{(-1)^n}{n}\left[m^{(n)}_{{\O}}\circ 
\big(\bar g \hat{\otimes} \ldots \hat{\otimes} \bar g \big)\circ \cp^{(n)}_C\right]
,
}
where $g\in  \HOM_{\ccdgc}\big(C,{\O}\big)$ and ${\ups}\in  \THOM_{\ccdgc}\big(C,{\O}\big)$ 
are arbitrary representatives of the homotopy types
$[g]$ and $[{\ups}]$, respectively; and $\bar g :=g-  u_\O\circ \ep_C$.
\end{theorem}

\begin{remark}
The above theorem implies that the presheaf of groups $\bm{\mP}_{\O}$ 
can be recovered from the presheaf of Lie algebras $\bm{T\!\mP}_{\O}$ 
using the Baker-Campbell-Hausdorff formula.
If  $\mI/\mI^2$ is finite dimensional we can dualize $\bm{\mP}_{\O}$
 to obtain a pro-unipotent affine group dg-scheme, which is a subject of the sequel to this paper.
\end{remark}

We divide the proof into pieces.  
\begin{lemma}\label{gusl} 
Let $C$ be a ccdg-coalgebra.  Then we have
\begin{enumerate}[label=({\alph*})]
\item
$\a_1\star_{C\!,\O}\cdots\star_{C\!,\O} \a_n 
= m^{(n)}_\O\circ(\a_1\hat{\otimes}\ldots\hat{\otimes} \a_n)\circ \cp^{(n)}_C$,
$\forall \a_1,\ldots,\a_n \in \Hom(C,\O)$ and $n\geq 1$;
\item 
$\cp^{(n)}_{C{\otimes} C}\circ \cp_C =\overbrace{(\cp_C{\otimes} \ldots {\otimes} \cp_C)}^n \circ \cp^{(n)}_{C}$
 for all $n\geq 1$.
\end{enumerate}

\end{lemma}

\begin{proof}
The property
(a)  is trivial for $n=1$, is the definition of the convolution product $\star_{C\!,\O}$ for $n=2$ 
and the rest can be checked easily using an induction.
The property
(b) is trivial for $n=1$, since $\cp^{(1)}_{C{\otimes} C}:=\I_{C\otimes C}$ and $\cp^{(1)}_{C}:=\I_{C}$.
For $n=2$, from 
$\cp^{(2)}_{C\otimes C} := \cp_{C\otimes C}=(\I_C\otimes \t \otimes \I_C)\circ (\cp_C\otimes \cp_C)$ 
and $\cp^{(2)}_C =\cp_C$
it becomes the identity
$(\I_C\otimes \t \otimes \I_C)\circ (\cp_C\otimes \cp_C)\circ \cp_C  =(\cp_C\otimes \cp_C)\circ \cp_C$,
which is valid due to the \emph{cocommutativity} of $\cp_C$.
The rest can be checked easily by an induction.
\qed
\end{proof}

\begin{lemma}\label{gusla} 
We have $(\b_1\star_{C\!,\O}\ldots\star_{C\!,\O} \b_n)(c) \in \mI^n$, $n\geq 1$,   for every $c \in C$ 
whenever $\b_i \in \Hom(C,\O)$ has  the property $\ep_\O\circ \b_i=0$ for all $i=1,\ldots, n$.
\end{lemma}

\begin{proof}
From   $\ep_\O\circ  m^{(n)}_\O = m^{(n)}_\O \circ (\ep_\O\hat{\otimes}\ldots \hat{\otimes} \ep_\O)$, we obtain that
$\ep_\O\circ\big(\b_1\star_{C\!,\O}\ldots\star_{C\!,\O} \b_n\big) 
= m^{(n)}_\O\circ \Big( \ep_\O\circ \b_1\hat{\otimes}\ldots\hat{\otimes} \ep_\O\circ\b_n\Big)\circ \cp^{(n)}_C=0$.
It follows that $(\b_1\star\ldots\star \b_n)(c) \in \mI^n$ for all $c \in C$ 
since  
$(\b_1\star\ldots\star \b_n)(c)=\sum_{(c)} m^{(n)}_B\Big(\b_1(c_{(1)})\hat{\otimes}\ldots\hat{\otimes} \b_n(c_{(n)})\Big)$.
\qed
\end{proof}

\begin{lemma}\label{quslx}
Let $\O$ be a complete ccdg-Hopf algebra.  For any ccdg-coalgebra $C$, we have 
an isomorphism 
$\xymatrix{\THOM_{\ccdgc}(C, \O)\ar@/^/[r]^-{\exp^C}&\ar@/^/[l]^-{\ln^C} \HOM_{\ccdgc}(C, \O)}$, where
\begin{itemize}

\item for all $ \ups \in \THOM_{\ccdgc}(C, \O)$, we have
$$
\exp^C(\ups)
:=u_\O\circ \ep_C 
+ \sum_{n=1}^\infty\Fr{1}{n!}m^{(n)}_{{\O}}\circ \big({\ups}\hat{\otimes} \ldots \hat{\otimes}{\ups}\big)\circ \cp^{(n)}_C
\in \HOM_{\ccdgc}(C, \O)
$$

\item for all $\forall g \in \HOM_{\ccdgc}(C, \O)$, we have
$$
\ln^C(g)
:=
- \sum_{n=1}^\infty\Fr{(-1)^n}{n}
m^{(n)}_{{\O}}\circ \big(\bar g \hat{\otimes} \ldots \hat{\otimes} \bar g \big)\circ \cp^{(n)}_C
\in \THOM_{\ccdgc}(C, \O),
$$

\end{itemize}
such that $\exp^C(\ups)\sim \exp^C(\tilde\ups) \in \HOM_{\ccdgc}(C, \O)$
whenever $\ups\sim \tilde\ups  \in \THOM_{\ccdgc}(C, \O)$, and  
$\ln^C(g)\sim \ln^C(\tilde g) \in \THOM_{\ccdgc}(C, \O)$
whenever $g\sim \tilde g  \in \HOM_{\ccdgc}(C, \O)$.

\end{lemma}

\begin{proof}
We use some shorthand notations. 
We set $e =u_\O\circ \ep_C$.
We also set $\star =\star_{C\!,\O}$, $\a^{\star 0}=e$ and
$\a^{\star n} = \overbrace{\a\star\ldots\star \a}^n$, $n\geq 1$, for all $\a \in \Hom(C,\O)_0$.
Then, by Lemma \ref{gusl}(a),  we have
\eqnalign{sexpln}{
\exp^C(\ups)=& e + \sum_{n=1}^\infty\Fr{1}{n!}\ups^{\star n},
\qquad
\ln^C(g)=- \sum_{n=1}^\infty\Fr{(-1)^n}{n} \bar g^{\star n}.
}
Remind that $e \in  \HOM_{\ccdgc}(C, \O)$, i.e.,
$\rd_{C\!,\O}e=0$,  $\ep_\O\circ e =\ep_C$ and
$\cp_\O\circ e = (e\hat{\otimes} e\big)\circ \cp_C$.

1. We have to justify that the infinite sums in the definition of $\exp^C$ and $\ln^C$ make sense. 
Note the $\ep_\O\circ \ups=0$ by definition. We also have $\ep_\O\circ\bar g=0$ 
since   $\ep_\O\circ g =\ep_\O\circ e=\ep_C$ and $\bar g=g-e$.
By Lemma \ref{gusl}, we have $\ups^{\star n}(c), \bar g^{\star n}(c) \in \mI^n$ for all $c\in C$ and $n\geq 1$.
Therefore both $\exp^C(\ups)$ and $\ln^C(g)$ are well-defined.

2.  We check that $\exp^C(\ups)\in  \HOM_{\ccdgc}(C, \O)$ for every $\ups \in \THOM_{\ccdgc}(C, \O)$:
\eqalign{
\rd_{C\!,\O}\exp^C (\ups)=0
,\quad
\ep_\O\circ \exp^C(\ups)= \ep_C
,\quad
\cp_\O\circ\exp^C(\ups) =\big(\exp^C(\ups)\hat{\otimes} \exp^C(\ups)\big)\circ \cp_C.
}
The $1$st relation is  trivial since $\rd_{C\!,\O}$ is a derivation of $\star$ and $\rd_{C\!,\O}e=\rd_{C\!,\O}\ups=0$.
The $2$nd relation is also  trivial since $\ep_\O\circ e = \ep_C$ 
and $\ep_\O\circ \ups^{\star n}= (\ep_\O\circ \ups)^{\star n}=0$ for all $n\geq 1$.
It remains to check the $3$rd relation, which is equivalent to the following relations, $\forall n\geq 0$,
\eqn{checkthis}{
\cp_\O\circ\ups^{\star n}= \sum_{k=0}^n \Fr{n!}{(n-k)!k!}  \big(\ups^{\star n-k}\hat{\otimes} \ups^{\star k}\big)\circ \cp_C.
}
For $n=0$ the above becomes $\cp_\O\circ e = (e\hat{\otimes} e\big)\circ \cp_C$, which is true.

For cases with $n \geq 1$, we adopt a new notation.
Recall that $\big(\O\hat{\otimes} \O, u_{\O\hat{\otimes} \O}, m_{\O\hat{\otimes} \O}\big)$ 
is a $\Z$-graded associative algebra
and  $\big(C{\otimes} C, \ep_{C{\otimes} C}, \cp_{C\otimes C}\big)$ 
is a $\Z$-graded coassociative coalgebra.
Therefore we have  a $\Z$-graded associative algebra 
$\big( \Hom(C\otimes C, \O\hat{\otimes} \O), e\hat{\otimes} e,  \hbox{\scriptsize\ding{75}} \big)$,
where $\chi_1{\cstar}\chi_2 := m_{\O\hat{\otimes} \O}\circ (\chi_1\hat{\otimes} \chi_2)\circ \cp_{C\otimes C}$,  
$\forall \chi_1,\chi_2\in \Hom(C\otimes C, \O\hat{\otimes} \O)$.
We also have, for all $n\geq 1$ and $\chi_1,\ldots,\chi_n\in \Hom(C\otimes C, \O\hat{\otimes} \O)$,
\eqn{quslc}{
\chi_1{\cstar}\ldots \hbox{\scriptsize\ding{75}}\chi_n 
=  m^{(n)}_{\O\hat{\otimes} \O}\circ (\chi_1\hat{\otimes}\ldots\hat{\otimes} \chi_n)\circ \cp^{(n)}_{C\otimes C}.
}
For example, consider $\a_1,\a_2,\b_1,\b_2 \in \Hom(C,\O)$ 
so that $\a_1\hat{\otimes} \a_2, \b_1\hat{\otimes} \b_2\in \Hom(C\otimes C, \O\hat{\otimes} \O)$.
Then we have 
$(\a_1\hat{\otimes} \b_1){\cstar}(\a_2\hat{\otimes} \b_2) =(-1)^{|\b_1||\a_2|} \a_1\star\a_2 \hat{\otimes} \b_1\star \b_2$.

It follows that 
$(\ups\hat{\otimes} e) {\cstar}(e\hat{\otimes} \ups) =(e\hat{\otimes} \ups){\cstar}(\ups\hat{\otimes} e)$ 
since the both terms are $\ups\hat{\otimes} \ups$.
We also have $(\ups\hat{\otimes} e)^{{\cstar}n} = \ups^{\star n}\hat{\otimes} e$ 
and $(e\hat{\otimes} \ups)^{{\cstar}n} = e\hat{\otimes} \ups^{\star n}$.
Combined with the binomial identity, we obtain that, $\forall n\geq 1$,
$$
(\ups\hat{\otimes} e + e\hat{\otimes} \ups)^{{\cstar}n} 
=\sum_{k=0}^n \Fr{n!}{(n-k)!k!} \big(\ups^{\star n-k}\hat{\otimes} \ups^{\star k}\big)
.
$$
Therefore the RHS of \eq{checkthis} becomes
\eqalign{
\mathit{RHS}=
&
(\ups\hat{\otimes} e + e\hat{\otimes} \ups)^{{\cstar}n}  \circ \cp_C
= 
m^{(n)}_{\O\hat{\otimes} \O}\circ \big(\ups\hat{\otimes} e + e\hat{\otimes} \ups\big)^{\hat{\otimes}  n}
\circ \cp^{(n)}_{C\otimes C}\circ \cp_C
\cr
=
&m^{(n)}_{\O\hat{\otimes} \O}\circ \big(\ups\hat{\otimes} e + e\hat{\otimes} \ups\big)^{\hat{\otimes}  n}
\circ (\cp_C\otimes \ldots \otimes \cp_C) \circ \cp^{(n)}_{C},
}
where we  use Lemma \ref{gusl}(b) for the last equality.
Consider the LHS of \eq{checkthis}:
\eqalign{
\mathit{LHS}=
&
\cp_\O\circ \ups^{\star n}
=
\cp_\O\circ m^{(n)}_{\O}\circ (\ups\hat{\otimes}\ldots\hat{\otimes} \ups)\circ \cp^{(n)}_{C}
=
m^{(n)}_{\O\hat{\otimes} \O}\circ (\cp_\O\circ\ups \hat{\otimes} \ldots \hat{\otimes} \cp_\O\circ \ups)\circ\cp^{(n)}_{C}
\cr
=
&
m^{(n)}_{\O\hat{\otimes} \O}\circ \big(\ups\hat{\otimes} e + e\hat{\otimes} \ups\big)^{\hat{\otimes} n}
\circ (\cp_\O\hat{\otimes} \ldots \hat{\otimes} \cp_\O)\circ \cp^{(n)}_{C}
,
}
where we use \eq{hsqa} for the $3$rd equality and
the property $\cp_\O\circ \ups =(\ups\hat{\otimes} e+e\hat{\otimes}\ups)\circ\cp_C$ for the last equality.
Therefore we conclude that $\cp_\O\circ\exp^C(\ups) =\big(\exp^C(\ups)\hat{\otimes} \exp^C(\ups)\big)\circ \cp_C$.

2.  We check that $\ln^C(g)\in   \THOM_{\ccdgc}(C, \O)$ for every  $g\in\HOM_{\ccdgc}(C, \O)$:
\eqalign{
\rd_{C\!,\O}\ln^C (g)=0
,\quad
\ep_\O\circ \ln^C(g)= 0
,\quad
\cp_\O\circ \ln^C (g)= \big(\ln^C(g)\hat{\otimes} e+e\hat{\otimes} \ln^C(g)\big)\circ \cp_C.
}
The $1$st relation is obvious
since $\rd_{C\!,\O}(\bar g)=\rd_{C\!,\O}g -\rd_{C\!,\O}e=0$ and $\rd_{C\!,\O}$ is a derivation of $\star$. 
The $2$nd relation is also obvious  
since $\ep_\O\circ \bar g^{\star n}=(\ep_\O\circ  \bar g)^{\star n}=0$ for all $n\geq 1$.
Therefore it remains to check the $3$rd relation.

Define 
$\ln_{\cstar}^C(\chi) =-\sum_{n=1}^\infty\Fr{(-1)^n}{n} (\chi -e\hat{\otimes} e)^{{\cstar}n}$ 
for all $\chi \in \Hom(C{\otimes} C, \O\hat{\otimes} \O)$ satisfying
$(\ep_\O\hat{\otimes} \ep_\O)\circ \chi =\ep_C\hat{\otimes} \ep_C$.  Then, we have
\eqalign{
\ln^C_{\cstar}(g\hat{\otimes} e) =-\sum_{n=1}^\infty\Fr{(-1)^n}{n} ( g\hat{\otimes} e -e\hat{\otimes} e)^{{\cstar}n}
=-\sum_{n=1}^\infty\Fr{(-1)^n}{n} (\bar g\hat{\otimes} e)^{{\cstar}n}=\ln^C (g)\hat{\otimes} e,
}
and, similarly,  $e\hat{\otimes} \ln^C(g)=\ln^C_{\cstar}( g\hat{\otimes} e)$.
Therefore, we have 
\eqalign{
\ln^C (g)\hat{\otimes} e+e\hat{\otimes} \ln^C(g)
=
& \ln_{\cstar}^C(g\hat{\otimes} e) +\ln_{\cstar}^C(e\hat{\otimes} g)
=  
\ln_{\cstar}^C \big((g \hat{\otimes} e) {\cstar}(e\hat{\otimes} g)\big)
=
\ln_{\cstar}^C(g\hat{\otimes} g).
}
On the other hand,
we have
\eqalign{
\cp_\O\circ \bar g^{\star n}
=
&\cp_\O\circ  m_\O^{(n)}\circ \bar g^{\hat{\otimes} n}\circ \cp^{(n)}_C
= 
m^{(n)}_{\O\hat{\otimes} \O}
\circ (\cp_\O\hat{\otimes} \ldots\hat{\otimes} \cp_\O)\circ  \bar g^{\hat{\otimes} n}\circ \cp^{(n)}_C
\cr
=
&
m^{(n)}_{\O\hat{\otimes} \O}\circ(g\hat{\otimes} g -e\hat{\otimes} e)^{\hat{\otimes} n} 
\circ  (\cp_C\otimes \ldots \otimes \cp_C)\circ \cp^{(n)}_{C} 
\cr
=
& 
m^{(n)}_{\O\hat{\otimes} \O}
\circ(g\hat{\otimes} g -e\hat{\otimes} e)^{\hat{\otimes} n} 
\circ \cp^{(n)}_{C\otimes C}\circ \cp_C 
\cr
=
& (g \hat{\otimes} g -e\hat{\otimes} e)^{{\cstar}n}\circ \cp_C
,
}
where we have used 
$\cp_\O\circ\bar g = \cp_\O \circ g  - \cp_\O\circ e=(g\hat{\otimes} g -e\hat{\otimes} e)\circ\cp_\O$ 
for the $3$rd equality.
Therefore, we obtain that
\eqalign{
\cp_\O\circ \ln^C (g) = 
&
- \sum_{n=1}^\infty\Fr{(-1)^n}{n} (g \hat{\otimes} g -e\hat{\otimes} e)^{{\cstar}n} \circ \cp_C  
=
\ln_{\cstar}^C(g\hat{\otimes} g) \circ \cp_C 
\cr
= 
&
\Big(\ln^C (g)\hat{\otimes} e+e\hat{\otimes} \ln^C(g)\Big)\circ \cp_C. 
}

3.  It is obvious now that 
$\ln^C \big(\exp^C(\ups)\big) =\ups$ and $\exp^C \big(\ln^C(g)\big) =g$. Hence $(\exp^C ,\ln^C)$ is
an isomorphism.

4. 
Let $\ups\sim \tilde\ups \in \THOM_{\ccdgc}(C, \O)$. 
Then we have a corresponding  homotopy pair $\big(\ups(t),\s(t)\big)$  on $\THOM_{\ccdgc}(C, \O)$
such that $\ups(0)=\ups$ and $\ups(1)=\tilde\ups$. Let
\eqalign{
g(t) :=& \exp^C \big(\ups(t)\big)
,\cr
\l(t) :=& 
\sum_{n=1}^\infty\sum_{j=1}^n\Fr{1}{n!} \ups(t)^{\star j-1}\star \s(t) \star \ups(t)^{\star n-j}
.
}
Then it is trivial to check that $\big(g(t),\l(t)\big)$ is a homotopy pair on $\HOM_{\ccdgc}(C, \O)$,
so that $\exp^C (\ups)=g(0)\sim g(1)= \exp^C (\tilde\ups) \in \HOM_{\ccdgc}(C, \O)$.

5. 
Let $g\sim \tilde g \in \HOM_{\ccdgc}(C, \O)$ 
and $\big(g(t),\l(t)\big)$  be the corresponding  homotopy pair on $\HOM_{\ccdgc}(C, \O)$
such that $g(0)=g$ and $g(1)=\tilde g$. Let
\eqalign{
v(t) :=& \ln^C \big(g(t)\big)
,\cr
\s(t) :=& - \sum_{n=1}^\infty\sum_{j=1}^n\Fr{(-1)^n}{n} \bar g(t)^{\star j-1}\star \l(t) \star \bar g(t)^{\star n-j}
.
}
Then it is trivial to check that $\big(\ups(t),\s(t)\big)$ is a homotopy pair on $\THOM_{\ccdgc}(C, \O)$,
so that $ \ln^C (g)=\ups(0) \sim \ups(1)= \ln^C (\tilde g) \in \THOM_{\ccdgc}(C, \O)$.
\qed
\end{proof}

\begin{lemma}\label{quslz}
Let $\O$ be a complete ccdg-Hopf algebra.  Then
we have a natural isomorphism
$
\xymatrixcolsep{3pc}
\xymatrix{ \grave{\bm{T\!\CP}}_{{\O}}\ar@/^/@{=>}[r]^{{\exp}}  
& \ar@/^/@{=>}[l]^{{\ln}} \grave{\bm{\CP}}_{{\O}}}
:\mathring{\category{ccdgC}}(\Bbbk)\rightsquigarrow \category{Set}
$ of presheaves on ${\category{ccdgC}}(\Bbbk)$,
whose  component at each ccdg-coalgebra $C$ is $(\exp^C, \ln^C)$
defined in Lemma \ref{quslx}.
\end{lemma}

\begin{proof}
We have shown that  
$\xymatrix{ \grave{\bm{T\!\CP}}_{{\O}}(C)  \ar@/^/[r]^-{\exp^C}&\ar@/^/[l]^-{\ln^C}  \grave{\bm{\CP}}_{{\O}}(C)}$ 
is an isomorphism for every ccdg-coalgebra $C$. It remains to check
the naturalness  that for every morphism $f:C \rightarrow C^\pr$ of ccdg-coalgebras
the diagrams are commutative:
\eqalign{
\xymatrixrowsep{1.5pc}
\xymatrixcolsep{3pc}
\xymatrix{
\ar[d]_-{\exp^{C^\pr}}
\grave{\bm{T\!\CP}_{\O}}(C^\pr)\ar[r]^-{\grave{\bm{T\!\CP}}_{\O}(f)} & \grave{\bm{T\!\CP}}_{\O}(C) 
\ar[d]^-{\exp^{C}}
\cr
\grave{\bm{\CP}}_{\O}(C^\pr)\ar[r]^-{\grave{\bm{\CP}}_{\O}(f)} & \grave{\bm{\CP}}_{\O}(C) 
}
,\qquad
\xymatrix{
\ar[d]_-{\ln^{C^\pr}}
\grave{\bm{\CP}}_{\O}(C^\pr) \ar[r]^-{\grave{\bm{\CP}}_{\O}(f)} & \grave{\bm{\CP}}_{\O}(C) 
\ar[d]^-{\ln^{C}}
\cr
\grave{\bm{T\!\CP}}_{\O}(C^\pr)\ar[r]^-{\grave{\bm{T\!\CP}}_{\O}(f)} & \grave{\bm{T\!\CP}}_{\O}(C) 
}.
}
That   is,  
$\grave{\bm{\CP}}_{\O}(f)\circ\exp^{C^\pr} =\exp^C\circ \grave{\bm{T\!\CP}}_{\O}(f)$ and
$\grave{\bm{T\!\CP}}_{\O}(f)\circ\ln^{C^\pr} =\ln^C\circ \grave{\bm{\CP}}_{\O}(f)$.  These are
straightforward since for every $\ups^\pr \in \grave{\bm{T\!\CP}}_{\O}(C^\pr)$ we have
\eqalign{
\grave{\bm{\CP}}_{\O}(f)\left(\exp^{C^\pr} (\ups^\pr)\right) 
&= \exp^{C^\pr}(\ups^\pr)\circ f 
=u_\O\circ \ep_{C^\pr} \circ f
+ \sum_{n=1}^\infty\Fr{1}{n!}
m^{(n)}_{{\O}}\circ \big({\ups}^\pr\hat{\otimes} \ldots \hat{\otimes}{\ups}^\pr\big)\circ \cp^{(n)}_{C^\pr}\circ f
\cr
&=u_\O\circ \ep_{C} 
+ \sum_{n=1}^\infty\Fr{1}{n!}
m^{(n)}_{{\O}}\circ \big({\ups}^\pr\circ f\hat{\otimes} \ldots \hat{\otimes}{\ups}^\pr\circ f\big)\circ \cp^{(n)}_C
=  \exp^{C}(\ups^\pr\circ f) 
\cr
&= \exp^C\left( \grave{\bm{T\!\CP}}_{\O}(f)(\ups^\pr)\right).
}
The naturalness of $\ln$ can be checked similarly.
\qed
\end{proof}

\begin{proof}[Theorem \ref{dgquillen}]
We note that
the components 
$\xymatrixcolsep{1.5pc}
\xymatrix{ \grave{\bm{T\mP}}_{{\O}}(C)\ar@/^0.5pc/[r]^{\bm{\exp}^C}  
& \ar@/^0.5pc/[l]^{\bm{\ln}^C} \grave{\bm{\mP}}_{\O}(C)}$
of  $\bm{\exp}$ and $\bm{\ln}$ at every ccdg-coalgebra $C$ 
are defined such that $\bm{\exp}^{C}\big([\ups]\big) =\left[ \exp^C(\ups) \right]$
and $\bm{\ln}^{C}\big([g]\big) =\left[ \ln^C(g) \right]$.  Due to Lemma \ref{quslx},
they are well defined, depending only on the homotopy types of $\ups$ and $g$,
isomorphisms for every ccdg-coalgebra $C$.
It remains to check the naturalness of $\bm{\exp}$ and $\bm{\ln}$
that for every $[f] \in \HOM_{\hccdgc}(C, C^\pr)$ the following diagrams commute
\eqalign{
\xymatrixrowsep{1.5pc}
\xymatrixcolsep{3pc}
\xymatrix{
\ar[d]_-{\bm{\exp}^{C^\pr}}
\grave{\bm{T\!\mP}_{\O}}(C^\pr)\ar[r]^-{\grave{\bm{T\!\mP}}_{\O}([f])} & \grave{\bm{T\!\mP}}_{\O}(C) 
\ar[d]^-{\bm{\exp}^{C}}
\cr
\grave{\bm{\mP}}_{\O}(C^\pr)\ar[r]^-{\grave{\bm{\mP}}_{\O}([f])} & \grave{\bm{\mP}}_{\O}(C) 
}
,\qquad
\xymatrix{
\ar[d]_-{\bm{\ln}^{C^\pr}}
\grave{\bm{\mP}}_{\O}(C^\pr) \ar[r]^-{\grave{\bm{\CP}}_{\O}([f])} & \grave{\bm{\mP}}_{\O}(C) 
\ar[d]^-{\bm{\ln}^{C}}
\cr
\grave{\bm{T\!\mP}}_{\O}(C^\pr)\ar[r]^-{\grave{\bm{T\!\mP}}_{\O}([f])} & \grave{\bm{T\!\mP}}_{\O}(C) 
}.
}
Here we will check the naturalness of $\bm{\exp}$ only, since the proof is similar of $\bm{\ln}$.

Let $f \in \HOM_{\ccdgc}(C, C^\pr)$ be an arbitrary representative of $[f]$. 
Consider any $[\ups^\pr] \in \THOM_{\hccdgc}(C^\pr,\O)$ and
let  $\ups^\pr \in \THOM_{\ccdgc}(C^\pr,\O)$ be  an arbitrary representative of $[\ups^\pr]$.
Then it is straightforward to check that the homotopy type 
$[\ups^\pr\circ f]$ of $\ups^\pr\circ f \in \THOM_{\ccdgc}(C,\O)$ 
depends only on $[f]$ and $[\ups^\pr]$. From Lemma \ref{quslz}, it also follow that 
 the homotopy type $\big[\exp^{C}(\ups^\pr\circ f)\big]$ of $\exp^{C}(\ups^\pr\circ f) \in \HOM_{\ccdgc}(C,\O)$
depends only on $[f]$ and $[\ups^\pr]$.  It is also obvious that the homotopy type
$\left[\exp^{C^\pr}(\ups^\pr)\circ f \right]$ of  $\exp^{C^\pr}(\ups^\pr)\circ f \in \HOM_{\ccdgc}(C,\O)$
depends only on $[f]$ and $[\ups^\pr]$. 
Combined with the identity
$\exp^{C^\pr}(\ups^\pr)\circ f= \exp^{C}(\ups^\pr\circ f)$ in the proof of  
Lemma \ref{quslz}, we have
\eqalign{
\grave{\bm{\mP}}_{\O}([f])\left(\bm{\exp}^{C^\pr} ([\ups^\pr])\right) 
&= \left[\exp^{C^\pr}(\ups^\pr)\circ f \right]
=\left[ \exp^{C}(\ups^\pr\circ f) \right]
=
\bm{\exp}^C\left( \grave{\bm{T\!\mP}}_{\O}([f])([\ups^\pr])\right).
}
Hence 
$\bm{\exp}:\grave{\bm{T\mP}}_\O \Rightarrow \grave{\bm{\mP}}_\O
: \mathring{\mathit{ho}\category{ccdgC}}(\Bbbk)\rightsquigarrow \category{Set}$
is a natural isomorphism. 
\qed
\end{proof}

\section{Linear representation of a representable presheaf   of groups}

Throughout this section we fix a  ccdg-Hopf algebra $\O=(\O,u_\O, m_\O,\ep_\O, \cp_\O,\vs_\O,\rd_\O)$.
We define a linear representation of 
the presheaf of groups 
$\bm{\mP}_{\!\!\O}:\mathring{\mathit{ho}\category{ccdgC}}(\Bbbk)\rightsquigarrow \category{Grp}$
on the homotopy category $\mathit{ho}\category{ccdgC}(\Bbbk)$ of ccdg-coalgebras
via a linear representation of
the associated presheaf of groups 
$\bm{\CP}_{\!\!\O}:\mathring{\category{ccdgC}}(\fieldk)\rightsquigarrow \category{Grp}$ 
on the category $\category{ccdgC}(\fieldk)$ of ccdg-coalgebras.
Remind that   $\bm{\CP}_{\!\!\O}$  is represented by $\O$ and induces $\bm{\mP}_\O$ 
on the homotopy category $\mathit{ho}\category{ccdgC}(\fieldk)$.

The linear representations   of $\bm{\CP}_{\!\!\O}$ form a dg-tensor category
$\gdcat{Rep}(\bm{\CP}_{\!\!\O})$, which shall be isomorphic 
to the dg-tensor category $\gdcat{dgMod}_L(\O)$ of 
left dg-modules over $\O$.  Working with the linear representations of 
$\bm{\CP}_{\!\!\O}$ instead of the linear representations of $\bm{\mP}_{\!\!\O}$ 
will be a crucial step for a Tannakian reconstruction of  $\bm{\mP}_{\!\!\O}$.

\subsection{Preliminary}

Our main concern here is a dg-tensor category formed by cofree left dg-comodules
over a  ccdg-coalgebra $C=(C,\ep_C,\cp_C,\rd_C)$.
We shall need the following basic lemma, which is  due to the defining properties  of dg-coalgebra$C$.
\begin{lemma}
\label{cbasicl} 
For every pair $(M,N)$ of chain complexes we have an exact sequence of chain complexes
\eqalign{
\xymatrixcolsep{2pc}
\xymatrixrowsep{0pc}
\xymatrix{
0\ar[r]  & \Hom\big(C\!\otimes\! M, N\big)\ar[r]^-{\check\mp} & \ar@/^1pc/@{..>}[l]^{\check\mq} 
\Hom\big(C\!\otimes\! M, C\!\otimes\! N\big)\ar[r]^-{\check\mr} &   \ar@/^1pc/@{..>}[l]^{\check\ms} 
\Hom\big(C\!\otimes\! M, C\!\otimes\! C\!\otimes\! N\big)
,
}
}
where $\forall \a_{i+1} \in \Hom\big(C\!\otimes\! M, C^{\otimes i}\otimes N\big)$, $i=0,1,2$,
\eqn{coalgext}{
\begin{aligned}
\check{\mp}(\a_1)&:=(\I_C\otimes \a_1)\circ (\cp_C\otimes \I_M)
,\cr
\check{\mq}(\a_2)&:=\imath_N\circ (\ep_C\otimes \I_N)\circ \a_2
,
\end{aligned}
\qquad
\begin{aligned}
\check\mr(\a_2)&:=(\cp_C\otimes \I_N)\circ \a_2- (\I_C\otimes \a_2)\circ (\cp_C \otimes I_M)
,\cr
\check{\ms}(\a_3)&:=\imath_N\circ (\I_C\otimes\ep_C\otimes \I_N)\circ \a_3
,
\end{aligned}
}
such that 
\eqn{coalgextx}{
\check\mr\circ\check\mp=0
,\qquad   
\check\mq\circ \check\mp =\I_{\Hom(C\otimes M,N)}
,\qquad
\check\mp\circ\check\mq+\check\ms\circ\check\mr =\I_{\Hom(C\otimes M,C\otimes N)}.
}
\end{lemma}

Remind that a left dg-comodule  $(M,\r_M)$ over a ccdg-coalgebra $C$ is a chain complex 
$M=(M,\rd_M)$ together with a chain map $\r_M: M\rightarrow C\otimes M$, called a coaction, making the following
diagrams commute
\eqn{leftcomodule}{
\xymatrixcolsep{3pc}
\xymatrix{
\ar[d]_-{\r_M}
M\ar[r]^{\r_M} & C\otimes M
\ar[d]^-{\I_C\otimes \r_M}
\cr
C\otimes M \ar[r]^-{\cp_C\otimes \I_M} & C\otimes C\otimes M
},
\qquad
\xymatrix{
\ar[dr]_-{\imath^{-1}_M}
M\ar[r]^{\r_M} & C\otimes M
\ar[d]^{\ep_C\otimes \I_M}
\cr
&\Bbbk\otimes M
}.
}

For every chain complex $M$ we have 
a \emph{cofree} left dg-comodule  $(C\otimes M, \cp_C\otimes \I_M)$ over $C$ with the cofree coaction
$\xymatrixcolsep{3pc}\xymatrix{C\otimes M \ar[r]^-{\cp_C\otimes \I_M}& C\otimes C\otimes M}$.
We 
can form a dg-category   $\gdcat{{dgComod}}^{\mathit{cofr}}_L(C)$ of cofree left dg-comodules over $C$,
whose the set of morphisms from $(C\!\otimes\! M,\cp_C\!\otimes\! \I_M)$ to  $(C\!\otimes\! N, \cp_C\!\otimes\! \I_M)$
is  $\Hom_{\cp_C\!}\big(C\!\otimes\! M, C\!\otimes\! N\big)$ with the differential $\rd_{C\otimes M\!,C\otimes N}$,
where $\Hom_{\cp_C\!}\big(C\!\otimes\! M, C\!\otimes\! N\big)$ is the set of
every $\Bbbk$-linear map $\w:C\otimes M \rightarrow C\otimes {N}$ making the following diagram commutative
$$
\xymatrixcolsep{3pc}
\xymatrix{
\ar[d]_-{\w}
C\otimes M \ar[r]^-{\cp_C\otimes \I_M}& C\otimes C\otimes M
\ar[d]^-{\I_C\otimes \w}
\cr
C\otimes {N} \ar[r]_-{\cp_C\otimes\I_{N}}& C\otimes C\otimes {N}
},\quad\hbox{i.e.,}\quad
(\cp_C\otimes \I_{N})\circ \w = (\I_{C}\otimes \w)\circ (\cp_C\otimes \I_M) \Longleftrightarrow \check\mr(\w)=0.
$$

\begin{corollary} 
There is a bijection 
$\xymatrix{\check{\mp}:\Hom(C\!\otimes\! M, N) \ar@/^/[r] & \ar@/^/[l]: \Hom_{\cp_C\!}\big(C\!\otimes\! M, C\!\otimes\! N\big): \check{\mq}}$.
\end{corollary}

By Lemma \ref{cbasicl}, 
we can check that
$\Hom_{\cp_C\!}\big(C\!\otimes\! M, C\!\otimes\! {M^\pr}\big)$ is a chain complex with the differential
$\rd_{C\!\otimes\! M, C\!\otimes\! {M^\pr}}$ 
and $\w^\pr\circ \w \in \Hom_{\cp_C\!}\big(C\!\otimes\! M, C\!\otimes\! {M^\ppr}\big)$
whenever $\w\in \Hom_{\cp_C\!}\big(C\!\otimes\! M, C\!\otimes\! {M^\pr}\big)$ and 
$\w^\pr\in \Hom_{\cp_C\!}\big(C\!\otimes\! M^\pr, C\!\otimes\! {M^\ppr}\big)$. It is obvious that
differentials are derivations of  the composition operation. 
Therefore, $\gdcat{{dgComod}}^{\mathit{cofr}}_L(C)$ is a dg-category.

\begin{lemma}[Definition]\label{tensorfrcomod}
The dg-category  $\gdcat{{dgComod}}^{_{\mathit{cofr}}}_L(C)$ 
is a dg-tensor category with the following tensor structure.

\begin{enumerate}
\item
The tensor product of cofree left dg-comodules 
$(C\otimes M,\cp_C\otimes \I_{M})$ 
and  $(C\otimes M^\pr,\cp_C\otimes \I_{M^\pr})$ 
over $C$ is the cofree left dg-comodule 
$\big(C\otimes M\otimes M^\pr, \cp_C\otimes \I_{M\otimes M^\pr}\big)$ 
over $C$,  and $(C\otimes \Bbbk, \cp_C\otimes \I_\Bbbk)$ is the unit object.

\item
The tensor product of morphisms $\w \in \Hom_{\cp_C}\big(C\otimes M, C\otimes N\big)$
and $\w^\pr \in \Hom_{\cp_C}\big(C\otimes M^\pr, C\otimes N^\pr\big)$ is the morphism
$\w\otimes_{\cp_C}\!\w^\pr \in \Hom_{\cp_C}\big(C\otimes M\otimes M^\pr, C\otimes N\otimes N^\pr\big)$,
where
\eqalign{
\w \otimes_{\cp_C}& \!\w^\pr
: = \big(\w\otimes \check{\mq}(\w^\pr)\big)\circ (\I_C\otimes \t\otimes \I_{M^\pr})
\circ (\cp_C\otimes \I_M\otimes \I_{M^\pr})
\cr
&
\xymatrixcolsep{3.2pc}
\xymatrix{
C{\!\otimes\!} M{\!\otimes\!} M^\pr\ar[r]^-{\cp_C{\otimes} \I_{M}{\otimes} \I_{M^\pr}}
&C{\!\otimes\!} C{\!\otimes\!} M{\!\otimes\!} M^\pr \ar[r]^-{\I_C{\otimes} \t{\otimes} \I_{M^\pr}} 
& C{\!\otimes\!} M{\!\otimes\!} C{\!\otimes\!} M^\pr
\ar[r]^-{\w{\otimes} \check{\mq}(\w^\pr)}
& C{\!\otimes\!} N^\pr {\!\otimes\!} N^\pr.}
}

Equivalently, $\w\otimes_{\cp_C}\w^\pr$ is determined by the following equality:
\[
\check{\mq}(\w\otimes_{\cp_C}\w^\pr)
=\big(\check{\mq}(\w)\otimes\check{\mq}(\w^\pr)\big)
\circ(\I_C\otimes\t\otimes \I_{M^\pr})\circ(\cp_C\otimes\I_{M\otimes M^\pr})
:C\otimes M\otimes M^\pr\to N\otimes N^\pr.
\]

\item
$\rd_{C\otimes M\otimes M^\pr,\! C\otimes N\otimes N^\pr}
\big(\w\otimes_{\cp_C}\!\w^\pr\big)
=\rd_{C\otimes M\!, C\otimes N} \w\otimes_{\cp_C}\!\w^\pr   
+(-1)^{|\w|} \w\otimes_{\cp_C}\!\rd_{C\otimes M^\pr\!, C\otimes N^\pr}\w^\pr$.
\end{enumerate}

\end{lemma}

\begin{proof} Property $1$ is obvious. For property $2$, 
it is straightforward to check that $\check{\mr}(\w\otimes_{\cp_C} \!\w^\pr)=0$ 
whenever $\check{\mr}(\w)=\check{\mr}(\w^\pr)=0$.
Property $3$ can be checked by a straightforward computation.
\qed
\end{proof}

\begin{lemma}\label{ctensoring}
 We have  a dg-tensor functor 
$C\otimes :\gdcat{Ch}(\Bbbk) \rightsquigarrow \gdcat{dgComod}^{\mathit{cofr}}_L(C)$
for  every ccdg-coalgebra $C$
\eqalign{
C\otimes \Big(
\xymatrix{ M \ar[r]^-{\p} & M^\pr}\Big)
\rightsquigarrow  \Big( \xymatrix{(C\otimes M, \cp_C\otimes \I_M)\ar[r]^-{\I_C\otimes \p}& 
(C\otimes M^\pr, \cp_C\otimes \I_{M^\pr})}\Big).
}
\end{lemma}

\begin{proof}
It is obvious that $C\otimes$ is a dg-functor  whose tensor property follows from
the easy identity 
$
(\I_C\otimes \p)\otimes_{\cp_C\!} (\I_C\otimes \p^\pr) 
=\I_C\otimes \p\otimes \p^\pr: C\otimes M\otimes N\rightarrow
C\otimes M^\pr\otimes N^\pr
$
for all linear maps $\p:M \rightarrow N$ and $\p^\pr:M^\pr \rightarrow N^\pr$.
\qed
\end{proof}

\subsection{Linear representations and their dg-tensor category}

For every ccdg-coalgebra  $C$, we have the following constructions.
\begin{enumerate}
\item
Let
$\End_{\cp_C\!}(C\otimes M):=\Hom_{\cp_C\!}(C\otimes M,C\otimes M)$, 
which is the $\Z$-graded vector space of linear maps 
$\w:C\otimes M \rightarrow C\otimes M$ satisfying 
$(\cp_C\otimes \I_M)\circ \w = (\I_C\otimes \w)\circ (\cp_C\otimes \I_M)$.
Then we have a dg-algebra 
\eqn{EMC}{
\bm{\CE}_{\!\!M}(C)= \big(\End_{\cp_C\!}(C\otimes M), \I_{C\otimes M}, \circ ,\rd_{C\otimes M, C\otimes M}\big).
}

\item
Let $\operatorname{Z_0Aut}_{\cp_C\!}(C\otimes M)$ be the subset of $\End_{\cp_C\!}(C\otimes M)$ 
consisting every degree zero element $\w$ that has a composition inverse 
$\w^{-1}$ and satisfies $\rd_{C\otimes M,C\otimes M}\w=0$. 
Then we have a group 
\eqn{GLMC}{
\bm{\CG\!\ell}_{\!\!M}(C) :=\big(\operatorname{Z_0Aut}_{\cp_C\!}(C\otimes M), \I_{C\otimes M}, \circ\big).
}

\item 
Let $\operatorname{H_0Aut}_{\cp_C\!}(C\otimes M)$ be the set of  homology classes of elements in 
$\operatorname{Z_0Aut}_{\cp_C\!}(C\otimes M)$ 
that $\w, \tilde \w \in \operatorname{Z_0Aut}_{\cp_C\!}(C\otimes M)$
belongs to the same homology class 
$\w\sim \tilde \w$, i.e., $[\w]=[\tilde \w]\in \operatorname{H_0Aut}_{\cp_C\!}(C\otimes M)$, 
if $\tilde \w - \w =\rd_{C\otimes M, C\otimes M}\l$ for some $\l \in \End_{\cp_C\!}(C\otimes M)$. 
We can check that $\w_1\circ \w_2\sim \tilde \w_1\circ \tilde \w_2\in\operatorname{Z_0Aut}_{\cp_C\!}(C\otimes M)$
whenever $\w_1\sim \tilde \w_1,  \w_2\sim \tilde \w_2\in\operatorname{Z_0Aut}_{\cp_C\!}(C\otimes M)$,
and
the $\w^{-1}\sim \tilde \w^{-1}\in\operatorname{Z_0Aut}_{\cp_C\!}(C\otimes M)$
whenever $\w\sim \tilde \w \in\operatorname{Z_0Aut}_{\cp_C\!}(C\otimes M)$. 
Let $[\w_1]\diamond [\w_2]:=[\w_1\circ \w_2]$ and $[\w]^{-1}:= [\w^{-1}]$. Then we have a group
\eqn{MGLMC}{
\bm{\mG\!\ml}_{\!\!M}(C) :=\big(\operatorname{H_0Aut}_{\cp_C\!}(C\otimes M), [\I]_{C\otimes M}, \diamond\big).
}
\end{enumerate}

The above constructions are functorial .

\begin{lemma}\label{cdullpain}
For every chain complex $M$
we have a functor
$\bm{\CE}_{\!\!M}:\mathring{\category{ccdgC}}(\Bbbk) \rightsquigarrow \category{dgA}(\Bbbk)$, sending
each ccdg-coalgebra $C$ to
the dg-algebra $\bm{\CE}_{\!\!M}(C)$,
and
each morphism $f:C\rightarrow C^\pr$ of ccdg-coalgebras to
a morphism  $\bm{\CE}_{\!\!M}(f):\bm{\CE}_{\!\!M}(C^\pr)\rightarrow \bm{\CE}_{\!\!M}(C)$ of dg-algebras
defined by, $\forall \w^\pr \in\grave{\bm{\CE}}_{\!\!M}(C^\pr)= \End_{\cp_{C^\pr}\!}(C^\pr\otimes M)$,
\eqalign{
\bm{\CE}_{\!\!M}(f)(\w^\pr)
:=
&
\check{\mp}\Big(\check{\mq}\big(\w^\pr\big)\circ  (f\otimes \I_M)\big)\Big)
\cr
=
&
\big(\I_C\otimes \check\mq(\w^\pr)\big)
\circ
(\I_{C}\otimes f\otimes \I_M)
\circ (\cp_{C}\otimes \I_M)
\cr
=
&
\big(\I_C\otimes \imath_M\circ (\ep_{C^\pr}\otimes \I_M)\circ \w^\pr\big)
\circ(\I_{C}\otimes f\otimes \I_M)
\circ (\cp_{C}\otimes \I_M)
\cr
:&
\xymatrixcolsep{3pc}
\xymatrix{
C\otimes M\ar[r]^-{\cp_{C}\otimes \I_M} & C\otimes C\otimes M\ar[r]^{\I_{C}\otimes f\otimes I_M}
& C\otimes C^\pr\otimes M \ar[r]^-{\I_C\otimes \check\mq(\w^\pr)} & C\otimes M.
}  
}
That is, we have $\bm{\CE}_{\!\!M}(f)(\w^\pr) \in \grave{\bm{\CE}}_{\!\!M}(C)= \End_{\cp_{C}\!}(C\otimes M)$, and
\begin{enumerate}[label=({\alph*})]

\item $\bm{\CE}_{\!\!M}(f)(\I_{C^\pr\otimes M})= \I_{C\otimes M}$;

\item $\bm{\CE}_{\!\!M}(f)\big(\w^\pr_1\circ \w^\pr_2\big)
=\bm{\CE}_{\!\!M}(f)\big(\w^\pr_1\big)\circ \bm{\CE}_{\!\!M}(f)\big(\w^\pr_2\big)$;

\item $\bm{\CE}_{\!\!M}(f)\circ \rd_{C^\pr\otimes M, C^\pr\otimes M} 
= \rd_{C\otimes M, C\otimes M}\circ \bm{\CE}_{\!\!M}(f)$.

\end{enumerate}
\end{lemma}

\begin{proof}
We already know that $\bm{\CE}_{\!\!M}(C)$ is a dg-algebra \eq{EMC}. 
 We check that $\bm{\CE}_{\!\!M}(f)(\w^\pr) \in  \End_{\cp_{C}\!}(C\otimes M)$, i.e.,
$\check\mr\left(\bm{\CE}_{\!\!M}(f)(\w^\pr)\right)=0$. We have
\eqalign{
\check\mr\left(\bm{\CE}_{\!\!M}(f)(\w^\pr)\right)
:=
&
(\cp_C\otimes \I_{M^\pr})\circ (\I_C\otimes \check\mq(\w^\pr))\circ
(\I_{C}\otimes f\otimes I_M)\circ (\cp_{C}\otimes \I_M)
\cr
&
-
(\I_{C}\otimes\I_C\otimes \check\mq(\w^\pr))
\circ(\I_{C}\otimes\I_{C}\otimes f\otimes I_M)
\circ (\I_{C}\otimes\cp_{C}\otimes \I_M)
\circ (\cp_C\otimes \I_M)
\cr
=
&0,
}
where we used the coassociativity of $\cp_C$.
It remains to show that $\bm{\CE}_{\!\!M}(f)$ is a morphism of dg-algebras---the properties $(a)$, $(b)$ and $(c)$. 
Remind that
$\check{\mq}(\w^\pr):=\imath_M\circ (\ep_{C^\pr}\otimes \I_M)\circ \w^\pr$.

For the property $(a)$, we have
\eqalign{
\bm{\CE}_{\!\!M}(f)(\I_{C^\pr\otimes M})
:=
&
(\I_C\otimes \imath_M\circ (\ep_{C^\pr}\otimes \I_M))\circ
(\I_{C}\otimes f\otimes \I_M)\circ (\cp_{C}\otimes \I_M)
= \I_{C\otimes M}
,
}
where we have used $\ep_{C^\pr}\circ f =\ep_C$
and the counit property of $\cp_C$.
For the property $(b)$, we have
\eqalign{
\bm{\CE}_{\!\!M}(f)&(\w^\pr_1)\circ \bm{\CE}_{\!\!M}(f)(\w^\pr_2)
:=
(\I_C\otimes \imath_M\circ (\ep_{C^\pr}\otimes \I_M)\circ\w^\pr_1)
\circ(\I_{C}\otimes f\otimes \I_M)
\circ (\cp_{C}\otimes \I_M)
\cr
&
\circ
(\I_C\otimes\imath_M\circ (\ep_{C^\pr}\otimes \I_M)\circ\w^\pr_2)
\circ (\I_{C}\otimes f\otimes \I_M)
\circ (\cp_{C}\otimes \I_M)
\cr
=
&
(\I_C\otimes \imath_M\circ (\ep_{C^\pr}\otimes \I_M)\circ\w^\pr_1)
\circ(\I_{C}\otimes f\otimes \I_M)
\circ(\I_C\otimes\I_C\otimes \imath_M\circ (\ep_{C^\pr}\otimes \I_M)\circ\w^\pr_2)
\cr
&
\circ (\I_{C}\otimes\I_C\otimes f\otimes \I_M)
\circ\big({\color{blue}(\cp_{C}\otimes\I_C)\circ\cp_{C}} \otimes \I_M\big)
\cr
=
&
(\I_C\otimes \imath_M\circ (\ep_{C^\pr}\otimes \I_M)\circ\w^\pr_1)
\circ(\I_C\otimes \I_{C^\pr}\otimes \imath_M\circ (\ep_{C^\pr}\otimes \I_M))
\circ(\I_C\otimes \I_{C^\pr}\otimes  \w^\pr_2)
\cr
&
\circ  \big(\I_C\otimes  {\color{blue}(f\otimes f)\circ \cp_C}\otimes \I_M\big)
\circ (\cp_{C}\otimes \I_M)
\cr
=
&
(\I_C\otimes \imath_M\circ (\ep_{C^\pr}\otimes \I_M)\circ\w^\pr_1)
\circ(\I_C\otimes \I_{C^\pr}\otimes \imath_M\circ (\ep_{C^\pr}\otimes \I_M))
\cr
&
\circ
 \big(\I_C\otimes {\color{blue}(\I_{C^\pr}\otimes \w^\pr_2)\circ (\cp_{C^\pr}\otimes \I_M)}\big)
\circ (\I_{C}\otimes f\otimes \I_M)
\circ (\cp_{C}\otimes \I_M)
\cr
=
&
(\I_C\otimes \imath_M\circ (\ep_{C^\pr}\otimes \I_M)\circ\w^\pr_1)
\circ(\I_C\otimes \I_{C^\pr}\otimes \imath_M)
\circ \big(\I_C\otimes {\color{blue}(\I_{C^\pr}\otimes\ep_{C^\pr})\circ  \cp_{C^\pr}}\otimes \I_M\big)
\cr
&
\circ (\I_C\otimes \w^\pr_2)
\circ (\I_{C}\otimes f\otimes \I_M)
\circ (\cp_{C}\otimes \I_M)
\cr
=
&
(\I_C\otimes \imath_M\circ (\ep_{C^\pr}\otimes \I_M)\circ\w^\pr_1\circ \w^\pr_2)
\circ (\I_{C}\otimes f\otimes \I_M)
\circ (\cp_{C}\otimes \I_M)
\cr
=
&\bm{\CE}_{\!\!M}(f)\big(\w_1^\pr\circ \w_2^\pr\big)
,
}
where we have used the coassociativity of $\cp_C$ for the $3$rd equality,  
$f$ being a coalgebra map for the $4$th equality, 
$\w^\pr_2 \in \End_{\cp_{C^\pr}\!}(C^\pr\otimes M)$ that 
$ (\I_{C^\pr}\otimes \w^\pr_2)\circ (\cp_{C^\pr}\otimes \I_M)=(\cp_{C^\pr}\otimes \I_{M})\circ \w^\pr_2$
for the $5$th equality
and the counit property of $\cp_{C^\pr}$ for the $6$th equality,
while all the other moves are plain.
For the property $(c)$, we have
\eqalign{
\rd_{C\otimes M, C\otimes M}&\Big( \bm{\CE}_{\!\!M}(f)(\w^\pr)\Big)
:=
\rd_{C\otimes M}\circ \big(\I_C\otimes \imath_M\circ (\ep_{C^\pr}\otimes \I_M)\circ \w^\pr\big)
\circ(\I_{C}\otimes f\otimes \I_M)
\circ (\cp_{C}\otimes \I_M)
\cr
&
-(-1)^{|\w^\pr|}\big(\I_C\otimes \imath_M\circ (\ep_{C^\pr}\otimes \I_M)\circ \w^\pr\big)
\circ(\I_{C}\otimes f\otimes \I_M)
\circ (\cp_{C}\otimes \I_M)\circ \rd_{C\otimes M}
\cr
=
&
\bm{\CE}_{\!\!M}(f)\Big(\rd_{C\otimes M}\circ \w^\pr\Big) 
-(-1)^{\w^\pr}\bm{\CE}_{\!\!M}(f)\Big(\w^\pr\circ \rd_{C\otimes M}\Big)
=
 \bm{\CE}_{\!\!M}(f)\Big(\rd_{C\otimes M, C\otimes M}\w^\pr\Big)
 ,
}
where we have used properties that $\rd_C$ is a coderivation of $\cp_C$ and $f:C\rightarrow C^\pr$ is a chain map, together with
some obvious moves and cancelations.
\qed
\end{proof}

\begin{lemma}\label{cdullpaina}
For every chain complex $M$ we have 
a  presheaf of groups
$$
\bm{\CG\!\ell}_{\!\!M}:\mathring{\category{ccdgC}}(\Bbbk) \rightsquigarrow \category{Grp},
$$
sending each ccdg-coalgebra $C$ to the  group 
$\bm{\CG\!\ell}_{\!\!M}(C)$ 
and each morphism $f:C\rightarrow C^\pr$ of ccdg-coalgebras
to a  homomorphism
$\bm{\CG\!\ell}_{\!\!M}(f):\bm{\CG\!\ell}_{\!\!M}(C^\pr)\rightarrow \bm{\CG\!\ell}_{\!\!M}(C)$ of groups
defined by $\bm{\CG\!\ell}_{\!\!M}(f) := \bm{\CE}_{\!\!M}(f)$. Moreover, we have
\begin{enumerate}[label=({\alph*})]
\item
$\bm{\CG\!\ell}_{\!\!M}(\tilde f)(\w^\pr) \sim \bm{\CG\!\ell}_{\!\!M}(f)(\w^\pr)  
\in \operatorname{Z_0Aut}_{\cp_{C}\!}(C\otimes M)$
for all $\w^\pr \in \operatorname{Z_0Aut}_{\cp_{C^\pr}\!}(C^\pr\otimes M)$ 
whenever $f \sim \tilde f\in \HOM_{\ccdgc}(C, C^\pr)$, and

\item $\bm{\CG\!\ell}_{\!\!M}( f)(\tilde\w^\pr) \sim \bm{\CG\!\ell}_{\!\!M}( f)(\w^\pr)
\in \operatorname{Z_0Aut}_{\cp_{C}\!}(C\otimes M)$
for all $f \in \HOM_{\ccdgc}(C, C^\pr)$ 
whenever $\w^\pr \sim \tilde\w^\pr \in \operatorname{Z_0Aut}_{\cp_{C^\pr}\!}(C^\pr\otimes M)$.
\end{enumerate}
\end{lemma}

\begin{proof}
We already know that $\bm{\CG\!\ell}_{\!\!M}(C)$ is a group, \eq{GLMC}. 
We check that $\bm{\CG\!\ell}_{\!\!M}(f)$ is a group homomorphism as follows:
Due to properties $(a)$ and $(b)$ in Lemma \ref{cdullpain}, it is suffice to
check that  $\rd_{C\otimes M, C\otimes M}\Big(\bm{\CG\!\ell}_{\!\!M}(f)(\w^\pr)\Big)=0$ for every
$\w^\pr \in  \grave{\bm{\CG\!\ell}}_{\!\!M}(C^\pr)=\operatorname{Z_0Aut}_{\cp_{C^\pr}\!}(C^\pr\otimes M)$. 
This is obvious by property $(c)$ in Lemma \ref{cdullpain},
since $\bm{\CG\!\ell}_{\!\!M}(f)(\w^\pr):= \bm{\CE}_{\!\!M}(f)(\w^\pr)$
and $\rd_{C\otimes M, C\otimes M}\w^\pr=0$ by definitions.
Therefore $\bm{\CG\!\ell}_{\!\!M}$ is a presheaf of groups on ${\category{ccdgC}}(\Bbbk)$.

Property $(a)$ is checked as follows.
From the condition  $f\sim \tilde f \in \HOM_{\ccdgc}(C,C^\pr)$,  there is 
a homotopy pair $\big(f(t), \l(t)\big)$ on $\HOM_{\ccdgc}(C, C^\pr)$ such that $f(0)=f$, $f(1)=\tilde f$
and $\tilde f = f + \rd_{C\!, C^\pr}\c$, where $\c:=\int^1_0\c(t)\, \mathit{dt} \in \Hom(C,C^\pr)_{1}$.  
Then, for every $\w^\pr \in \operatorname{Z_0Aut}_{\cp_C\!}(C\otimes M)$ we have
\eqalign{
\bm{\CG\!\ell}_{\!\!M}(\tilde f)(\w^\pr) -\bm{\CG\!\ell}_{\!\!M}(f)(\w^\pr)
&=\check{\mp}\Big(\check{\mq}\big(\w^\pr\big)\circ  ( \rd_{C\!, C^\pr}\chi\otimes \I_M)\Big)
\cr
&=\rd_{C\otimes M,C\otimes M}\check{\mp}\Big(\check{\mq}\big(\w^\pr\big)\circ  (\chi\otimes \I_M)\Big),
}
since both $\check{\mp}$ and $\check{\mq}$ are chain maps and $\rd_{C^\pr\otimes M\!,C^\pr\otimes M}\w^\pr=0$.

Property $(b)$ is checked as follows.
From the condition $\w^\pr \sim \tilde\w^\pr \in \operatorname{Z_0Aut}_{\cp_{C^\pr}\!}(C^\pr\otimes M)$,
we have 
$ \tilde\w^\pr -\w^\pr =\rd_{C^\pr\otimes M\!, C^\pr\otimes M}\l$ for some 
$\l \in  \operatorname{End}_{\cp_{C^\pr}\!}(C^\pr\otimes M)$.
Then, for every $f \in \HOM_{\ccdgc}(C, C^\pr)$ we have
\eqalign{
\bm{\CG\!\ell}_{\!\!M}(f)(\tilde\w^\pr) -\bm{\CG\!\ell}_{\!\!M}(f)(\w^\pr)
&=\check{\mp}\Big(\check{\mq}\big(\rd_{C^\pr\otimes M\!, C^\pr\otimes M}\l\big)\circ  (f\otimes \I_M)\Big)
\cr
&=\rd_{C\otimes M,C\otimes M}\check{\mp}\Big(\check{\mq}\big(\l\big)\circ  (f\otimes \I_M)\Big)
,
}
since both $\check{\mp}$ and $\check{\mq}$ are chain maps and $\rd_{C\!,C^\pr}f=0$.
\qed
\end{proof}

\begin{lemma}\label{cdullpainax}
For every chain complex $M$ we have a  presheaf of groups
$$\bm{\mG\!\ml}_{\!\!M}:\mathring{\mathit{ho}\category{ccdgC}}(\Bbbk) \rightsquigarrow \category{Grp}$$
on ${\mathit{ho}\category{ccdgC}}(\Bbbk)$,
sending
each ccdg-coalgebra $C$ to the  group $\bm{\mG\!\ml}_{\!\!M}(C)$ and
each $[f]\in \HOM_{\mathit{ho}\ccdgc}(C, C^\pr)$ to
a group homomorphism
$\bm{\mG\!\ml}_{\!\!M}([f]) :\bm{\mG\!\ml}_{\!\!M}(C^\pr)\rightarrow \bm{\mG\!\ml}_{\!\!M}(C)$
defined by, for all 
$[\w^\pr]\in \operatorname{H_0Aut}_{\cp_{C^\pr}\!}(C^\pr\otimes M)=\grave{\bm{\mG\!\ml}}_{\!\!M}(C) $,
$$
\bm{\mG\!\ml}_{\!\!M}([f])([\w^\pr])
:=\big[\bm{\CG\!\ell}_{\!\!M}(f)(\w^\pr)\big]
,
$$
where $f \in  \HOM_{\ccdgc}(C, C^\pr)$ and $\w^\pr\in \operatorname{Z_0Aut}_{\cp_{C^\pr}\!}(C^\pr\otimes M)$ 
are arbitrary representatives of $[f]$
and $[\w^\pr]$, respectively.
\end{lemma}

\begin{proof}
We already know that $\bm{\CG\!\ell}_{\!\!M}(C)$ is a group, \eq{MGLMC}. 
Due to Lemma \ref{cdullpaina},  it remains to check that 
the homology class $\big[\bm{\CG\!\ell}_{\!\!M}(f)(\w^\pr)\big]$ of $\bm{\CG\!\ell}_{\!\!M}(f)(\w^\pr)$
depends only on the homology class $[\w^\pr]$ of $\w^\pr$ and the homotopy type $[f]$ of $f$,
which are  evident since 
$\bm{\CE}_{\!\!M}(f)\equiv \bm{\CG\!\ell}_{\!\!M}(f)$ is a chain map---property $(c)$ in Lemma  \ref{cdullpain},
and the homology class of $\bm{\CG\!\ell}_{\!\!M}(\tilde f)(\w^\pr)$ depends only on $[f]$ and $[g^\pr]$
due to properties $(a)$ and $(b)$ in Lemma \ref{cdullpaina}.
\qed
\end{proof}

We are ready to define  a linear representation of a representable presheaf of groups 
$\bm{\CP}_{\!\!\O}:\mathring{\category{ccdgC}}(\Bbbk)\rightsquigarrow \category{Grp}$.
\begin{definition}
A linear resentation of a representable presheaf  of groups
$\bm{\CP}_{\!\!\O}$ on the category $\category{ccdgC}(\Bbbk)$ of ccdg-coalgebras
is a pair $\big(M,\bm{\r}_{\!M}\big)$, where $M$ is a chain complex and  
${\bm{\r}_{\!M}: \bm{\CP}_{\!\!\O}\Rightarrow\bm{\CG\!\ell}_{\!\!M}
:\mathring{\category{ccdgC}}(\Bbbk)\rightsquigarrow \category{Grp}}$
is a natural transformation of the presheaves.
\end{definition}

\begin{remark}\label{repyoneda}
Note that $\bm{\r}_{\!M}: \bm{\CP}_{\!\!\O}\Rightarrow\bm{\CG\!\ell}_{\!\!M}$ 
is a natural transformation of contravariant functors:
\begin{itemize}
\item the component $\bm{\r}_{\!M}^{C}$ 
of $\bm{\r}_{\!M}$ at each cdg-coalgebra
$C$ is a homomorphism
$\bm{\r}_{\!M}^{C}:\bm{\CP}_{\!\!\O}(C)\rightarrow \bm{\CG\!\ell}_{\!\!M}(C)$ 
 of groups---we have $\bm{\r}_{\!M}^{C}(g) \in \operatorname{Z_0Aut}_{\cp_C\!}(C\otimes M)$ 
for every $g\in \HOM_{\ccdgc}(C, \O)$,
and  
\item
for every morphism $f:C\rightarrow C^\pr$ of ccdg-coalgebras the diagram commutes
\eqnalign{repnatural}{
\xymatrixrowsep{1.5pc}
\xymatrixcolsep{4pc}
\xymatrix{
\ar[d]_-{\bm{\r}_{\!M}^{C^\pr}}
\bm{\CP}_{\!\!\O}(C^\pr)\ar[r]^-{\bm{\CP}_{\!\!\O}(f)}&\bm{\CP}_{\!\!\O}(C)
\ar[d]^-{\bm{\r}_{\!M}^{C}}
\cr
\bm{\CG\!\ell}_{\!\!M}(C^\pr)\ar[r]^-{\bm{\CG\!\ell}_{\!\!M}(f)}&\bm{\CG\!\ell}_{\!\!M}(C)
},\quad{i.e.,}\quad
\bm{\r}_{\!M}^{C}\circ\bm{\CP}_{\!\!\O}(f) =\bm{\CG\!\ell}_{\!\!M}(f)\circ \bm{\r}_{\!M}^{C^\pr}
.
}
\end{itemize}
Since $\bm{\CP}_{\!\!\O}$ is representable, the Yoneda lemma implies that 
$\bm{\r}_{\!M}:\bm{\CP}_{\!\!\O}\Rightarrow \bm{\CG\!\ell}_{\!\!M}$
is completely determined by the universal element $\bm{\r}_{\!M}^\O(\I_\O):\O\otimes M\to \O\otimes M$.
Indeed,  the naturalness  of $\bm{\r}_{\!M}$ impose that
$\bm{\r}_{\!M}^{C}\big(g\big)=\bm{\r}_{\!M}^{C}\left(\bm{\CP}_{\!\!\O}(g)(\I_\O)\right)
=\bm{\CG\!\ell}_{\!\!M}(g)\left( \bm{\r}_{\!M}^{\O}(\I_\O)\right)$
for every  morphism $C\xrightarrow{g}  \O$ of ccdg-coalgebras.
Explicitly, we have
\eqnalign{univnat}{
\bm{\r}_{\!M}^{C}\big(g\big)
&=\check{\mp}\Big(\check{\mq}\big(\bm{\r}_{\!M}^{\O}(\I_\O)\big)\circ  (g\otimes \I_M)\Big)
\quad\Longleftrightarrow \quad
\check{\mq}\Big(\bm{\r}_{\!M}^{C}\big(g\big)\Big) = \check{\mq}\big(\bm{\r}_{\!M}^{\O}(\I_\O)\big)\circ  (g\otimes \I_M),
}
where $\check{\mp}$ and $\check{\mq}$ are defined in Lemma \ref{cbasicl}.
\end{remark}

\begin{lemma}
A linear representation 
$\bm{\r}_{\!M}: \bm{\CP}_{\!\!\O}\Rightarrow\bm{\CG\!\ell}_{\!\!M}$
of $\bm{\CP}_{\!\!\O}$ induces a natural transformation
$[\bm{\r}]_{\!M}: \bm{\mP}_{\!\O}\Rightarrow\bm{\mG\!\ell}_{\!\!M}
:\mathring{\mathit{ho}\category{ccdgC}}(\Bbbk)\rightsquigarrow \category{Grp}$,
whose component $[\bm{\r}]_{\!M}^{C}$ at each ccdg-coalgebra
$C$ is the  homomorphism  $[\bm{\r}]_{\!M}^{C}:\bm{\mP}_{\!\O}(C)\rightarrow \bm{\mG\!\ell}_{\!\!M}(C)$ of groups 
defined by 
$[\bm{\r}]_{\!M}^{C}([g])=\left[\bm{\r}_{\!M}^{C}(g)\right] \in \operatorname{H_0Aut}_{\cp_C\!}(C\otimes M)$ 
for all  $[g]\in \HOM_{\hccdgc}(C, \O)$,
where $g\in \HOM_{\ccdgc}(C, \O)$ is an arbitrary representative of $[g]$.
\end{lemma}

\begin{proof}
Let $g \sim \tilde g \in \HOM_{\ccdgc}(C,\O)$. 
Then there is a homotopy pair $\big(g(t),\chi(t)\big)$ on $\HOM_{\ccdgc}(C,\O)$ such
that we have a family $g(t) = g + \rd_{C\!,\O} \int^t_0 \chi(s) \mathit{ds}$ of morphism of ccdg-coalgebras 
satisfying $g(0) = g$ and $g(1)=\tilde g$.
From \eq{univnat} it follows that 
$\bm{\r}_{\!M}^{C}(\tilde g) \sim \bm{\r}_{\!M}^{C}(g)  \in \operatorname{Z_0Aut}_{\cp_C\!}(C\otimes M)$,
since both $\check{\mp}$ and $\check{\mq}$ are chain maps 
and $\bm{\r}_{\!M}^{\O}(\I_\O) \in \operatorname{Z_0Aut}_{\cp_\O\!}(\O\otimes M)$.
Therefore we have 
$\big[\bm{\r}_{\!M}^{C}(\tilde g)\big] 
=\big[\bm{\r}_{\!M}^{C}(g)\big]  \in \operatorname{H_0Aut}_{\cp_C\!}(C\otimes M)$,
so that $[\bm{\r}]_{\!M}^{C}:\bm{\mP}_{\!\O}(C)\rightarrow \bm{\mG\!\ell}_{\!\!M}(C)$  
is well-defined homomorphism of groups for every $C$.
The naturalness  of $[\bm{\r}]_{\!M}$, i.e.,
for every $[f] \in \HOM_{\hccdgc}(C, C^\pr)$ we have 
$[\bm{\r}]_M^{C}\circ\bm{\mP}_{\!\!\O}([f]) =\bm{\mG\!\ell}_{\!\!M}([f])\circ [\bm{\r}]_M^{C^\pr}$,
follows from the naturalness  \eq{repnatural} of $\bm{\r}_{\!M}$ and by the definitions of 
$[\bm{\r}]_{\!M}^{C}$,  $\bm{\mP}_{\!\!\O}([f])$ and $\bm{\mG\!\ell}_{\!\!M}([f])$.
\qed
\end{proof}

\begin{definition}
A linear representation of the presheaf of groups $\bm{\mP}_{\!\O}$ on
the homotopy category $\mathit{ho}\category{ccdgC}(\Bbbk)$ is
a pair $\big(M, [\bm{\r}]_{\!M}\big)$ of chain complex $M$ and 
a natural transformation 
${[\bm{\r}]_{\!M}: \bm{\mP}_{\!\O}\Rightarrow\bm{\mG\!\ell}_{\!\!M}
:\mathring{\mathit{ho}\category{ccdgC}}(\Bbbk)\rightsquigarrow \category{Grp}}$,
which is
{\bf induced} from a linear representation 
$\bm{\r}_{\!M}: \bm{\CP}_{\!\!\O}\Rightarrow\bm{\CG\!\ell}_{\!\!M}
:\mathring{\category{ccdgC}}(\Bbbk)\rightsquigarrow \category{Grp}$
of the presheaf of groups $\bm{\CP}_{\!\!\O}$ on ${\category{ccdgC}}(\Bbbk)$. 
\end{definition}

\begin{remark}
Despite of the above definition we will  work with  linear representations of 
$\bm{\CP}_{\!\O}$ rather than those of $\bm{\mP}_{\!\O}$.
The linear representations   of $\bm{\CP}_{\!\!\O}$ form a dg-tensor category
$\gdcat{Rep}(\bm{\CP}_{\!\!\O})$.
 Working with the dg-tensor category of linear representations of 
$\bm{\CP}_{\!\!\O}$  will be a crucial step for  Tannakian reconstructions 
of both $\bm{\CP}_{\!\!\O}$ and $\bm{\mP}_{\O}$
in the next section.
We may regard $\gdcat{Rep}(\bm{\CP}_{\!\!\O})$ as "the  dg-tensor category"
of linear representations of $\bm{\mP}_{\!\O}$, where the category of linear representations of $\bm{\mP}_{\!\O}$
can be defined as the homotopy category of $\gdcat{Rep}(\bm{\CP}_{\!\!\O})$. We will not elaborate 
this point as we will never use it.
\end{remark}

Here are two basic examples of linear representations of $\bm{\CP}_{\!\!\O}$.

\begin{example}[The trivial representation] \label{trivalrepsh}
The ground field $\Bbbk$ as a chain complex $\Bbbk=(\Bbbk, 0)$ with zero differential
defines the trivial representation  $\big(\Bbbk, \bm{\r}_{\!\Bbbk}\big)$, where the component  
$\bm{\r}^C_{\!\Bbbk}$ 
of $\bm{\r}_{\!\Bbbk}:\bm{\CP}_{\!\!\O}\Rightarrow \bm{\CG\!\ell}_{\!\!\Bbbk}$
at every ccdg-coalgebra $C$ is the 
homomorphism $\bm{\r}^C_{\!\Bbbk}:\bm{\CP}_{\!\!\O}(C)\Rightarrow\bm{\CG\!\ell}_{\!\Bbbk}(C)$ of groups
defined by, $\forall g \in \HOM_{\ccdgc}(C,\O)$,
\eqn{trivialrep}{
\bm{\r}^C_{\!\Bbbk}(g) := \I_C \otimes \I_\Bbbk: C\otimes \Bbbk\rightarrow C\otimes \Bbbk
.
}
\end{example}

\begin{example}[The regular representation] \label{regularrepsh}
Associated to the ccdg-Hopf algebra $\O$ as a chain complex
we have the regular representation  $\big(\O, \bm{\r}_{\!\O}\big)$, where the component  
$\bm{\r}^C_{\!\O}$ 
of $\bm{\r}_{\!\O}:\bm{\CP}_{\!\!\O}\Rightarrow \bm{\CG\!\ell}_{\!\!\O}$
at every ccdg-coalgebra $C$ is the 
homomorphism $\bm{\r}^C_{\!\O}:\bm{\CP}_{\!\!\O}(\O)\Rightarrow\bm{\CG\!\ell}_{\!\!\O}$ of groups
defined by, $\forall g \in \HOM_{\ccdgc}(C,\O)$,
\eqnalign{regularrep}{
\bm{\r}^C_{\!\O}(g)
=
&
(\I_C\otimes m_\O)\circ (\I_C\otimes g\otimes \I_\O)\circ (\cp_C\otimes \I_\O)
=\check{\mp}\big(m_\O\circ(g\otimes\I_\O)\big)
\cr
&
\xymatrixcolsep{3pc}
\xymatrix{
C\otimes \O \ar[r]^-{\cp_C\otimes \I_\O} & C\otimes C \otimes \O \ar[r]^-{\I_C\otimes g\otimes \I_\O} 
& C\otimes \O\otimes \O \ar[r]^-{\I_C\otimes m_\O}& C\otimes \O
}.
}
We can check that $\big(\O, \bm{\r}_{\!\O}\big)$ is a linear representation as follows
\begin{itemize}
\item 
We have $\rd_{C\otimes M,C\otimes M}\bm{\r}^C_{\!\O}(g)
= \check{\mp}\big(m_\O\circ(\rd_{C\!,\O}g\otimes\I_\O)\big)=0$
for  all $g \in \HOM_{\ccdgc}(C,\O)$,
since both $\check{\mp}$ and $m_\O$ are chain maps;

\item
We have $\bm{\r}^C_{\!\O}(u_\O\circ \ep_C)
=  (\I_C\otimes m_\O)\circ \big(\I_C\otimes( u_\O\circ \ep_C) \otimes \I_\O\big)\circ (\cp_C\otimes \I_\O)
=\I_{C\otimes \O}
.
$

\item
For all $g_1,g_2 \in \HOM_\ccdgc(C,\O)$:
\eqalign{
\bm{\r}^C_{\!\O}\big(g_1 &\star_{C\!,\O} g_2\big)
:=
(\I_C\otimes m_\O)\circ \Big(\I_C\otimes \big(m_\O\circ (g_1\otimes g_2)
\circ \cp_C\big)\otimes \I_\O\Big)\circ (\cp_C\otimes \I_\O)
\cr
=
&
(\I_C\otimes m_\O)\circ (\I_C\otimes g_1\otimes \I_\O)\circ (\cp_C\otimes \I_\O)
\circ (\I_C\otimes m_\O)\circ (\I_C\otimes g_2\otimes \I_\O)\circ (\cp_C\otimes \I_\O)
\cr
=
&\bm{\r}^C_{\!\O}(g_1) \circ \bm{\r}^C_{\!\O}(g_2).
}
The 2nd equality is due to coassociativity of $\cp_C$ and the associativity of $m_\O$.
\qed
\end{itemize}

\end{example}

\begin{definition}[Lemma]
\label{repofpshgdcat}
Linear representations
of the representable presheaf  of groups $\bm{\CP}_{\!\!\O}$ 
form a dg-tensor category $\gdcat{{Rep}}\big(\bm{\CP}_{\!\!\O}\big)$
defined as follows.

\begin{enumerate}[label=({\alph*})]
\item
An object is a linear presentation $\big(M,\bm{\r}_{\!M}\big)$ of 
$\bm{\CP}_{\!\!\O}$.

\item
A morphism $\p: \big(M,\bm{\r}_{\!M}\big)\rightarrow \big(M^\pr,\bm{\r}_{\!M^\pr}\big)$ 
of linear representations  of
$\bm{\CP}_{\!\!\O}$ is a linear map $\p:M \rightarrow M^\pr$
making the following diagram commutative
for every ccdg-coalgebra $C$
and every $g \in \bm{\CP}_{\!\!\O}(C)$  
$$
\xymatrixcolsep{3pc}
\xymatrix{
\ar[d]_-{\bm{\r}_{\!M}^{C}(g)}
C\otimes M \ar[r]^-{\I_C\otimes \p} & C\otimes M^\pr 
\ar[d]^-{\bm{\r}_{\!M^\pr}^{C}(g)}
\cr
C\otimes M \ar[r]^-{\I_C\otimes \p} & C\otimes M^\pr
,}
\quad\hbox{i.e.,}\quad
(\I_C\otimes \p)\circ\bm{\r}_{\!M}^{C}(g)
=\bm{\r}_{\!M^\pr}^{C}(g)\circ (\I_C\otimes \p)
,
$$
where  $\bm{\r}_{\!M}^C: \bm{\CP}_{\!\!\O}(C)\rightarrow \bm{\CG\!\ell}_{\!\!M}(C)$
is the  component of the natural transformation $\bm{\r}_{\!M}$ at $C$.

\item 
The differential of a morphism 
$\p: \big(M,\bm{\r}_{\!M}\big)\rightarrow \big(M^\pr,\bm{\r}_{\!M^\pr}\big)$ of linear representations 
is the morphism  
$\rd_{M\!,M^\pr}\p: \big(M,\bm{\r}_{\!M}\big)\rightarrow \big(M^\pr,\bm{\r}_{\!M^\pr}\big)$ of linear representations.

\item The tensor product $\big(M,\bm{\r}_{\!M}\big)\bm{\otimes} \big(M^\pr,\bm{\r}_{\!M^\pr}\big)$ of two objects
is the linear representation $(M\otimes M^\pr, \bm{\r}_{\!M\otimes M^\pr})$, where
$M\otimes M^\pr=(M\otimes M^\pr, \rd_{M\otimes M^\pr})$ is the tensor product of chain complexes
and  $ \bm{\r}_{\!M\otimes M^\pr} : \bm{\CP}_{\!\!\O}\Rightarrow\bm{\CG\!\ell}_{\!\!M\otimes M^\pr}$ is the natural transformation
whose component $\bm{\r}_{\!M\otimes M^\pr}^C:\bm{\CP}_{\!\!\O}(C)\rightarrow\bm{\CG\!\ell}_{\!\!M\otimes M^\pr}(C)$ 
at every ccdg-coalgebra $C$ is the group homomorphism
defined by, $\forall g\in \HOM_{\ccdgc}(C,\O)$,
\eqalign{
\bm{\r}_{\!M\otimes M^\pr}^C(g)
:=& \bm{\r}^C_{\!M}(g)\otimes_{\cp_C}\bm{\r}^C_{\!M^\pr}(g)
.
}
The unit object for the tensor product is the trivial representation 
$(\Bbbk, \bm{\r}_{\!\Bbbk})$ in Example \ref{trivalrepsh}.
\end{enumerate}
\end{definition}

\begin{remark}
The explicit form of $\bm{\r}_{\!M\otimes M^\pr}^C(g):= \bm{\r}^C_{\!M}(g)\otimes_{\cp_C}\bm{\r}^C_{\!M^\pr}(g)$ is
\eqalign{
&\bm{\r}_{\!M\otimes M^\pr}^C(g)
=\big(\bm{\r}^C_{\!M}(g) \otimes \check{\mq}(\bm{\r}^C_{\!M^\pr}(g))\big)\circ (\I_C\otimes \t\otimes \I_{M^\pr})
\circ (\cp_C\otimes \I_M\otimes \I_{M^\pr})
\cr
&\qquad
=(\I_{C}\otimes\I_{ M}\otimes \imath_M\circ (\ep_{C}\otimes \I_{M^\pr}) )\circ
\big(\bm{\r}^C_{M}(g) \otimes \bm{\r}^C_{M^\pr}(g)\big)
\circ (\I_C\otimes \t\otimes \I_{M^\pr})
\circ (\cp_C\otimes \I_M\otimes \I_{M^\pr})
.
}
Equivalently, $\bm{\r}_{\!M\otimes M^\pr}^C(g)$ is determined by the following equality:
\eqnalign{tensor of natural transformations}{
\check{\mq}\big(\bm{\r}_{\!M\otimes M^\pr}^C(g)\big)=\big(\check{\mq}(\bm{\r}_{\!M}^C(g))\otimes\check{\mq}(\bm{\r}_{\!M^\pr}^C(g))\big)\circ(\I_C\otimes\tau\otimes \I_{M^\pr})\circ(\cp_C\otimes\I_{M\otimes M^\pr}).
}
\end{remark}

\begin{proof}
It is trivial to check that $\rd_{M\!,M^\pr}\p$ 
is a linear representation whenever $\p$ is a linear representation. Then it becomes
obvious that $\gdcat{{Rep}}\big(\bm{\CP}_{\!\!\O}\big)$ is a dg-category. It is also trivial to check that the tensor product
and the unit object in $(d)$ endow 
$\gdcat{{Rep}}\big(\bm{\CP}_{\!\!\O}\big)$ with a structure of dg-tensor category. 
\qed
\end{proof}

\subsection{An isomorphism with the dg-tensor category of left dg-modules}

A left dg-module over  the ccdg-Hopf algebra $\O$
is a tuple $(M, {\g}_M)$, where $M=(M, \rd_M)$ is a chain complex
and ${\g}_M: \O\otimes M \rightarrow M$ is a chain map making the diagrams commutative
$$
\xymatrixcolsep{3pc}
\xymatrix{
\O\otimes M \ar[r]^{{\g}_M} & M 
\cr
& \Bbbk\otimes M\ar[ul]^-{u_\O\otimes \I_M}\ar[u]_-{\imath_M}
},
\qquad\quad
\xymatrix{
\ar[d]_-{\I_\O\otimes {\g}_M}
\O\otimes  \O\otimes M \ar[r]^-{m_\O\otimes \I_M} & \O\otimes M
\ar[d]^-{{\g}_M}
\cr
\O\otimes  M \ar[r]^-{{\g}_M} &  M
}
$$
That is 
$$
\g_M\circ \rd_{\O\otimes M} = \rd_M\circ \g_M
,\qquad
\begin{cases}
{\g}_M\circ (u_\O\otimes \I_M)=\imath_M
,\cr
{\g}_M\circ (\I_\O\otimes {\g}_M) = {\g}_M\circ (m_\O\otimes \I_M).
\end{cases}
$$
A morphism $\xymatrix{(M, {\g}_M)\ar[r]^-{\p}& (M^\pr, {\g}_{M^\pr})}$ of left dg-modules over $\O$
is a linear map $\p:M\rightarrow M^\pr$ making the following diagram commutes
$$
\xymatrix{
\ar[d]_{\I_\O\otimes \p}
\O\otimes M \ar[r]^-{{\g}_M} & M \ar[d]^\p
\cr
\O\otimes M^\pr \ar[r]^-{{\g}_{M^\pr}} & M^\pr
},\quad\hbox{i.e.,}\quad  \p\circ {\g}_M ={\g}_{M^\pr}\circ (\I_\O\otimes \p).
$$
It is trivial to check that 
$\rd_{M,M^\pr}\p: (M, {\g}_M)\rightarrow (M^\pr, {\g}_{M^\pr})$ 
is a morphism of left dg-modules over $\O$ whenever $\p$ is so,
and we have $\rd_{M,M^\pr}\circ\rd_{M,M^\pr}=0$. 
For every consecutive morphism 
$\p^\pr: (M^\pr, {\g}_{M^\pr})\rightarrow (M^\ppr, {\g}_{M^\ppr})$ of left dg-modules over $\O$
we also have 
$\rd_{M,M^\ppr}\big(\p^\pr\circ \p \big) =\rd_{M^\pr,M^\ppr}\p^\pr\circ \p +(-1)^{|\p^\pr|}\p^\pr\circ\rd_{M,M^\pr}\p$.
Therefore,  we have a dg-category $\gdcat{dgMod}_L(\O)$ of left dg-modules over $\O$.

The tensor product $(M, {\g}_M)\otimes_{\cp_\O\!} (M^\pr, {\g}_{M^\pr})$ 
of  left dg-modules $(M, {\g}_M)$ and $(M^\pr, {\g}_{M^\pr})$ over $\O$
is the left dg-module  $\big(M\otimes M^\pr, {\g}_{M\otimes_{\cp_\O\!} M^\pr}\big)$  over $\O$, 
where $M\otimes M^\pr$ is the chain complex with the differential $\rd_{M\otimes M^\pr}$ and
\eqnalign{modtensor}{
&{\g}_{M\otimes_{\cp_\O}\! M^\pr}
:=  (\g_M\otimes \g_{M^\pr})\circ (\I_\O\otimes\t\otimes \I_{M^\pr})\circ(\cp_\O\otimes \I_M\otimes \I_{M^\pr})
\cr
&
:\xymatrixcolsep{2.5pc}
\xymatrix{
\O\!\otimes\! M\!\otimes\! M^\pr \ar[r]^-{\cp_\O\!\otimes\! \I_{M\!\otimes\! M^\pr}}& \O\!\otimes\!  \O\!\otimes\! M\!\otimes\! M^\pr
\ar[r]^-{\I_\O\!\otimes\! \t\!\otimes\! \I_{M^\pr}}&\O\!\otimes\! M\!\otimes\! \O\!\otimes\! M^\pr 
\ar[r]^-{{\g}_{M}\!\otimes\! {\g}_{M^\pr}}& M\!\otimes\! M^\pr
.
}
}
The ground field $\Bbbk$ has a structure $(\Bbbk,\g_\Bbbk)$
of left dg-module $\xymatrix{\O\otimes \Bbbk \ar[r]^-{{\g}_\Bbbk}& \Bbbk}$
over $\O$, where ${\g}_\Bbbk :=m_\Bbbk\circ( \ep_\O\otimes \I_\Bbbk)$.
We can check that the dg-category of left dg-modules over $\O$ 
is a dg-tensor category $\left( \gdcat{dgMod}_L(\O), \otimes _{\cp_\O\!}, (\Bbbk,\g_\Bbbk)\right)$
as follows:
\begin{itemize}
\item
For left dg-modules 
$\big(M,\g_M\big)$,  $\big(M^\pr,\g_{M^\pr}\big)$  and $\big(M^{\pr\pr},\g_{M^{\pr\pr}}\big)$ over $\O$,
the isomorphism $(M\otimes M^\pr)\otimes M^{\pr\pr}\cong M\otimes(M^\pr\otimes M^{\pr\pr})$ 
of the underlying chain complexes induces an isomorphism
\[
\Big(\big(M,\g_M\big)\otimes_{\cp_\O}\big(M^\pr,\g_{M^\pr}\big)\Big)\otimes_{\cp_\O}\big(M^{\pr\pr},\g_{M^{\pr\pr}}\big)
\cong
\big(M,\g_M\big)\otimes_{\cp_\O}\Big(\big(M^\pr,\g_{M^\pr}\big)\otimes_{\cp_\O}\big(M^{\pr\pr},\g_{M^{\pr\pr}}\big)\Big)
\]
of  left dg-modules over $\O$, since $\cp_\O$ is coassociative.
 \item
For every left dg-module $\big(M,\g_M\big)$ over $\O$, 
the isomorphisms $M\otimes\Bbbk\cong M\cong \Bbbk\otimes M$ of the underlying chain complexes induce  isomorphisms
$$
\big(M,{\g}_M\big)\otimes_{\cp_\O\!}  \big(\Bbbk, {\g}_\Bbbk\big)\cong \big(M,{\g}_M\big) \cong  \big(\Bbbk, {\g}_\Bbbk\big) \otimes_{\cp_\O\!}  \big(M,{\g}_M\big)
$$
of left dg-modules over $\O$, due to the counit axiom 
$\imath_\O\circ(\ep_\O\otimes \I_\O)\circ\cp_\O=\jmath_\O\circ(\I_\O\otimes\ep_\O)\circ\cp_\O=\I_\O$.
\end{itemize}

%
%

\begin{theorem}\label{repmod}
The dg-category $\gdcat{Rep}\big(\bm{\CP}_{\!\!\O}\big)$ of linear representations
of  $\bm{\CP}_{\!\!\O}$   is isomorphic to the dg-category
of $\gdcat{dgMod}_L(\O)$ of left dg-modules over $\O$
as dg-tensor categories.
Explicitly, 
we have an isomorphism of dg-tensor categories 
$$\xymatrix{{\functor{X}}:\gdcat{Rep}\big(\bm{\CP}_{\!\!\O}\big)\ar@{~>}@/^/[r]
& \ar@{~>}@/^/[l] \gdcat{dgMod}_L(\O): {\functor{Y}}}$$
defined as follows.
\begin{itemize}
\item $\functor{X}$ sends each representation $\big(M,\bm{\r}_{\!M}\big)$ to the left dg-module 
$\big(M,\widebreve\g_{\!M}\big)$, where
\eqnalign{funcx}{
\widebreve\g_{\!M}:=&\check{\mq}\big(\bm{\r}_{\!M}^\O(\I_\O)\big)
\cr
=&\jmath_M \circ (\ep_\O\otimes \I_M)\circ \bm{\r}_{\!M}^\O(\I_{\O}):\O\otimes M\to M,
}
and each morphism $\p:\big(M,\bm{\r}_{\!M}\big)\rightarrow \big(M^\pr,\bm{\r}_{\!M^\pr}\big)$ of representations
to the morphism 
$\p:\big(M,\widebreve{\g}_{\!M}\big)\rightarrow \big(M^\pr,\widebreve{\g}_{\!M^\pr}\big)$ of left dg-modules.

\item $\functor{Y}$ sends each left dg-module 
$\big(M,\g_{\!M}\big)$ to the representation $\big(M,\widebreve{\bm{\r}}_{\!M}\big)$,
where the component  $\widebreve{\bm{\r}}^C_{\!M}$ of $\widebreve{\bm{\r}}_{\!M}$ at a ccdg-coalgebra
$C$ is defined by, $\forall g \in \HOM_{\ccdgc}(C, \O)$,
\eqnalign{funcy}{
\widebreve{\bm{\r}}_{\!M}^{C}(g):=
&\check{\mp}\Big(\g_{\!M}\circ(g\otimes \I_M)\Big)
\cr
=
&(\I_C\otimes {\g}_{\!M})\circ  (\I_C\otimes g\otimes \I_M)\circ (\cp_{\!C}\otimes \I_M)
:C\otimes M\rightarrow C\otimes M,
}
and each morphism $\p:\big(M,\g_{\!M}\big)\rightarrow \big(M^\pr,\g_{\!M^\pr}\big)$ of left dg-modules
to the morphism  $\p:\big(M,\widebreve{\bm{\r}}_{\!M}\big)\rightarrow \big(M^\pr,\widebreve{\bm{\r}}_{\!M^\pr}\big)$ 
of representations.
\end{itemize}

\end{theorem}

\begin{proof}
We need to check that both $\functor{X}$
and $\functor{Y}$ are dg-tensor functors and show that they are inverse to each other.

1. We check that  
$\big(M,\widebreve\g_{\!M}\big)=\functor{X}\big(M,\bm{\r}_{\!M}\big)$ is a left dg-module over $\O$ as follows.

From \eq{univnat}  in Remark  \ref{repyoneda} and the definition \eq{funcx},
we have the following relation for every morphism $g:C\to \O$ of ccdg-coalgebras:
\eqnalign{Yoneda lemma implies}{
\check{\mq}\Big(\bm{\r}_{\!M}^{C}\big(g\big)\Big)=\widebreve\g_{\!M}\circ(g\otimes \I_M):C\otimes M\to M.
}
The component $\bm{\r}_{\!M}^{C}$ of  $\bm{\r}_{\!M}$ at  every ccdg-coalgebra $C$, by definition,
is a morphism
$\bm{\r}_{\!M}^{C}:\bm{\CP}_{\!\!\O}(C)\rightarrow \bm{\CG\!\ell}_{\!\!M}(C)$ of groups, i.e.,
for every pair of morphisms $g_1,g_2:C\to \O$ of ccdg-coalgebras, we have
\eqnalign{group homomorphism condition}{
\bm{\r}_{\!M}^{C}\big(u_\O\circ\ep_C\big)=\I_{C\otimes M},
\qquad
\bm{\r}_{\!M}^{C}\big(g_1\star_{C\!,\O}g_2\big)=\bm{\r}_{\!M}^{C}(g_1)\circ \bm{\r}_{\!M}^{C}(g_2).
}
\begin{itemize}
\item
Applying $\check{\mq}$ on the $1$st equality of \eq{group homomorphism condition}
and using  \eq{Yoneda lemma implies}, we have
\eqnalign{unital action actiom}{
\widebreve\g_{\!M}\circ(u_\O\otimes\I_M)\circ(\ep_C\otimes \I_M)=\ep_C\otimes\I_M.
}
By putting $C=\Bbbk^\vee$, we obtain that  $\widebreve\g_M\circ(u_\O\otimes\I_M)=\imath_M$.

\item
Applying $\check{\mq}$ on the $2$nd equality of \eq{group homomorphism condition}, we have
\eqnalign{showing the action axiom}{
\widebreve\g_{\!M}\circ(m_\O\otimes \I_M)\circ &(g_1\otimes g_2\otimes \I_M)\circ(\cp_C\otimes\I_M)
\cr
&
=\widebreve\g_{\!M}\circ(\I_\O\otimes\widebreve\g_M)\circ(g_1\otimes g_2\otimes \I_M)\circ(\cp_C\otimes\I_M).
}
Consider the ccdg-coalgebra $\O\otimes \O$ and the projection maps $\pi_1,\pi_2:\O\otimes \O\to \O$
\[
\pi_1:=
\xymatrix{
\O\otimes \O\ar[r]^-{\I_\O\otimes\ep_\O}&\O\otimes \Bbbk\ar[r]^-{\jmath_\O}&\O
,
}
\quad\quad
\pi_2:=
\xymatrix{
\O\otimes \O\ar[r]^-{\ep_\O\otimes\I_\O}&\Bbbk\otimes \O\ar[r]^-{\imath_\O}&\O
,
}
\]
which are morphisms of ccdg-coalgebras. 
We can check that $(\pi_1\otimes\pi_2)\circ\cp_{\O\otimes\O}=\I_{\O\otimes\O}$ from an elementary calculation. 
By substituting  $C =\O\otimes \O$,  $g_1=\pi_1$ and $g_2=\pi_2$ in \eq{showing the action axiom} we obtain that
$\widebreve\g_M\circ(m_\O\otimes \I_M)=\widebreve\g_M\circ(\I_\O\otimes\widebreve\g_M)$.

\item
Finally we check that  $\widebreve\g_{\!M}:\O\otimes M\to M$ is a chain map:
\eqalign{
\rd_{\O\otimes M\!,M}\widebreve\g_{\!M}=
\rd_{\O\otimes M\!,M}\check{\mq}\big(\bm{\r}_{\!M}^\O(\I_\O)\big)
=\check{\mq}\big(\rd_{\O\otimes M,\O\otimes M}\bm{\r}_{\!M}^\O(\I_\O)\big)
=0,
}
where we have used the facts that $\check{\mq}$ is a chain map defined in Lemma \ref{cbasicl}
and  $\bm{\r}_{\!M}^{\O}(\I_\O) \in \operatorname{Z_0Aut}_{\cp_\O\!}(\O\otimes M)$.

\end{itemize}

2. We show that $\functor{X}$ is a dg-tensor functor.
Given a morphism $\p:\big(M,\bm{\r}_{\!M}\big)\to \big(M',\bm{\r}_{M'}\big)$ of representations, $\functor{X}(\p)
=\p:\big(M,\widebreve{\g}_{\!M}\big)\rightarrow \big(M^\pr,\widebreve{\g}_{\!M^\pr}\big)$ 
is a morphism of left dg-modules over $\O$,
since the following diagram commutes
\[
\xymatrixrowsep{1.5pc}
\xymatrixcolsep{3.5pc}
\xymatrix{
\O\otimes M \ar[d]^-{\I_\O\otimes\p} \ar[r]^-{\bm{\r}_{\!M}^\O(\I_\O)}&
\O\otimes M \ar[d]^-{\I_\O\otimes\p} \ar[r]^-{\ep_\O\otimes\I_M}&
\Bbbk\otimes M \ar[d]^-{\I_\Bbbk\otimes\p} \ar[r]^-{\imath_M}&
M \ar[d]^-{\p}\\
\O\otimes M^\pr \ar[r]^-{\bm{\r}_{M^\pr}^\O(\I_\O)}&
\O\otimes M^\pr \ar[r]^-{\ep_\O\otimes\I_{M^\pr}}&
\Bbbk\otimes M^\pr\ar[r]^-{\imath_{M^\pr}}&
M'
},
\]
where the very left square commutes since $\p$ is a morphism of representations
and the commutativity of the other squares are obvious---
note that the horizontal compositions are exactly $\widebreve\g_M$ and $\widebreve\g_{M'}$.
It is obvious that 
$\functor{X}(\rd_{M,M^\pr}\p)=\rd_{M,M^\pr}\p=\rd_{M,M^\pr}\big(\functor{X}(\p)\big)$.
Therefore $\functor{X}$ is a dg-functor.
The tensor property of $\functor{X}$ is checked as follows:
\begin{itemize}
\item
From Example \ref{trivalrepsh} we have $\functor{X}\big(\Bbbk,\bm{\r}_\Bbbk\big)=\big(\Bbbk,\g_\Bbbk\big)$.

\item
For two representations $\big(M,\bm{\r}_{\!M}\big)$ and $\big(M',\bm{\r}_{M'}\big)$ of $\bm{\CP}_{\!\!\O}$, we have
\eqalign{
\g_{\functor{X}(M,\bm{\r}_{\!M})\otimes_{\cp_\O}\functor{X}(M^\pr,\bm{\r}_{M^\pr})}
&=
\big(
\g_{\functor{X}(M,\bm{\r}_{\!M})}\otimes\g_{\functor{X}(M^\pr,\g_{M^\pr})}
\big)
\circ(\I_\O\otimes\tau\otimes\I_{M^\pr})\circ(\cp_\O\otimes \I_{M\otimes M^\pr})\\
&=
\g_{\functor{X}\left(
(M,\bm{\r}_{\!M})\otimes(M^\pr,\bm{\r}_{M^\pr})
\right)}
.
}
The $1$st equality is from \eq{modtensor}, and the $2$nd equality is from \eq{tensor of natural transformations}.
We conclude that 
$\functor{X}\Big(\big(M,\bm{\r}_{\!M}\big)\otimes\big(M^\pr,\bm{\r}_{M^\pr}\big)\Big)
=\functor{X}\big(M,\bm{\r}_{\!M}\big)\otimes_{\cp_\O}\! \functor{X}\big(M^\pr,\bm{\r}_{M^\pr}\big).$
\end{itemize}

3. 
We show  that 
$\big(M,\widebreve{\bm{\r}}_{\!M}\big)=\functor{Y}\big(M,\g_{\!M}\big)$ 
is a representation of $\bm{\CP}_{\!\!\O}$ as follows.
We first show that $\widebreve{\bm{\r}}_{\!M}^{C}:\bm{\CP}_{\!\!\O}(C)\rightarrow \bm{\CG\!\ell}_{\!\!M}(C)$ 
is a homomorphism of groups for every  $C$:
\begin{itemize}

\item
We have $\bm{\r}^C_{\!M}(u_\O\circ \ep_C)
=  (\I_C\otimes \g_{\!M})\circ \big(\I_C\otimes( u_\O\circ \ep_C) \otimes \I_M\big)\circ (\cp_C\otimes \I_M)
=\I_{C\otimes M}
$, where we have used the counity of $\cp_C$ and  the property ${\g}_M\circ (u_\O\otimes \I_M)=\imath_M$.

\item
For all $g_1,g_2 \in \HOM_\ccdgc(C,\O)$:
\eqalign{
\bm{\r}^C_{\!M}\big(g_1 &\star_{C\!,\O} g_2\big)
:=
(\I_C\otimes \g_{\!M})\circ \Big(\I_C\otimes \big(m_{\O}\circ (g_1\otimes g_2)\circ \cp_C\big)\otimes \I_M\Big)
\circ (\cp_C\otimes \I_M)
\cr
=
&
(\I_C\otimes \g_M)\circ (\I_C\otimes g_1\otimes \I_M)\circ (\cp_C\otimes \I_M)
\circ (\I_C\otimes \g_M)\circ (\I_C\otimes g_2\otimes \I_M)\circ (\cp_C\otimes \I_M)
\cr
=
&\bm{\r}^C_{\!M}(g_1) \circ \bm{\r}^C_{\!M}(g_2),
}
where
the $2$nd equality is due to coassociativity of $\cp_C$ and the property 
$ {\g}_M\circ (m_\O\otimes \I_M)={\g}_M\circ (\I_\O\otimes {\g}_M)$.

\item We have  $\rd_{C\otimes M,C\otimes M}\bm{\r}^C_{\!M}(g)
= \check{\mp}\big(\g_{\!M}\circ(\rd_{C\!,\O}g\otimes\I_\O)\big)=0$
for  all $g \in \HOM_{\ccdgc}(C,\O)$,
since both $\check{\mp}$ and $\g_{\!M}$ are chain maps.
\end{itemize}
Combining all the above, we conclude that
$\widebreve{\bm{\r}}_M^{C}\big(g\big)\in Z_0\Aut_{\cp_C}(C\otimes M)$ for all $g \in \HOM_{\ccdgc}(C,\O)$
and $\widebreve{\bm{\r}}_{\!M}^{C}$ is a homomorphism of groups.
We check the naturalness  of  $\widebreve{\bm{\r}}_{\!M}$ 
that for every morphism 
$f: C\rightarrow C^\pr$ of ccdg-coalgebras we have  
$\widebreve{\bm{\r}}_M^{C}\circ\bm{\CP}_{\!\!\O}(f) =\bm{\CG\!\ell}_{\!\!M}(f)\circ \widebreve{\bm{\r}}_M^{C^\pr}$
as follows: for all $g^\pr \in \HOM_{\ccdgc}(C^\pr, \O)$ we have
\eqalign{
\bm{\r}_{\!M}^{C}\left(\bm{\CP}_{\!\!\O}(f)(g^\pr)\right)
=&\widebreve{\bm{\r}}_{\!M}^{C}(g^\pr\circ f)
=
\check{\mp}\Big(\g_{\!M}\circ(g^\pr\circ f\otimes \I_M)\Big)
,\cr
\bm{\CG\!\ell}_{\!\!M}(f)\left( \bm{\r}_{\!M}^{C^\pr}(g^\pr)\right)
= &
\check{\mp}\Big(\check{\mq}\left( \bm{\r}_{\!M}^{C^\pr}(g^\pr)\right)\circ  (f\otimes \I_M)\Big)
=
\check{\mp}\Big(
\check{\mq}\left(\check{\mp}\Big(\g_{\!M}\circ(g^\pr\otimes \I_M)\Big)\right)\circ  (f\otimes \I_M)
\Big)
\cr
=
&
\check{\mp}\Big(\g_{\!M}\circ(g^\pr\otimes \I_M)\circ  (f\otimes \I_M)\Big)
=
\check{\mp}\Big(\g_{\!M}\circ(g^\pr\circ f\otimes \I_M)\Big)
,
}
where we have used  $\check\mq\circ \check\mp =\I_{\Hom(C^\pr\otimes M,N)}$.

4. We show that $\functor{Y}$ is a dg-tensor functor.
Given a morphism $\p:\big(M,\g_{\!M}\big)\to\big(M^\pr,\g_{\!M^\pr}\big)$ of left dg-modules over $\O$, 
$\functor{Y}(\p)=\p:\big(M,\widebreve{\bm{\r}}_{\!M}\big)\to\big(M^\pr,\widebreve{\bm{\r}}_{\!M^\pr}\big)$
is a morphism of representations, since the following diagram commutes  for every morphism 
$g:C\to \O$ of ccdg-coalgebras:
\[
\xymatrixcolsep{3.5pc}
\xymatrix{
C\otimes M \ar[r]^-{\cp_C\otimes \I_M} \ar[d]^-{\I_C\otimes\p}&
C\otimes C\otimes M \ar[r]^-{\I_C\otimes g\otimes \I_M} \ar[d]^-{\I_{C\otimes C}\otimes\p}&
C\otimes \O\otimes M \ar[r]^-{\I_C\otimes \g_M} \ar[d]^-{\I_{C\otimes\O}\otimes\p}&
C\otimes M \ar[d]^-{\I_C\otimes\p}\\
C\otimes M^\pr \ar[r]^-{\cp_C\otimes \I_{M^\pr}}&
C\otimes C\otimes M^\pr \ar[r]^-{\I_C\otimes g\otimes \I_{M^\pr}}&
C\otimes \O\otimes M^\pr \ar[r]^-{\I_C\otimes \g_{M^\pr}}&
C\otimes M^\pr.
}
\]
It is obvious that $\functor{Y}(\rd_{M,M^\pr}\p)=\rd_{M,M^\pr}\p=\rd_{M,M^\pr}\big(\functor{Y}(\p)\big)$.
Therefore $\functor{Y}$ is a dg-functor.
The tensor property of $\functor{Y}$ is checked as follows:
\begin{itemize}
\item From Example \ref{trivalrepsh}, we have $\functor{Y}\big(\Bbbk,\g_\Bbbk\big)=\big(\Bbbk,\bm{\r}_{\!\Bbbk}\big)$.

\item
Let $\big(M,\g_M\big)$ and $\big(M',\g_{M'}\big)$ be left dg-modules over 
$\O$, and $g:C\to \O$ be a morphism of ccdg-coalgebras. Then by \eq{tensor of natural transformations}, we have
\eqalign{
\check{\mq}\left(
\bm{\r}^C_{\!\functor{Y}(M,\g_M)\otimes\functor{Y}(M^\pr,\g_{M^\pr})}(g)
\right)
&=
(\g_M\otimes \g_{M^\pr})\circ(\I_C\otimes \tau\otimes \I_{M^\pr})\circ\big(((g\otimes g)
\circ\cp_C)\otimes \I_{M\otimes M^\pr}\big)\\
&=(\g_M\otimes \g_{M^\pr})\circ(\I_C\otimes \tau\otimes \I_{M^\pr})
\circ(\cp_\O\otimes \I_{M\otimes M^\pr})\circ(g\otimes\I_{M\otimes M^\pr})\\
&=\g_{M\otimes_{\cp_\O}M^\pr}\circ(g\otimes\I_{M\otimes M^\pr})
=\check{\mq}\left(
\bm{\r}^C_{\!\functor{Y}(M\otimes M^\pr,\g_{M\otimes_{\cp_\O}M^\pr})}(g)
\right).
}
Therefore we have 
$\functor{Y}\left(\big(M,\g_M\big)\otimes_{\cp_\O}\big(M^\pr,\g_{M^\pr}\big)\right)
=\functor{Y}(M,\g_M)\otimes \functor{Y}(M^\pr,\g_{M^\pr}\big)$.
\end{itemize}

5. It is immediate from the constructions that $\functor{X}$ and $\functor{Y}$ are inverse to each other.
\qed
\end{proof}

\subsection{A ccdg-Hopf algebra versus the dg-category of its left  dg-modules}

Key properties of  the ccdg-Hopf algebra $\O$ are reflected in  the dg-category 
$\gdcat{dgMod}_L(\O)\cong \gdcat{Rep}(\bm{\CP}_{\!\!\O})$,
which phenomena will make our Tannakian reconstruction possible.

Remind that $(\O,m_\O)$ of both left and right dg-modules over $\O$.

The following lemma is due to the  ccdg-bialgebra structure of $\O$.

\begin{lemma}\label{modpr}
We have the following morphisms of left  dg-modules over $\O$:
\begin{enumerate}[label=({\alph*})]

\item  $\xymatrix{\big(\O,m_\O\big)\ar[r]^-{\cp_\O}& \big(\O\otimes \O, \g_{\O\otimes_{\!\cp_\O\!}\O} \big)}$
by  the coproduct $\cp_\O:\O\to\O\otimes\O$;

\item   $\xymatrix{(\O, m_\O)\ar[r]^-{\ep_\O} & (\Bbbk,\g_\Bbbk)}$ by the counit $\ep_\O:\O\rightarrow \Bbbk$; 

\item $\xymatrix{(\O\otimes M, m_\O\otimes \I_M)\ar[r]^-{\g_M}& (M, \g_M)}$
by the action $\g_M:\O\otimes M\to M$ of each left dg-module $(M,\g_M)$ over $\O$.

\end{enumerate}
\end{lemma}

\begin{proof}
\emph{(a)} follows from $\cp_\O$ being a morphism of dg-algebras:
$\g_{\O\otimes_{\cp_\O}\O}\circ(\I_\O\otimes\cp_\O)
=(m_\O\otimes m_\O)\circ(\I_\O\otimes\tau\otimes\I_\O)\circ(\cp_\O\otimes\cp_\O)=\cp_\O\circ m_\O$.
Here, we used
$\g_{\O\otimes_{\cp_\O}\O} =m_{\O\otimes \O}\circ (\cp_\O\otimes \I_{M\otimes M})$.
\emph{(b)} follows from $\ep_\O$ being a morphism of dg-algebras:
$\g_\Bbbk\circ(\I_\O\otimes \ep_\O)=m_\Bbbk\circ(\ep_\O\otimes \ep_\O)=\ep_\O\circ m_\O$.
Finally, \emph{(c)} follows from the module axiom that $\g_M$ satisfies:
$\g_M\circ(\I_\O\otimes \g_M)=\g_M\circ(m_\O\otimes \I_M)$.
\qed
\end{proof}

Since the antipode $\vs_\O:\O\rightarrow \O$ is an anti-homomorphism of dg-algebras,
we have a left dg-module  $(\O^*, \g_{\O^*})$ over $\O$,
where $\O^*=\O$ as a chain complex but  with the following alternative action:
$$
\g_{\O^*}:= m_\O\circ (\I_\O\otimes \vs_\O)\circ \t :\O\otimes \O^*\rightarrow \O^*
,\quad\hbox{i.e.,}\quad
y\otimes z \mapsto (-1)^{|y||z|}m_\O\big(z\otimes \vs_\O( y)\big).
$$
For every left dg-module $\big(M,\g_M\big)$ over $\O$,
we can associate another left dg-module $\big(M_*,\g_{M_*}\big)$ over $\O$ 
using the counit $\ep_\O:\O\rightarrow \Bbbk$,
where $M_*=M$ as a chain complex while the left action is given by
$\xymatrixcolsep{2.5pc}
\g_{M_*}:=\imath_M\circ (\ep_\O\otimes \I_M):
\xymatrix{\O\otimes M \ar[r]^-{\ep_\O\otimes \I_M} &\Bbbk \otimes M \ar[r]^-{\imath_M} & M}$.

The following lemma is due to  the antipode $\vs_\O$ of $\O$.

\begin{lemma} \label{Three module maps about O*}
We have following morphisms of left  dg-modules over $\O$:
\begin{enumerate}[label=({\alph*})]

\item 
$\xymatrix{(\O^*,\g_{\O^*})\ar[r]^-{\cp_\O}&(\O^*\otimes\O^*,\g_{\O^*\otimes_{\cp_\O}\O^*})}$ 
by  the coproduct $\cp_\O:\O\to\O\otimes\O$;

\item $\xymatrix{(\O^*,\g_{\O^*})\ar[r]^-{\ep_\O}&(\Bbbk,\g_\Bbbk)}$ by the counit $\ep_\O:\O\to \Bbbk$;

\item $\xymatrix{(\O^*\otimes M,\g_{\O^*\otimes_{\cp_\O}M}) \ar[r]^-{\g_M}&(M_*,\g_{M_*})}$ 
by the action $\g_M:\O\otimes M\to M$ of each left dg-module $(M,\g_M)$ over $\O$.
\end{enumerate}
\end{lemma}

\begin{proof}
We shall see that \emph{(a)} 
and \emph{(b)} follows from $\vs_\O$ being a morphism of ccdg-coalgebras
and $\emph{(c)}$ follows from the antipode axiom for $\vs_\O$.
\begin{enumerate}[label=$({\alph*})$,leftmargin=.8cm]
\item $\g_{\O^*\otimes_{\cp_\O}\O^*}\circ(\I_\O\otimes\cp_\O)=\cp_\O\circ\g_{\O^*}$:
We have
\eqalign{
\g_{\O^*\otimes_{\cp_\O}\O^*}\circ(\I_\O\otimes\cp_\O)
:=
&(m_\O\otimes m_\O)\circ(\I_\O\otimes\vs_\O\otimes\I_\O\otimes\vs_\O)
\circ(\tau\otimes\tau)\circ(\I_\O\otimes\tau\otimes\I_\O)\circ(\cp_\O\otimes\cp_\O)
\cr
=
&(m_\O\otimes m_\O)\circ\sigma\circ(\vs_\O\otimes\vs_\O\otimes\I_\O\otimes\I_\O)\circ(\cp_\O\otimes\cp_\O)
,\cr
\cp_\O\circ\g_{\O^*}
:=
&
\cp_\O\circ m_\O\circ(\I_\O\otimes\vs_\O)\circ\tau
\cr
=
&
(m_\O\otimes m_\O)\circ(\I_\O\otimes\tau\otimes\I_\O)\circ(\cp_\O\otimes\cp_\O)
\circ(\I_\O\otimes\vs_\O)\circ\tau
\cr
=
&
(m_\O\otimes m_\O)\circ\sigma\circ(\cp_\O\otimes\cp_\O)\circ(\vs_\O\otimes\I_\O),
}
where $\sigma:=(\t\otimes \t)\circ (\I_\O\otimes \t\otimes \I_\O):\O^{\otimes4}\to\O^{\otimes4}$.
From the property $(\vs_\O\otimes\vs_\O)\circ\cp_\O=\cp_\O\circ\vs_\O$, 
we obtain that $\g_{\O^*\otimes_{\cp_\O}\O^*}\circ(\I_\O\otimes\cp_\O)=\cp_\O\circ\g_{\O^*}$. 

\item $\ep_\O\circ\g_{\O^*}=\g_\Bbbk\circ(\I_\O\otimes\ep_\O)$:
We have
\eqalign{
\ep_\O\circ\g_{\O^*}=
&
\ep_\O\circ m_\O\circ(\I_\O\otimes\vs_\O)\circ\tau
=
m_\Bbbk\circ(\ep_\O\otimes\ep_\O)\circ(\I_\O\otimes\vs_\O)\circ\tau
=
m_\Bbbk\circ(\ep_\O\otimes\ep_\O)
\cr
=
&
\g_\Bbbk\circ(\I_\O\otimes\ep_\O),
}
where we used the property $\ep_\O=\ep_\O\circ\vs_\O$ and commutativity of 
$m_\Bbbk$.

\item $\g_M \circ\g_{\O^*\otimes_{\cp_\O}M}=\g_{M_*}\circ(\I_\Omega\otimes\g_M)$:
We have
\eqalign{
\g_M & \circ\g_{\O^*\otimes_{\cp_\O}M}
\cr
:=
&
\g_M\circ(\I_\O\otimes\g_M)\circ(\g_{\O^*}\otimes\I_{\O}\otimes \I_{M})
\circ(\I_\O\otimes\tau\otimes \I_M)\circ(\cp_\O\otimes \I_{\O\otimes M})
\cr
=
&
\g_M\circ(\I_\O\otimes\g_M)\circ(m_\O\otimes \I_{\O}\otimes \I_{M})
\circ(\I_\O\otimes\vs_\O\otimes \I_\O\otimes \I_M)\circ(\I_\O\otimes\cp_\O\otimes \I_M)\circ(\tau\otimes \I_M)
\cr
=
&
\g_M\circ(m_\O\otimes\I_M)\circ(m_\O\otimes \I_{\O}\otimes\I_{M})
\circ(\I_\O\otimes\vs_\O\otimes \I_\O\otimes \I_M)\circ(\I_\O\otimes\cp_\O\otimes \I_M)\circ(\tau\otimes \I_M)
\cr
=
&\g_M\circ(m_\O\otimes\I_M)\circ(\I_\O\otimes m_\O\otimes \I_M)
\circ(\I_\O\otimes\vs_\O\otimes \I_\O\otimes \I_M)\circ(\I_\O\otimes\cp_\O\otimes \I_M)\circ(\tau\otimes \I_M)
\cr
=
&\g_M\circ(m_\O\otimes \I_M)\circ\big(\I_\O\otimes(u_\O\circ\epsilon_\O)\otimes\ \I_M\big)
\circ(\tau\otimes \I_M)
\cr
=
&\g_M\circ(\I_\O\otimes\g_M)\circ(\I_\Omega\otimes u_\Omega\otimes \I_M)
\circ(\I_\Omega\otimes\epsilon_\Omega\otimes \I_M)\circ(\tau\otimes \I_M)
\cr
=
&\g_M\circ(\I_\O\otimes \imath_M)\circ(\I_\O\otimes\epsilon_\O\otimes \I_M)\circ(\tau\otimes \I_M)
\cr
=
&\imath_M\circ(\epsilon_\O\otimes \I_M)\circ(\I_\O\otimes\g_M)
\cr
=
&\g_{M_*}\circ(\I_\Omega\otimes\g_M).
}
In the above we have used 
$\g_M\circ(\I_\O\otimes \g_M)=\g_M\circ(m_\O\otimes \I_M)$
for the $3$rd and the $6$th equalities, 
the  associativity of $m_\O$
for the $4$th equality, 
and the antipode axiom $m_\O\circ(\vs_\O\otimes \I_\O)\circ\cp_\O=u_\O\circ\ep_\O$
for the $5$th equality. 
The rest equalities are straightforward.
\end{enumerate}
\qed
\end{proof}

The morphisms of left  dg-modules over $\O$
in Lemmas \ref{modpr} and \ref{Three module maps about O*} shall be used  crucially  in the next section.

\section{Tannakian reconstruction theorem
}


Let $\O=(\O, u_\O, m_\O, \ep_\O, \cp_\O, \vs_\O, \rd_\O)$ be a ccdg-Hopf algebra.
Consider the forgetful  functor $\bm{\o}: \gdcat{dgMod}_L(\O)\rightsquigarrow \gdcat{Ch}(\Bbbk)$
from the dg-category of left dg-modules over $\O$ to the dg-category of chain complexes over $\Bbbk$.
The functor
$\bm{\o}$ sends a left dg-module $\big(M, {\g}_M\big)$ over $\O$ to its underlying chain complex $M$,
and a morphism $\p: \big(M, {\g}_M\big)\rightarrow \big(M^\pr, {\g}_{M^\pr}\big)$ of left dg-modules over $\O$
to the underlying $\Bbbk$-linear map $\p: M\rightarrow M^\pr$.

Out of $\bm{\o}$,  we will construct the following three presheaves  
\eqalign{
\bm{\CE}_{\!\!\bm{\o}}:\mathring{\category{ccdgC}}(\Bbbk) \rightsquigarrow \category{dgA}(\Bbbk)
,\qquad
\bm{\CP}^\otimes_{\!\!\bm{\o}}:\mathring{\category{ccdgC}}(\Bbbk) \rightsquigarrow \category{Grp}
,\qquad
\bm{\mP}^\otimes_{\!\bm{\o}}:\mathring{\mathit{ho}\category{ccdgC}}(\Bbbk) \rightsquigarrow \category{Grp},
}
in turns 
and establish natural isomorphisms:
\eqalign{
\bm{\CE}_{\!\!\bm{\o}}\cong \bm{\CE}_{\O}
,\qquad
\bm{\CP}^\otimes_{\!\!\bm{\o}}\cong \bm{\CP}_{\!\!\O}
,\qquad
\bm{\mP}^\otimes_{\!\bm{\o}}\cong \bm{\mP}_{\!\O},
}
which constitute our reconstruction theorem.
The forgetful functor $\bm{\o}$ is a dg-tensor functor
sending 
\begin{itemize}
\item
$\big(\Bbbk,\g_\Bbbk\big)$ to $\Bbbk$,
\item
the tensor product 
$\big(M,\g_M\big)\otimes_{\cp_\O}\big(M^\pr,\g_{M^\pr}\big)$ of left dg-modules over $\O$
 to the tensor product $M\otimes M^\pr$ of the underlying chain complexes, and

\item the following isomorphisms of left dg-modules over $\O$
\eqalign{ 
&\Big(\big(M,\g_M\big)\otimes_{\cp_\O}\big(M^\pr,\g_{M^\pr}\big)\Big)
\otimes_{\cp_\O}\big(M^{\pr\pr},\g_{M^{\pr\pr}}\big)
\cong \big(M,\g_M\big)\otimes_{\cp_\O}\Big(\big(M^\pr,\g_{M^\pr}\big)
\otimes_{\cp_\O}\big(M^{\pr\pr},\g_{M^{\pr\pr}}\big)\Big)
,\cr
&\big(M,\g_M\big)\otimes_{\cp_\O}\big(\Bbbk,\g_\Bbbk\big)
\cong\big(M,\g_M\big)\cong\big(\Bbbk,\g_\Bbbk\big)\otimes_{\cp_\O}\big(M,\g_M\big)
}
to the corresponding isomorphisms 
$(M\otimes M^\pr)\otimes M^{\pr\pr}\cong M\otimes(M^\pr\otimes M^{\pr\pr})$
and
$M\otimes \Bbbk\cong M\cong \Bbbk\otimes M$
of the underlying chain complexes.
\end{itemize}

In Lemma \ref{ctensoring}, we have defined the dg-tensor functor 
$C\otimes :\gdcat{Ch}(\Bbbk) \rightsquigarrow \gdcat{dgComod}^{\mathit{cofr}}_L(C)$ 
for  each ccdg-coalgebra $C$.  
By composing it with $\bm{\o}$, we get a dg-tensor functor
$$
C\!\otimes \bm{\o}
: \gdcat{dgMod}_L(\O)\rightsquigarrow \gdcat{Ch}(\Bbbk)\rightsquigarrow\gdcat{dgComod}^{\mathit{cofr}}_L(C)
,
$$
sending
\begin{itemize}
\item each left dg-module $\big(M, {\g}_M\big)$ over $\O$ to a cofree left dg-comodule  
$\big(C\otimes M,  \cp_C\otimes\I_M\big)$ over $C$, and
\item each morphism $\p:\big(M, {\g}_M\big)\rightarrow \big(M^\pr, {\g}_{M^\pr}\big)$ 
of left dg-modules over $\O$ to a morphism 
$\I_C\otimes \p: \big(C\otimes M,  \cp_C\otimes \I_M\big)
\rightarrow \big(C\otimes M^\pr,  \cp_C\otimes \I_{M^\pr}\big)$ of left dg-comodules over $C$.
\end{itemize}


Let $\mathsf{End}\big(C\!\otimes \!\bm{\o}\big):=\mathsf{Nat}\big(C\!\otimes\! \bm{\o}, C\!\otimes\! \bm{\o}\big)$
be the set of natural endomorphisms of the functor $C\!\otimes\! \bm{\o}$.
 We write an element in $\mathsf{End}\big(C\!\otimes \!\bm{\o}\big)$ as $\eta^C$, 
 and denote $\eta^C_M$ as its component at  a left dg-module $\big(M,\g_M\big)$ over $\O$.
The component of $\eta^C$ at the tensor product $\big(M,\g_M\big)\otimes_{\cp_\O\!}\big(M^\pr,\g_{M^\pr}\big)$ 
is denoted by $\eta^C_{M\otimes_{\cp_\O}M'}$.
Be aware that for a chain complex $M$, the component   of $\eta^C$ at the free left dg-module 
$\big(\O\otimes M,m_\O\otimes\I_M\big)$ over $\O$ is denoted by $\eta^C_{\O\otimes M}$.
We have the following structure of dg-algebra on $\mathsf{End}\big(C\!\otimes\! \bm{\o}\big)$:
\eqn{costpd}{
\bm{\CE}_{\!\!\bm{\o}}(C):=
\big(\mathsf{End}\big(C\!\otimes\! \bm{\o}\big), \I^C, \circ , \d^C\big),
}
where $\I^C:=\I_{C\!\otimes\! \bm{\o}}$ is the identity natural transformation, 
$\circ$ is the composition and $\d^C$ is 
the differential given by
$\big(\d^C \eta^C\big)_M :=\rd_{C\otimes M, C\otimes M}\eta^C_M$.

\begin{lemma}\label{costpda}
We have a presheaf of dg-algebras
$\bm{\CE}_{\!\!\bm{\o}}:\mathring{\category{ccdgC}}(\Bbbk) \rightsquigarrow \category{dgA}(\Bbbk)$
on ${\category{ccdgC}}(\Bbbk)$,  sending
\begin{itemize}
\item each ccdg-coalgebra $C$ to the dg-algebra 
$\bm{\CE}_{\!\!\bm{\o}}(C)$, and
\item
each morphism $C\xrightarrow{f} C^\pr$ of ccdg-coalgebras
to a morphism
$\bm{\CE}_{\!\!\bm{\o}}(f):\bm{\CE}_{\!\!\bm{\o}}(C^\pr)\rightarrow \bm{\CE}_{\!\!\bm{\o}}(C)$ of dg-algebras,
where  for every $\eta^{C^\pr}\in\mathsf{End}(C^\pr\otimes\bm{\o})$ the component
of  $\bm{\CE}_{\!\!\bm{\o}}(f)\big(\eta^{C^\pr}\big)$ at each left dg-module $\big(M,\g_M\big)$ over $\O$
is defined by
\eqalign{
\bm{\CE}_{\!\!\bm{\o}}(f)\big(\eta^{C^\pr}\big)_M:=
&
\check{\mp}\Big( \check{\mq}\big(\eta^{C^\pr}_M\big)\circ(f\otimes \I_M)\Big)
\cr
=
&
\big(\I_C\otimes \check{\mq}\big(\eta^{C^\pr}_M\big)\big)\circ (\I_C\otimes  f\otimes \I_M)\circ (\cp_C\otimes \I_M)
\cr
=
&
\left(\I_C\otimes\big( \imath_M\circ (\ep_{C^\pr}\otimes \I_M)\circ\eta^{C^\pr}_M\big)\right)
\circ (\I_C\otimes  f\otimes \I_M)\circ (\cp_C\otimes I_M)
.
}
\end{itemize}

\end{lemma}

\begin{proof} 
For every $\eta^{C^\pr}\in\mathsf{End}\big(C^\pr\!\otimes\! \bm{\o}\big)$
and a morphism $f:C\to C^\pr$ of ccdg-coalgebras, we can check that
$\bm{\CE}_{\!\!\bm{\o}}(f)\big(\eta^{C^\pr}\big)\in\mathsf{End}\big(C\!\otimes\! \bm{\o}\big)$
of degree $|\eta^{C^\pr}|$
 as follows. 
For a morphism $\p:\big(M,\g_M\big)\to \big(M^\pr,\g_{M^\pr}\big)$ of left dg-modules over $\O$, 
the following diagram commutes since $\eta^{C^\pr}$ is a natural transformation:
\[
\xymatrixrowsep{1.3pc}
\xymatrixcolsep{3pc}
\xymatrix{
C^\pr\otimes M \ar[r]^-{\I_{C^\pr}\otimes \p} \ar[d]_{\eta^{C^\pr}_M}&
C^\pr\otimes M^\pr \ar[d]^{\eta^{C^\pr}_{M^\pr}}\\
C^\pr\otimes M \ar[r]^-{\I_{C^\pr}\otimes \p}&
C^\pr\otimes M^\pr
}
,\qquad\text{i.e.},\qquad
(\I_{C^\pr}\otimes\p)\circ\eta^{C^\pr}_M=(-1)^{|\eta^{C^\pr}||\p|}\eta^{C^\pr}_{M^\pr}\circ(\I_{C^\pr}\otimes\p).
\]
 Therefore we have
\eqalign{
\bm{\CE}_{\!\!\bm{\o}}(f)\big(\eta^{C^\pr}\big)_{M^\pr}\circ(\I_C\otimes\p)
&=\Big(\I_C\otimes\big(\imath_M\circ(\ep_{C^\pr}\otimes\I_M)
\circ\eta^{C^\pr}_{M^\pr}\circ(f\otimes \p)\big)\Big)\circ(\cp_C\otimes\I_M)
\cr
&=(-1)^{|\eta^{C^\pr}||\p|}\Big(\I_C\otimes\big(\imath_M\circ(\ep_{C^\pr}\otimes\p)
\circ\eta^{C^\pr}_M\circ(f\otimes \I_M)\big)\Big)\circ(\cp_C\otimes\I_M)
\cr
&=(-1)^{|\eta^{C^\pr}||\p|}(\I_C\otimes\p)\circ\bm{\CE}_{\!\!\bm{\o}}(f)\big(\eta^{C^\pr}\big)_{M}.
}
It remains to show that
\begin{itemize}
\item $\bm{\CE}_{\!\!\bm{\o}}(\I_C)=\I^C$, 
\item $\bm{\CE}_{\!\!\bm{\o}}(f)$ is a morphism of dg-algebras, and
\item $\bm{\CE}_{\!\!\bm{\o}}(g\circ f)=\bm{\CE}_{\!\!\bm{\o}}(f)\circ\bm{\CE}_{\!\!\bm{\o}}(g)$ for another morphism $g:C^\pr\to C^{\pr\pr}$ of ccdg-coalgebras.
\end{itemize}
These follow from the analogous properties of $\bm{\CE}_{M}$
 for chain complexes $M=\bm{\o}\big(M,\g_M\big)$, as stated in Lemma \ref{cdullpain}.
\qed
\end{proof}


Now we turn to construct 
the presheaf  of groups $\bm{\CP}_{\!\bm{\o}}^{\otimes} :\mathring{\category{ccdgC}}(\Bbbk) \rightsquigarrow \category{Grp}$
after some preparations.

\begin{definition}
We consider the following subsets  of $\mathsf{End}(C\otimes \bm{\o})$:

\begin{enumerate}[label=$({\alph*})$,leftmargin=.6cm]

\item $Z_0\mathsf{End}(C\otimes \bm{\o})$
consisting of every element  $\eta^C \in \mathsf{End}(C\otimes \bm{\o})_0$ satisfying $\d^C\eta^C=0$.

\item  $\mathsf{End}^\otimes(C\otimes \bm{\o})$  
consisting of every element  $\eta^C \in \mathsf{End}(C\otimes \bm{\o})_0$ satisfying the conditions
\eqn{ctensorial}{ 
\begin{aligned}
\eta^C_\Bbbk&=\I^C_\Bbbk:=\I_{C\otimes\Bbbk}
,\cr
\eta^C_{M\otimes_{\cp_\O\!} M^\pr} &=\eta^C_M\otimes_{\cp_C} \eta^C_{M^\pr}
:= \big(\eta^C_M \otimes \check{\mq}(\eta^C_{M^\pr})\big)\circ (\I_C\otimes \t\otimes \I_{M^\pr})
\circ (\cp_C\otimes \I_{M\otimes M^\pr}),
\end{aligned}
}
where the $2$nd relation hold for
all left dg-modules $\big(M,\g_M\big)$ and $\big(M^\pr,\g_{M^\pr}\big)$ over $\O$.

\item $Z_0\mathsf{End}^\otimes(C\otimes \bm{\o})
:=Z_0\mathsf{End}(C\otimes \bm{\o})\cap\mathsf{End}^\otimes(C\otimes \bm{\o})$.
\end{enumerate}
\end{definition}

We say an  $\eta^C\in \mathsf{End}^\otimes\big(C\!\otimes\! \bm{\o}\big)$ a \emph{tensor} natural transformation
and an $\eta^C\in \mathsf{Z_0End}^\otimes(C\otimes \bm{\o})$  a \emph{dg-tensor} natural transformation.

\begin{lemma}\label{ctensorderi}
If $\eta^C \in \mathsf{End}^\otimes \big(C\!\otimes\! \bm{\o}\big)$, then 
for every pair of left dg-modules $\big(M,\g_M\big)$ and $\big(M^\pr,\g_{M^\pr}\big)$ over $\O$, we have
$$
\big(\d^C\eta^C\big)_{M\otimes_{\cp_\O\!} M^\pr} =\big(\d^C\eta^C\big)_M\otimes_{\cp_C} \eta^C_{M^\pr} 
+\eta^C_M\otimes_{\cp_C} \big(\d^C\eta^C\big)_{M^\pr}
.
$$
\end{lemma}

\begin{proof}
Since $\eta^C$ is a tensor natural transformation, we have
\[
\begin{aligned}
\big(\d^C\eta^C\big)_{M\otimes_{\cp_\O}M^\pr}
&=\rd_{C\otimes M\otimes M^\pr,C\otimes M\otimes M^\pr}\eta^C_{M\otimes_{\cp_\O}M^\pr}\\
&=\rd_{C\otimes M\otimes M^\pr,C\otimes M\otimes M^\pr}\eta^C_M\otimes_{\cp_C}\eta^C_{M^\pr}\\
&=\big(\rd_{C\otimes M,C\otimes M}\eta^C_M\big)\otimes_{\cp_C}\eta^C_{M^\pr}+(-1)^{|\eta^C|}\eta^C_M\otimes_{\cp_C}\big(\rd_{C\otimes M^\pr,C\otimes M^\pr}\eta^C_{M^\pr}\big)\\
&=\big(\d^C\eta^C\big)_M\otimes_{\cp_C} \eta^C_{M^\pr}+\eta^C_M\otimes_{\cp_C} \big(\d^C\eta^C\big)_{M^\pr}.
\end{aligned}
\]
\qed
\end{proof}

Clearly, $Z_0\mathsf{End}^\otimes\big(C\!\otimes\! \bm{\o}\big)$ is closed under composition and contains 
$\I^C=\I_{C\!\otimes\!\bm{\o}}$. 
Thus, we have a monoid
\eqn{costpdex}{
\bm{\CP}^\otimes_{\!\!\bm{\o}}(C):=
\big(Z_0\mathsf{End}^\otimes\big(C\!\otimes\! \bm{\o}\big),\I^C,\circ\;\big).
} 
We shall show that this is in fact a group. 
We begin with a technical lemma.

\begin{lemma} \label{free modules and eta}
For every $\eta^C\in \mathsf{End}\big(C\!\otimes\! \bm{\o}\big)$  its component $\eta^C_{\O\otimes M}$ at 
the free left $\O$ dg-module $\big(\O\otimes M,m_\O\otimes \I_M\big)$ generated by a chain complex $M$ is
$\eta^C_{\O\otimes M}=\eta^C_{\O}\otimes \I_M$.
\end{lemma}
\begin{proof}
For each $z\in M$, define a linear map $f_z:\O\to\O\otimes M$ of degree $|z|$ by 
$f_z(a):=(-1)^{|a||z|}a\otimes z$ for all  $a\in \O$. 
Then $f_z:\big(\O,m_\O\big)\to\big(\O\otimes M,m_\O\otimes\I_M\big)$ is a morphism of left dg-modules over $\O$.
Since $\eta^C$ is a natural transformation, the following diagram commutes
\[
\xymatrixrowsep{1.3pc}
\xymatrixcolsep{3pc}
\xymatrix{
C\otimes\O \ar[r]^-{\I_C\otimes f_z} \ar[d]_{\eta^C_\O}&
C\otimes\O\otimes M \ar[d]^{\eta^C_{\O\otimes M}}\\
C\otimes\O \ar[r]^-{\I_C\otimes f_z}&
C\otimes\O\otimes M
}
,\qquad\text{i.e.},\qquad
\eta^C_{\O\otimes M}\circ(\I_C\otimes f_z)=(-1)^{|\eta^C||z|}(\I_C\otimes f_z)\circ\eta^C_\O.
\]
For every $c\in C$ and $a\in \O$ we can write $\eta^C_\O(c\otimes a)$
as a finite sum $\eta^C_\O(c\otimes a)=\sum_{i}c_i\otimes a_i$
for some $c_i\in C$ and $a_i\in \O$, where $|c_i|+|a_i|=|c|+|a|+|\eta^C|$.  
Then we obtain that 
\eqalign{
\eta^C_{\O\otimes M}\big(c\otimes a\otimes z\big)
&=(-1)^{(|a|+|c|)|z|}\eta^C_{\O\otimes M}\circ(\I_C\otimes f_z)\big(c\otimes a\big)
\\
&=(-1)^{(|a|+|c|+|\eta^C|)|z|}(\I_C\otimes f_z)\circ\eta^C_{\O}\big(c\otimes a\big)
\\
&=(-1)^{(|a|+|c|+|\eta^C|)|z|}\sum_i(\I_C\otimes f_z)\big(c_i\otimes a_i\big)
\\
&=\sum_i c_i\otimes a_i\otimes z
=(\eta^C_\O\otimes \I_M)\big(c\otimes a\otimes z\big).
}
It follows that ${\eta}^C_{\O\otimes M}={\eta}^C_{\O}\otimes \I_M$
since the above equality holds for all $c$, $a$ and $z$.
\qed
\end{proof}

\begin{proposition} \label{group-valued}
$\bm{\CP}^\otimes_{\!\!\bm{\o}}(C)=\big(Z_0\mathsf{End}^{\otimes}\big(C\!\otimes\! \bm{\o}\big), \I^C, \circ\big)$
 is a group
for every ccdg-coalgebra $C$.
\end{proposition}

\begin{proof}
Associated with  each $\eta^C\in Z_0\mathsf{End}^{\otimes}\big(C\!\otimes\! \bm{\o}\big)$,  
we introduce a natural endomorphism $\vs\big(\eta^C\big)\in\mathsf{End}\big(C\!\otimes\! \bm{\o}\big)$,
whose component $\vs\big(\eta^C\big)_M$ at each left dg-module $\big(M,\g_M\big)$ over $\O$ is defined by
\eqalign{
\vs\big(\eta^C\big)_M
:=& (\I_C\otimes \g_M)\circ (\eta^C_{\O^*}\otimes \I_M)\circ (\I_C\otimes u_\O\otimes \I_M)
\circ (\I_C\otimes \imath^{-1}_M)
:C\otimes M\to C\otimes M.
}
We verify that  $\vs\big(\eta^C\big)$ is  a natural transformation, since for every morphism  
$\p: \big(M, \g_M\big)\to \big(M^\pr, \g_{M^\pr}\big)$ of left dg-modules over $\O$,
the following diagram commutes
\[
\xymatrixrowsep{3pc}
\xymatrixcolsep{3.3pc}
\xymatrix{
\ar@/^1.2pc/[rrrr]^{\vs(\eta^C)_M}
C{\!\otimes\!} M \ar[r]_{\I_C\otimes \imath^{-1}_M} \ar[d]_{\I_C\otimes \p}&
C{\!\otimes} \Bbbk {\otimes\!} M \ar[r]_-{\I_C\otimes u_\O\otimes \I_M} \ar[d]^{\I_{C\otimes \Bbbk}\otimes \p}&
C{\!\otimes\!} \O{\!\otimes\!} M \ar[r]_-{\eta^C_{\O^*}\otimes \I_M} \ar[d]^{\I_{C\otimes \O}\otimes \p}&
C{\!\otimes\!} \O{\!\otimes\!} M\ar[r]_-{\I_C\otimes \g_M} \ar[d]^{\I_{C\otimes \O}\otimes \p}&
C{\!\otimes\!} M \ar[d]^{\I_C\otimes \p}
\cr
\ar@/_1.2pc/[rrrr]_{\vs(\eta^C)_{M^\pr}}
C{\!\otimes\!} M^\pr \ar[r]^{\I_C\otimes \imath^{-1}_{M^\pr}}&
C{\!\otimes} \Bbbk {\otimes\!} M^\pr \ar[r]^-{\I_C\otimes u_\O\otimes \I_{M^\pr}}&
C{\!\otimes\!} \O{\!\otimes\!} M^\pr \ar[r]^-{\eta^C_{\O^*}\otimes \I_{M^\pr}}&
C{\!\otimes\!} \O{\!\otimes\!} M^\pr\ar[r]^-{\I_C\otimes \g_{M^\pr}}&
C{\!\otimes\!} M^\pr.
}
\]
We claim that $\vs\big(\eta^C\big)$ is also in $Z_0\mathsf{End}^{\otimes}(C\otimes\bm{\o})$. 
First, note that $\vs\big(\eta^C\big)$ is in $Z_0\mathsf{End}(C\otimes\bm{\o})$. 
This is because for each left dg-module $\big(M,\g_M\big)$ over $\O$, all the maps 
$\I_C\otimes \g_M$, $\eta^C_{\O^*}\otimes \I_M$, $\I_C\otimes u_\O\otimes \I_M$ and $\I_C\otimes \imath^{-1}_M$ 
are of degree $0$ and in the kernels of differentials. 
Next, we show $\vs\big(\eta^C\big)$ is a tensor natural transformation. From Lemma \ref{Three module maps about O*}\emph{(a)}, 
the coproduct 
$\cp_\O:\big(\O^*,\g_{\O^*}\big)\to\big(\O^*\otimes\O^*,\g_{\O^*\otimes_{\cp_\O}\O^*}\big)$
is a morphism of left dg-modules over $\O$.
Since $\eta^C$ is a tensor natural transformation, we have
\[
\xymatrixrowsep{1.3pc}
\xymatrixcolsep{3pc}
\xymatrix{
C\otimes\O \ar[r]^{\I_C\otimes\cp_\O} \ar[d]_{\eta^C_{\O^*}}&
C\otimes\O\otimes\O \ar[d]^{\eta^C_{\O^*\otimes_{\cp_\O}\O^*}=\eta^C_{\O^*}\otimes_{\cp_C}\eta^C_{\O^*}}\\
C\otimes\O \ar[r]^{\I_C\otimes\cp_\O}&
C\otimes\O\otimes\O
}
\quad\hbox{i.e.,}\quad
(\eta^C_{\O^*}\otimes_{\cp_C\!}\eta^C_{\O^*})\circ(\I_C\otimes\cp_\O)=(I_C\otimes\cp_\O)\circ\eta^C_{\O^*}.
\]
Thus for left dg-modules $\big(M,\g_M\big)$ and $\big(M^\pr,\g_{M^\pr}\big)$ over $\O$, we have
\[
\begin{aligned}
\vs\big(\eta^C\big)_M\otimes_{\cp_C}\vs\big(\eta^C\big)_{M^\pr}
=&(\I_C\otimes\g_M\otimes\g_{M^\pr})\circ(\I_{C\otimes \O}\otimes\tau\otimes\I_{M^\pr})\\
&\circ\Big(\big((\eta^C_{\O^*}\otimes_{\cp_C}\eta^C_{\O^*})\circ(\I_C\otimes\cp_\O)
\circ(\I_C\otimes u_\O)\big)\otimes\I_{M\otimes M^\pr}\Big)\circ(\I_C\otimes\imath^{-1}_M\otimes\I_{M^\pr})\\
=&(\I_C\otimes\g_M\otimes\g_{M^\pr})\circ(\I_{C\otimes \O}\otimes\tau\otimes\I_{M^\pr})\\
&\circ\Big(\big((I_C\otimes\cp_\O)\circ\eta^C_{\O^*}\circ(\I_C\otimes u_\O)\big)\otimes\I_{M\otimes M^\pr}\Big)
\circ(\I_C\otimes\imath^{-1}_M\otimes\I_{M^\pr})\\
=&(\I_C\otimes\g_{M\otimes_{\cp_\O}M^\pr})
\circ\Big(\big(\eta^C_{\O^*}\circ(\I_C\otimes u_\O)\big)\otimes\I_{M\otimes M^\pr}\Big)
\circ(\I_C\otimes\imath^{-1}_M\otimes\I_{M^\pr})\\
=&\vs\big(\eta^C\big)_{M\otimes_{\cp_\O}M^\pr}.
\end{aligned}
\]

Moreover, from Lemma \ref{Three module maps about O*}\emph{(b)},  
the counit $\ep_\O:\big(\O^*,\g_{\O^*}\big)\to\big(\Bbbk,\g_\Bbbk\big)$ 
is also a morphism of left dg-modules over $\O$. 
Since $\eta^C$ is a tensor natural transformation, we have
\[
\xymatrixrowsep{1.3pc}
\xymatrixcolsep{3pc}
\xymatrix{
C\otimes\O \ar[r]^{\I_C\otimes\ep_\O} \ar[d]_{\eta^C_{\O^*}}&
C\otimes\Bbbk \ar[d]^{\eta^C_\Bbbk=\I^C_\Bbbk}\\
C\otimes\O \ar[r]^{\I_C\otimes\ep_\O}&
C\otimes\Bbbk
}
\qquad\hbox{i.e.,}\qquad
(\I_C\otimes\ep_\O)=(\I_C\otimes\ep_\O)\circ\eta^C_{\O^*}.
\]
Therefore we have $\vs\big(\eta^C\big)_\Bbbk=(\I_C\otimes\ep_\O)\circ\eta^C_{\O^*}\circ(\I_C\otimes u_\O)
=(\I_C\otimes\ep_\O)\circ(\I_C\otimes u_\O)=\I^C_{\Bbbk}$. 
This shows $\vs\big(\eta^C\big)\in Z_0\mathsf{End}^{\otimes}\big(C\!\otimes\!\bm{\o}\big)$. 

Finally, we show that $\vs\big(\eta^C\big)$ is the left inverse of $\eta^C$. 
Lemma \ref{Three module maps about O*}\emph{(c)} states that for each left dg-module $\big(M, \g_M\big)$ over $\O$, 
the action map $\g_M:\big(\O^*\otimes M,\g_{\O^*\otimes_{\cp_\O}M}\big)\rightarrow \big(M_*,\g_{M_*}\big)$
is a morphism of left dg-modules over $\O$.
Since $\eta^C$ is a tensor natural transformation, we have
$$
\xymatrixrowsep{1.3pc}
\xymatrixcolsep{2.5pc}
\xymatrix{
\ar[d]_-{\eta^C_{\O^*\otimes_{\cp_\O}M}=\eta^C_{\O^*}\otimes_{\cp_C} \eta^C_M}
C\!\otimes\!\O\!\otimes\! M \ar[r]^-{\I_C\otimes \g_M} & C\!\otimes\! M
\ar[d]^-{\eta^C_{M_*}}
\cr
C\!\otimes\! \O\!\otimes\! M \ar[r]^-{\I_C\otimes \g_M} & C\!\otimes\! M
}
\quad\hbox{i.e.,}\quad
(\I_C\otimes \g_M)\circ(\eta^C_{\O^*}\otimes_{\cp_C}\!\eta^C_M) 
= \eta^C_{M_*}\circ (\I_C\otimes \g_M).
$$

Note that $\eta^C_{M_*}=\I^C_M$. Indeed, by Lemma \ref{modpr}, 
the action map $\g_{M_*}:\big(\O\otimes M,m_\O\otimes \I_M\big)\to\big(M_*,\g_{M_*}\big)$
and the counit $\ep_\O:\big(\O,m_\O\big)\to\big(\Bbbk,\g_\Bbbk\big)$ are morphisms of left dg-modules over $\O$. 
Since $\eta^C$ is a tensor natural transformation, the following diagrams commute:
\[
\xymatrixrowsep{1.3pc}
\xymatrix{
C\otimes \O\otimes M \ar[r]^{\I_C\otimes \g_{M_*}} \ar[d]_{\eta^C_{\O\otimes M}=\eta^C_\O\otimes \I_M}&
C\otimes M\ar[d]^{\eta^C_{M_*}}\\
C\otimes \O\otimes M \ar[r]^{\I_C\otimes \g_{M_*}}&
C\otimes M
},
\qquad\qquad
\xymatrix{
C\otimes\O \ar[r]^{\I_C\otimes\ep_\O} \ar[d]_{\eta^C_\O}&
C\otimes\Bbbk \ar[d]^{\eta^C_\Bbbk=\I^C_\Bbbk}\\
C\otimes\O \ar[r]^{\I_C\otimes\ep_\O}&
C\otimes\Bbbk
}.
\]
The equality on the left diagram is due to Lemma \ref{free modules and eta}. Therefore,
we have
\[
\begin{aligned}
\eta^C_{M_*}
=&\eta^C_{M_*}\circ\Big(\I_C\otimes\big(\g_{M_*}\circ(u_\O\otimes\I_M)\circ\imath^{-1}_M\big)\Big)\\
=&\eta^C_{M_*}\circ(\I_C\otimes\g_{M_*})\circ(\I_C\otimes u_\O\otimes \I_M)\circ(\I_C\otimes\imath^{-1}_M)\\
=&(\I_C\otimes\g_{M_*})\circ(\eta^C_\O\otimes\I_M)\circ(\I_C\otimes u_\O\otimes\I_M)
\circ(\I_C\otimes\imath^{-1}_M)\\
=&(\I_C\otimes\imath_M)\circ(\I_C\otimes\ep_\O\otimes\I_M)\circ (\eta^C_\O\otimes \I_M)
\circ(\I_C\otimes u_\O\otimes\I_M)\circ(\I_C\otimes\imath^{-1}_M)\\
=&(\I_C\otimes\imath_M)\circ(\I_C\otimes\ep_\O\otimes\I_M)
\circ(\I_C\otimes u_\O\otimes\I_M)\circ(\I_C\otimes\imath^{-1}_M)=\I^C_M.
\end{aligned}
\]

Using \eq{ctensorial}, we finally prove that $\vs\big(\eta^C\big)\circ\eta^C=\I^C$:
\eqalign{
\vs\big(\eta^C\big)_M\circ\eta^C_M
=&\vs\big(\eta^C\big)_M\circ\check{\mp}(\check{\mq}(\eta^C_M))\\
=&\vs\big(\eta^C\big)_M\circ(\I_C\otimes\check{\mq}(\eta^C_{M}))\circ(\cp_C\otimes\I_M)\\
=&(\I_C\otimes\g_M)\circ(\eta^C_{\O^*}\otimes_{\cp_C}\eta^C_M)
\circ(\I_C\otimes u_\O\otimes I_M)\circ(\I_C\otimes\imath^{-1}_M)\\
=&\eta^C_{M_*}\circ (\I_C\otimes \g_M)\circ(\I_C\otimes u_\O\otimes I_M)\circ(\I_C\otimes\imath^{-1}_M)\\
=&(\I_C\otimes \g_M)\circ(\I_C\otimes u_\O\otimes I_M)\circ(\I_C\otimes\imath^{-1}_M)=\I^C_M.
}
We conclude that 
$\bm{\CP}^\otimes_{\!\!\bm{\o}}(C)
=Z_0\mathsf{End}^\otimes(C\otimes\bm{\o})=Z_0\mathsf{Aut}^\otimes(C\otimes\bm{\o})$ is a group,
since every monoid with all left inverses is a group.
%
%
%
%
%
\qed
\end{proof}

The following lemma shows that the above construction is functorial.

\begin{lemma}\label{costpdx}
We have a presheaf of groups
$\bm{\CP}^\otimes_{\!\!\bm{\o}}:\mathring{\category{ccdgC}}(\Bbbk) \rightsquigarrow \category{Grp}$ 
on the category $\ccdgc$ of ccdg-coalgebras, sending
\begin{itemize}
\item
each ccdg-coalgebra $C$ to the group $\bm{\CP}^\otimes_{\!\!\bm{\o}}(C)$, and
\item
each morphism $f:C\rightarrow C^\pr$ of ccdg-coalgebras to a homomorphism
$\bm{\CP}^\otimes_{\!\!\bm{\o}}(f)
:\bm{\CP}^\otimes_{\!\!\bm{\o}}(C^\pr)\rightarrow \bm{\CP}^\otimes_{\!\!\bm{\o}}(C)$
of groups
defined by $\bm{\CP}^\otimes_{\!\!\bm{\o}}(f):=\bm{\CE}_{\!\!\bm{\o}}(f)$.
\end{itemize}
\end{lemma}

\begin{proof}
It suffices to check that for every morphism $f:C\to C^\pr$ of ccdg-coalgebras,
we have
$\bm{\CE}_{\!\!\bm{\o}}(f)\big(\eta^{C^\pr}\big)\in Z_0\mathsf{End}^\otimes \big(C\!\otimes\! \bm{\o}\big)$
whenever
$\eta^{C^\pr}\in Z_0\mathsf{End}^\otimes \big(C^\pr\!\otimes\! \bm{\o}\big)$, i.e.,
\begin{enumerate}[label=(\arabic*),leftmargin=.8cm]

\item $\bm{\CE}_{\!\!\bm{\o}}(f)\big(\eta^{C^\pr}\big)\in Z_0\mathsf{End}\big(C\!\otimes\! \bm{\o}\big)$;

\item 
$\bm{\CE}_{\!\!\bm{\o}}(f)\big(\eta^{C^\pr}\big)_\Bbbk=\I^C_\Bbbk$;

\item 
$\bm{\CE}_{\!\!\bm{\o}}(f)\big(\eta^{C^\pr}\big)_{M\otimes_{\cp_\O}M^\pr}
=\bm{\CE}_{\!\!\bm{\o}}(f)\big(\eta^{C^\pr}\big)_M
\otimes_{\cp_C}\bm{\CE}_{\!\!\bm{\o}}(f)\big(\eta^{C^\pr}\big)_{M^\pr}$
for all left dg-modules $\big(M,\g_M\big)$ and $\big(M^\pr,\g_{M^\pr}\big)$ over $\O$.
\end{enumerate}
Property $(1)$ is obvious since $\bm{\CE}_{\!\!\bm{\o}}(f)$ is a chain map.
Property $(2)$ follows from  $\eta^{C^\pr}_\Bbbk=\I^{C^\pr}_\Bbbk$ and $\ep_{C^\pr}\circ f=\ep_C$, since
we have
\[
\bm{\CE}_{\!\!\bm{\o}}(f)\big(\eta^{C^\pr}\big)_\Bbbk
=(\I_C\otimes m_\Bbbk)\circ\Big(\I_C\otimes\big((\ep_{C^\pr}\otimes \I_\Bbbk)
\circ \eta^{C^\pr}_\Bbbk\circ(f\otimes\I_\Bbbk)\big)\Big)\circ(\cp_C\otimes\I_\Bbbk)=\I^C_\Bbbk.
\]
Note that  Property $(3)$ is equivalent to the condition 
$$
\check{\mq}\Big(\bm{\CE}_{\!\!\bm{\o}}(f)\big(\eta^{C^\pr}\big)_{M\otimes_{\cp_\O}M^\pr}\Big)
=\check{\mq}\Big(\bm{\CE}_{\!\!\bm{\o}}(f)\big(\eta^{C^\pr}\big)_M
\otimes_{\cp_C}\bm{\CE}_{\!\!\bm{\o}}(f)\big(\eta^{C^\pr}\big)_{M^\pr}\Big)
,
$$
which can be checked as follows:
\eqalign{
\check{\mq}\Big(
\bm{\CE}_{\!\!\bm{\o}}(f)& \big(\eta^{C^\pr}\big)_M
\otimes_{\cp_C}\bm{\CE}_{\!\!\bm{\o}}(f)\big(\eta^{C^\pr}\big)_{M^\pr}
\Big)\\
&=\Big(\check{\mq}\big(\bm{\CE}_{\!\!\bm{\o}}(f) \big(\eta^{C^\pr}\big)_M\big)
\otimes\check{\mq}\big(\bm{\CE}_{\!\!\bm{\o}}(f)\big(\eta^{C^\pr}\big)_{M^\pr}\big)\Big)
\circ(\I_C\otimes\t\otimes \I_{M^\pr})\circ(\cp_C\otimes\I_{M\otimes M^\pr})\\
&=\Big(\check{\mq}\big(\eta^{C^\pr}_M\big)\otimes\check{\mq}\big(\eta^{C^\pr}_{M^\pr}\big)\Big)
\circ(f\otimes \I_M\otimes f\otimes \I_{M^\pr})\circ(\I_C\otimes\t\otimes \I_{M^\pr})\circ(\cp_C\otimes\I_{M\otimes M^\pr})\\
&=\Big(\check{\mq}\big(\eta^{C^\pr}_M\big)\otimes\check{\mq}\big(\eta^{C^\pr}_{M^\pr}\big)\Big)
\circ(\I_{C^\pr}\otimes\t\otimes \I_{M^\pr})\circ(\cp_{C^\pr}\otimes\I_{M\otimes M^\pr})\circ(f\otimes\I_{M\otimes M^\pr})\\
&=\check{\mq}\big(\eta^{C^\pr}_M\otimes_{\cp_{C^\pr}}\eta^{C^\pr}_{M^\pr}\big)
\circ(f\otimes\I_{M\otimes M^\pr})
=\check{\mq}\big(\eta^{C^\pr}_{M\otimes_{\cp_\O}M^\pr}\big)
\circ(f\otimes\I_{M\otimes M^\pr})\\
&=\check{\mq}\Big(
\bm{\CE}_{\!\!\bm{\o}}(f)\big(\eta^{C^\pr}\big)_{M\otimes_{\cp_\O}M^\pr}
\Big)
,
}
where we have used $(f\otimes f)\circ\cp_C=\cp_{C^\pr}\circ f$ on the $3$rd equality. \qed
\end{proof}

Now we turn to construct 
the presheaf  of groups 
$\bm{\mP}_{\!\bm{\o}}^{\otimes} :\mathring{\mathit{ho}\category{ccdgC}}(\Bbbk) \rightsquigarrow \category{Grp}$ 
on the homotopy category ${\mathit{ho}\category{ccdgC}}(\Bbbk)$.
Later in this section we shall construct an isomorphism
$\bm{\CP}^\otimes_{\!\bm{\o}}\cong \bm{\CP}_{\!\O} :\mathring{\category{ccdgC}}(\Bbbk) \rightsquigarrow \category{Grp}$,
where $\bm{\CP}_{\!\O}$  is the representable presheaf of groups which
is represented by the ccdg-Hopf algebra $\O$ and induces
$\bm{\mP}_{\!\O}:\mathring{\mathit{ho}\category{ccdgC}}(\Bbbk) \rightsquigarrow \category{Grp}$.
Similarly, $\bm{\CP}^\otimes_{\!\bm{\o}}$ shall induces $\bm{\mP}_{\!\bm{\o}}^{\otimes}$ on 
${\mathit{ho}\category{ccdgC}}(\Bbbk)$.

Remind
that $\bm{\CP}_{\!\O}(C)$ is the group formed by the set $\HOM_{\ccdgc}(C,\O)$ 
of all morphisms of ccdg-coalgebras, while
$\bm{\mP}_{\!\O}(C)$ is the group formed 
by the set $\HOM_{\mathit{ho}\ccdgc}(C,\O)$ of homotopy types 
of elements in  $\HOM_{\ccdgc}(C,\O)$.  
Likewise, we need to define homotopy types of elements in 
$Z_0\mathsf{End}^{\otimes}\big(C\!\otimes\! \bm{\o}\big)$--taking homology classes 
is not compatible with the tensor condition \eq{ctensorial}: let $\eta^C \in  Z_0\mathsf{End}^\otimes\big(C\!\otimes\! \bm{\o}\big)$ and  
$\tilde\eta^C = \eta^C +\d^C\l^C$ for some $\l^C \in \mathsf{End}\big(C\!\otimes\! \bm{\o}\big)$ of degree $1$,
then $\tilde\eta^C$ and $\eta^C$ belong to the same homology class but $\tilde\eta^C$, in general, is not
a tensor natural transformation.

\begin{definition}
A homotopy pair  on $Z_0\mathsf{End}^{\otimes}\big(C\!\otimes\! \bm{\o}\big)$
is  a pair $\big(\eta(t)^C, \l(t)^C\big)$  of $\eta^C(t)\in  \mathsf{End}\big(C\!\otimes\! \bm{\o}\big)_0[t]$
and $\l(t)^C\in   \mathsf{End}\big(C\!\otimes\! \bm{\o}\big)_{1}[t]$, 
where $t$ is a polynomial time variable of degree $0$, 
satisfying the   homotopy flow equation $\Fr{d}{dt}\eta(t)^C= \d^C\l(t)^C$ 
generated by $\l(t)^C$ subject to the following
conditions:
\eqalign{
\eta(0)^C\in  Z_0\mathsf{End}^\otimes \big(C\!\otimes\! \bm{\o}\big)
,\quad
\begin{cases}
\l(t)^C_{\Bbbk}=0,\cr
\l(t)^C_{M\otimes_{\cp_\O\!} M^\pr} 
= \l(t)^C_M\otimes_{\cp_C} \eta(t)^C_{M^\pr} +\eta(t)^C_M\otimes_{\cp_C} \l(t)^C_{M^\pr}
.
\end{cases}
}
\end{definition}

Let $\big(\eta(t)^C, \l(t)^C\big)$ be a homotopy pair on 
$Z_0\mathsf{End}^{\otimes}\big(C\!\otimes\! \bm{\o}\big)$.
It follows from the homotopy flow equation that $\eta(t)^C$ is uniquely determined by
$\eta(t)^C= \eta(0)^C + \d^C\int^t_0\l(s)^C \mathit{ds}$, and we have $\d^C\eta(t)^C=0$ since $\d^C\eta(0)^C=0$.
From the condition $\eta^C(0)^{\Bbbk}=\I_{C\otimes \Bbbk}$ and $\l(t)^C_{\Bbbk}=0$, we have
$\eta^C_{\Bbbk}(t)=\I_{C\otimes \Bbbk}$.
By applying  Lemma \ref{ctensorderi}, we also have
\eqalign{
\Fr{d}{dt}\Big(\eta(t)^C_{M\otimes_{\cp_\O}\! M^\pr}&-\eta(t)^C_M\otimes_{\cp_C}\! \eta(t)^C_{M^\pr}\Big)
\cr
&
=\d^C\left(\l(t)^C_{M\otimes_{\cp_\O}\! M^\pr} 
- \l(t)^C_M\otimes_{\cp_C}\! \eta(t)^C_{M^\pr} -\eta(t)^C_M\otimes_{\cp_C}\! \l(t)^C_{M^\pr}\right)
\cr
&
=0,
}
so that $\eta(t)^C_{M\otimes_{\cp_\O\!} M^\pr} =\eta(t)^C_M\otimes_{\cp_C} \eta(t)^C_{M^\pr}$
for all $t$ since $\eta(0)^C_{M\otimes_{\cp_\O\!} M^\pr} =\eta(0)^C_M\otimes_{\cp_C} \eta(0)^C_{M^\pr}$.
Therefore  $\eta(t)^C$ is a family of elements in $Z_0\mathsf{End}^{\otimes}\big(C\!\otimes\! \bm{\o}\big)$.
Then,  we declare that $\eta(1)^C$ is homotopic to $\eta(0)^C$ by the homotopy $\int^1_0\l(t)^C \mathit{dt}$,
and denote $\eta(0)^C\sim \eta(1)^C$, which is clearly an equivalence relation.
In other words,  two elements $\eta^C$ and $\tilde\eta^C$ 
in the set $Z_0\mathsf{End}^{\otimes}\big(C\!\otimes\! \bm{\o}\big)$ 
are homotopic, $\eta^C\sim \tilde\eta^C$, if there is a homotopy pair connecting  them (by the time $1$ map).  
Then,  we also say that $\eta^C$ and $\tilde\eta^C$ have the same homotopy type, 
and denote it as $[\eta^C]=[\tilde \eta^C]$.

Let $\mathit{ho}Z_0\mathsf{End}^{\otimes}\big(C\!\otimes\! \bm{\o}\big)$ be the set of homotopy types of elements in
$Z_0\mathsf{End}^{\otimes}\big(C\!\otimes\! \bm{\o}\big)$.
It is a routine check that 
$\eta^{\pr C}\circ \eta^C \sim \tilde\eta^{\pr C}\circ \tilde\eta^C 
\in Z_0\mathsf{End}^{\otimes}\big(C\!\otimes\! \bm{\o}\big)$
whenever $\eta^{\pr C}\sim \tilde\eta^{\pr C}, \eta^C \sim \tilde\eta^C 
\in Z_0\mathsf{End}^{\otimes}\big(C\!\otimes\! \bm{\o}\big)$
and the homotopy type of $\eta^{\pr C}\circ \eta^C$ 
depends only on the homotopy types of $\eta^{\pr C}$ and $\eta^C$.
Therefore we have well-defined associative composition 
$[\eta^{\pr C}]\diamond [\eta^C]:= [\eta^{\pr C}\circ \eta^C]$. 
This shows that 
we have a group
\eqn{costpdexxx}{
\bm{\mP}^\otimes_{\!\!\bm{\o}}(C):=
\big( \mathit{ho}Z_0\mathsf{End}^{\otimes}\big(C\!\otimes\! \bm{\o}\big), [\I^C], \diamond\;\big).
} 
The following lemma shows that this construction is functorial.

\begin{lemma}\label{costpdy}
We have a presheaf of groups
$\bm{\mP}^\otimes_{\!\bm{\o}}:\mathring{\mathit{ho}\category{ccdgC}}(\Bbbk) \rightsquigarrow \category{Grp}$
on the homotopy category $\mathit{ho}{\category{ccdgC}}(\Bbbk)$ of ccdg-coalgebras, sending
\begin{itemize}
\item
each ccdg-coalgebra $C$ to the group 
$\bm{\mP}^\otimes_{\!\bm{\o}}(C)
:=\big( \mathit{ho}Z_0\mathsf{End}^{\otimes}\big(C\!\otimes\! \bm{\o}\big), [\I^C], \diamond\big)$, and
\item
each morphism $f:C\rightarrow C^\pr$ of dg-coalgebras
to the group homomorphism 
$\bm{\mP}^\otimes_{\!\bm{\o}}([f]):\bm{\mP}^\otimes_{\!\bm{\o}}(C^\pr)\rightarrow \bm{\mP}^\otimes_{\!\bm{\o}}(C)$
defined by
\[
\bm{\mP}^\otimes_{\!\bm{\o}}([f])\big([\eta^{C^\pr}]\big):=
\Big[\bm{\CP}^\otimes_{\!\!\bm{\o}}(f)\big(\eta^{C^\pr}\big)\Big].
\]
\end{itemize}
\end{lemma}
\begin{proof}
All we need to show is that 
$\bm{\CP}^\otimes_{\!\!\bm{\o}}(f)\big(\eta^{C^\pr}\big)
\sim
\bm{\CP}^\otimes_{\!\!\bm{\o}}(\tilde{f})\big(\tilde{\eta}^{C^\pr}\big)
\in Z_0\mathsf{End}^{\otimes}\big(C\!\otimes\! \bm{\o}\big)$
whenever $f\sim \tilde{f} \in \HOM_{\ccdgc}(C, C^\pr)$ and $\eta^{C^\pr}\sim\tilde{\eta}^{C^\pr}
\in Z_0\mathsf{End}^{\otimes}\big(C^\pr\!\otimes\! \bm{\o}\big)$.
It suffices to show the following statement: 
Let $\big(f(t),s(t)\big)$ be a homotopy pair on $\HOM_\ccdgc(C,C^\pr)$ and $\big(\eta(t)^{C^\pr},\l(t)^{C^\pr}\big)$ 
be a homotopy pair on $Z_0\mathsf{End}^{\otimes}\big(C^\pr\!\otimes\! \bm{\o}\big)$. Then the pair
\[
\Big(
\xi(t)^C:=
\bm{\CE}_{\!\!\bm{\o}}\big(f(t)\big)\big(\eta(t)^{C^\pr}\big),
\;
\chi(t)^C:=
\bm{\CE}_{\!\!\bm{\o}}\big(f(t)\big)\big(\l(t)^{C^\pr}\big)+
\bm{\CE}_{\!\!\bm{\o}}\big(s(t)\big)\big(\eta(t)^{C^\pr}\big)
\Big)
\]
is a homotopy pair on $Z_0\mathsf{End}^{\otimes}\big(C\!\otimes\! \bm{\o}\big)$
that the pair $\big(\xi(t)^C, \chi(t)^C\big)$ has the following properties:

\begin{enumerate}[label=(\arabic*),leftmargin=.8cm]
\item
$\frac{d}{dt}\xi(t)^C=\d^C\chi(t)^C$;

\item
$\xi(0)^C\in Z_0\mathsf{End}^{\otimes}\big(C\!\otimes\! \bm{\o}\big)$;

\item
$\chi(t)^C_\Bbbk=0$;

\item
$\chi(t)^C_{M\otimes_{\cp_\O}M^\pr}
=\chi(t)^C_M\otimes_{\cp_C}\xi(t)^C_{M^\pr}+\xi(t)^C_M\otimes_{\cp_C}\chi(t)^C_{M^\pr}$
for all left dg-modules $\big(M,\g_M\big)$ and $\big(M^\pr,\g_{M^\pr}\big)$ over $\O$.
\end{enumerate}

For Property $(1)$, let $\xi(t)^C_M$ be the component of $\xi(t)^C$ at 
 a left dg-module $\big(M,\g_{\!M}\big)$ over $\O$. Then we have
\eqalign{
\frac{d}{dt}\xi(t)^C_M
=&\frac{d}{dt}
\check{\mp}
\Big(
\imath_M\circ\big(\ep_{C^\pr}\otimes\I_M\big)\circ\eta(t)^{C^\pr}_M\circ\big(f(t)\otimes\I_M\big)
\Big)\\
=&\check{\mp}
\Big(
\imath_M\circ\big(\ep_{C^\pr}\otimes\I_M\big)
\circ\rd_{C^\pr\otimes M,C^\pr\otimes M}\l(t)^{C^\pr}_M\circ\big(f(t)\otimes\I_M\big)
\Big)\\
&+\check{\mp}
\Big(
\imath_M\circ\big(\ep_{C^\pr}\otimes\I_M\big)\circ\eta(t)^{C^\pr}_M\circ\big(\rd_{C,C^\pr}s(t)\otimes\I_M\big)
\Big)\\
=&\big(\d^C\chi(t)^C\big)_M,
}
where we have used
$\rd_{C,C^\pr}f(t)=0$ and $\rd_{C^\pr\otimes M,C^\pr\otimes M}\eta(t)^{C^\pr}_M=0$ 
for the $3$rd equality.
Property $(2)$ is obvious since
$f(0):C\to C^\pr$ is a morphism of ccdg-coalgebras 
and $\eta(0)^{C^\pr}$ is in $Z_0\mathsf{End}^{\otimes}\big(C^\pr\!\otimes\! \bm{\o}\big)$.
Property $(3)$ follows from  $\l(t)^{C^\pr}_\Bbbk=0$, $\eta(t)^{C^\pr}_\Bbbk=\I^C_\Bbbk$ and $\ep_{C^\pr}\circ s(t)=0$, 
since we have
\eqalign{
\chi(t)^C_\Bbbk
=&
\check{\mp}
\Big(
m_\Bbbk\circ\big(\ep_{C^\pr}\otimes\I_\Bbbk\big)\circ \l(t)^{C^\pr}_\Bbbk\circ\big(f(t)\otimes\I_\Bbbk\big)
+m_\Bbbk\circ\big(\ep_{C^\pr}\otimes\I_\Bbbk\big)\circ \eta(t)^{C^\pr}_\Bbbk\circ\big(s(t)\otimes\I_\Bbbk\big)
\Big)\\
=&
\check{\mp}
\Big(
m_\Bbbk\circ\big(\ep_{C^\pr}\otimes\I_\Bbbk\big)\circ\big(s(t)\otimes\I_\Bbbk\big)
\Big)=0.
}
We note that Property $(4)$ is equivalent to the condition
\eqn{sdrwb}{
\check{\mq}\Big(\chi(t)^C_{M\otimes_{\cp_\O}M^\pr}\Big)
=\check{\mq}\Big(\chi(t)^C_M\otimes_{\cp_C}\xi(t)^C_{M^\pr}+\xi(t)^C_M\otimes_{\cp_C}\chi(t)^C_{M^\pr}\Big),
}
which can be checked as follows.
We consider the $1$st term in the RHS of \eq{sdrwb}:
\eqalign{
\check{\mq}\left(\chi(t)^C_M\otimes_{\cp_C}\xi(t)^C_{M^\pr}\right)
=&
\Big(\check{\mq}\big(\chi(t)^C_M\big)\otimes\check{\mq}\big(\xi(t)^C_{M^\pr}\big)\Big)
\circ(\I_C\otimes \t\otimes \I_{M^\pr})\circ(\cp_C\otimes\I_{M\otimes M^\pr})\\
=&
\Big(
\check{\mq}\big(\eta(t)^{C^\pr}_M\big)\otimes\check{\mq}\big(\eta(t)^{C^\pr}_{M^\pr}\big)
\Big)\\
&\circ
\big(s(t)\otimes\I_M\otimes f(t)\otimes\I_{M^\pr}\big)
\circ(\I_C\otimes \t\otimes \I_{M^\pr})\circ(\cp_C\otimes\I_{M\otimes M^\pr})\\
+&
\Big(
\check{\mq}\big(\l(t)^{C^\pr}_M\big)\otimes\check{\mq}\big(\eta(t)^{C^\pr}_{M^\pr}\big)
\Big)\\
&\circ
\big(f(t)\otimes\I_M\otimes f(t)\otimes\I_{M^\pr}\big)
\circ(\I_C\otimes \t\otimes \I_{M^\pr})\circ(\cp_C\otimes\I_{M\otimes M^\pr}).
}
Combining with the similar calculation for the $2$nd term in the RHS of \eq{sdrwb}, we obtain that
\[
\begin{aligned}
\check{\mq}\Big(\chi(t)^C_M & \otimes_{\cp_C}\xi(t)^C_{M^\pr}+\xi(t)^C_M\otimes_{\cp_C}\chi(t)^C_{M^\pr}\Big)\\
=&
\left(
\check{\mq}\big(\l(t)^{C^\pr}_M\big)\otimes \check{\mq}\big(\eta(t)^{C^\pr}_{M^\pr}\big)
+
\check{\mq}\big(\eta(t)^{C^\pr}_M\big)\otimes \check{\mq}\big(\l(t)^{C^\pr}_{M^\pr}\big)
\right)
\circ\big(f(t)\otimes\I_M\otimes f(t)\otimes \I_M\big)\\
&\circ(\I_C\otimes \t\otimes \I_{M^\pr})\circ(\cp_C\otimes\I_{M\otimes M^\pr})\\
+&\left(\check{\mq}\big(\eta(t)^{C^\pr}_M\big)\otimes \check{\mq}\big(\eta(t)^{C^\pr}_{M^\pr}\big)\right)
\circ
\big(s(t)\otimes\I_M\otimes f(t)\otimes \I_{M^\pr}+f(t)\otimes\I_M\otimes s(t)\otimes\I_{M^\pr}\big)\\
&\circ(\I_C\otimes \t\otimes \I_{M^\pr})\circ(\cp_C\otimes\I_{M\otimes M^\pr})\\
=&
\left(
\check{\mq}\big(\l(t)^{C^\pr}_M\big)\otimes \check{\mq}\big(\eta(t)^{C^\pr}_{M^\pr}\big)
+
\check{\mq}\big(\eta(t)^{C^\pr}_M\big)\otimes \check{\mq}\big(\l(t)^{C^\pr}_{M^\pr}\big)
\right)\\
&\circ
(\I_{C^\pr}\otimes\t\otimes\I_{M^\pr})\circ(\cp_{C^\pr}\otimes\I_{M\otimes M^\pr})\circ\big(f(t)\otimes\I_{M\otimes M^\pr}\big)\\
+&\left(\check{\mq}\big(\eta(t)^{C^\pr}_M\big)\otimes \check{\mq}\big(\eta(t)^{C^\pr}_{M^\pr}\big)\right)
\circ
(\I_{C^\pr}\otimes\t\otimes\I_{M^\pr})\circ(\cp_{C^\pr}\otimes\I_{M\otimes M^\pr})\circ\big(s(t)\otimes\I_{M\otimes M^\pr}\big)\\
=&
\check{\mq}\big(
\l(t)^{C^\pr}_M\otimes_{\cp_{C^\pr}}\eta(t)^{C^\pr}_{M^\pr}+
\eta(t)^{C^\pr}_M\otimes_{\cp_{C^\pr}}\l(t)^{C^\pr}_{M^\pr}
\big)
\circ\big(f(t)\otimes\I_{M\otimes M^\pr}\big)\\
+&\check{\mq}\big(\eta(t)^{C^\pr}_M\otimes_{\cp_{C^\pr}}\eta(t)^{C^\pr}_{M^\pr}\big)
\circ\big(s(t)\otimes\I_{M\otimes M^\pr}\big)\\
=&
\check{\mq}\big(\l(t)^{C^\pr}_{M\otimes_{\cp_\O}M^\pr}\big)\circ\big(f(t)\otimes \I_{M\otimes M^\pr}\big)
+\check{\mq}\big(\eta(t)^{C^\pr}_{M\otimes_{\cp_\O}M^\pr}\big)\circ\big(s(t)\otimes \I_{M\otimes M^\pr}\big)\\
=&
\check{\mq}\big(\chi(t)^{C^\pr}_{M\otimes_{\cp_\O}M^\pr}\big).
\end{aligned}
\]
In the above, 
we used $\big(s(t)\otimes f(t)+f(t)\otimes s(t)\big)\circ\cp_C=\cp_{C^\pr}\circ s(t)$ and
$\big(f(t)\otimes f(t)\big)\circ\cp_C=\cp_{C^\pr}\circ f(t)$ on the $2$nd equality, and used
$\l(t)^{C^\pr}_M\otimes\eta(t)^{C^\pr}_{M^\pr}+\eta(t)^{C^\pr}_M\otimes \l(t)^{C^\pr}_{M^\pr}
=\l(t)^{C^\pr}_{M\otimes_{\cp_\O}M^\pr}$ and
$\eta(t)^{C^\pr}_M\otimes_{\cp_{C^\pr}}\eta(t)^{C^\pr}_{M^\pr}=\eta(t)^{C^\pr}_{M\otimes_{\cp_\O}M^\pr}$
on the $4$th equality.
\qed
\end{proof}


Now we are ready to state the main theorem in this section.

\begin{theorem}\label{homainth}
We have a natural isomorphism of presheaves of groups
$$
\bm{\mP}^\otimes_{\!\bm{\o}}\cong \bm{\mP}_{\!\O}:
\mathring{\mathit{ho}\category{ccdgC}}(\Bbbk) \rightsquigarrow \category{Grp}
$$
on the homotopy category of ccdg-coalgebras.
Equivalently, the presheaf of groups 
$\bm{\mP}^\otimes_{\!\bm{\o}}$ on ${\mathit{ho}\category{ccdgC}}(\Bbbk)$ is representable
and represented by the ccdg-Hopf algebra $\O$.
\end{theorem}

The remaining part of this section is devoted to the proof of the above theorem, which
is  divided into several pieces.

\begin{proposition}\label{homainpr}
We have natural isomorphisms of presheaves
\eqalign{
\bm{\CE}_{\!\!\bm{\o}}\cong \bm{\CE}_{\O}:\mathring{\category{ccdgC}}(\Bbbk) \rightsquigarrow \category{dgA}(\Bbbk)
,\quad
\bm{\CP}^\otimes_{\!\!\bm{\o}}\cong \bm{\CP}_{\!\!\O}
:\mathring{\category{ccdgC}}(\Bbbk) \rightsquigarrow \category{Grp}
.
}
In particular the presheaf of groups $\bm{\CP}^\otimes_{\!\!\bm{\o}}$ on $\category{ccdgC}(\Bbbk)$ is representable
and represented by the ccdg-Hopf algebra $\O$.
\end{proposition}

The proof of this proposition is based on the forthcoming two lemmas.
Remind that in Lemma \ref{presheafone}, we defined the dg-algebra
$\bm{\CE}_{\!\!\O}(C)=\big(\Hom(C, \O), u_\O\circ \ep_C, \star_{C\!,\O},\rd_{C\!,\O}\big)$
for every ccdg-coalgebra $C$.

\begin{lemma}\label{ctanha}
 We have an isomorphism
$\xymatrix{\bm{\widebreve{\eta}}^C: \bm{\CE}_{\!\!\O}(C)
\ar@/^/[r] & \ar@/^/[l] \bm{\CE}_{\!\!\bm{\o}}(C):\bm{\widebreve{g}}^C}
$ 
of dg-algebras for every ccdg-coalgebra $C$, where
\begin{itemize}

\item
for each $\a\in \Hom(C,\O)$, the component of
$\bm{\widebreve{\eta}}^C(\a) \in \mathsf{End}\big(C\!\otimes\! \bm{\o}\big)$
at a left dg-module $\big(M,\g_M\big)$ over $\O$ is defined by
\eqalign{
\bm{\widebreve{\eta}}^C(\a)_{M}
:=&\check{\mp}\big(\g_M\circ(\alpha\otimes\I_M)\big)\\
=& (\I_C\otimes {\g}_M)\circ (\I_C\otimes \a\otimes \I_M)\circ (\cp_C\otimes \I_M):C\otimes M\to C\otimes M.
}

\item for  each ${\eta}^C \in\mathsf{End}\big(C\!\otimes\! \bm{\o}\big)$,
the linear map $\bm{\widebreve{g}}^C(\eta^C) \in \Hom(C, \O)$ is defined by
\eqalign{
\bm{\widebreve{g}}^C({\eta}^C)
:=&\check\mq(\eta^C_\O)\circ (\I_C\otimes u_\O)\circ \jmath^{-1}_C\\
=& \imath_\O\circ (\ep_C\otimes \I_\O)\circ {\eta}^C_{\O} \circ (\I_C\otimes u_\O)\circ \jmath^{-1}_C
:C\to \O.
}
\end{itemize}
\end{lemma}

\begin{proof}
The map $\bm{\widebreve{g}}^C$ is well-defined, 
since $\bm{\widebreve{g}}^C({\eta}^C)$ is obviously a $\Bbbk$-linear map.
We can check that  the map $\bm{\widebreve{\eta}}^C$ is also well-defined as follows.
For every morphism $\p:\big(M,\g_M\big)\to\big(M^\pr,\g_{M^\pr}\big)$ of left dg-modules over $\O$, 
the following  diagram commutes
\[
\xymatrixrowsep{3pc}
\xymatrixcolsep{3.3pc}
\xymatrix{
\ar@/^1.2pc/[rrr]^-{\bm{\widebreve{\eta}}^C(\alpha)_M}
C\otimes M \ar[r]_-{\cp_C\otimes\I_M} \ar[d]_-{\I_C\otimes \p}&
C\otimes C\otimes M \ar[r]_-{\I_C\otimes \a\otimes \I_M} \ar[d]^-{\I_{C\otimes C}\otimes \p}&
C\otimes \O\otimes M \ar[r]_-{\I_C\otimes \g_M} \ar[d]^-{\I_{C\otimes\O}\otimes \p}&
C\otimes M \ar[d]^-{\I_C\otimes \p}
\\
\ar@/_1.2pc/[rrr]_-{\bm{\widebreve{\eta}}^C(\alpha)_{M^\pr}}
C\otimes M^\pr \ar[r]^-{\cp_C\otimes\I_{M^\pr}}&
C\otimes C\otimes M^\pr \ar[r]^-{\I_C\otimes \a\otimes \I_{M^\pr}}&
C\otimes \O\otimes M^\pr \ar[r]^-{\I_C\otimes \g_{M^\pr}}&
C\otimes M^\pr
}
\]
so that $\bm{\widebreve{\eta}}^C(\alpha)$ is a natural transformation. 

Now we check that $\bm{\widebreve{g}}^C$ and $\bm{\widebreve{\eta}}^C$ are inverse to each other. 
\begin{itemize}
\item
$\bm{\widebreve{g}}^C\big(\bm{\widebreve{\eta}}^C(\a)\big)=\a$ holds for all $\a\in \Hom(C,\O)$:
\[
\begin{aligned}
\bm{\widebreve{g}}^C\big(\bm{\widebreve{\eta}}^C(\a)\big)
&=\check{\mq}\big(\bm{\widebreve{\eta}}^C(\a)_\O\big)\circ(\I_C\otimes u_\O)\circ\jmath^{-1}_C
=\check{\mp}\big(\check{\mq}\big(m_\O\circ(\alpha\otimes \I_\O)\big)\big)\circ(\I_C\otimes u_\O)\circ\jmath^{-1}_C\\
&=m_\O\circ(\alpha\otimes \I_\O)\circ(\I_C\otimes u_\O)\circ\jmath^{-1}_C=\alpha.
\end{aligned}
\]

\item
$\bm{\widebreve{\eta}}^C\big(\bm{\widebreve{g}}^C(\eta^C)\big)
=\eta^C$ holds for all $\eta^C\in\mathsf{End}\big(C\!\otimes\! \bm{\o}\big)$:
Let $\big(M,\g_M\big)$ be a left dg-module over $\O$.  Lemma \ref{modpr}\emph{(a)} states that 
$\g_M:\big(\O\otimes M,m_\O\otimes\I_M\big)\to\big(M,\g_M\big)$
is a morphism of left dg-modules over $\O$. Since $\eta^C$ is a natural transformation, the following diagram commutes
$$
\xymatrixrowsep{1.3pc}
\xymatrixcolsep{3pc}
\xymatrix{
\ar[d]_-{{\eta}^C_{\O\otimes M}=\eta^C_\O\otimes\I_M}
C\otimes \O\otimes M\ar[r]^-{\I_C\otimes {\g}_M} & C\otimes M
\ar[d]^-{{\eta}^C_{M}}
\cr
C\otimes \O\otimes M\ar[r]^-{\I_C\otimes{\g}_M} & C\otimes M
},\qquad\hbox{i.e.,}\qquad 
(\I_C\otimes {\g}_M)\circ ({\eta}^C_\O\otimes \I_M)= \eta^C_M \circ (\I_C\otimes {\g}_M).
$$
The equality on the diagram is by Lemma \ref{free modules and eta}. Thus we have
\eqalign{
\bm{\widebreve{\eta}}^C\big(\bm{\widebreve{g}}^C(\eta^C)\big)_M
=
&
(\I_C\otimes {\g}_M)\circ \Big(\I_C\otimes  \big( \hat\mq({\eta}^C_{\O})  
\circ (\I_C\otimes u_\O)\circ \jmath^{-1}_C\big)\otimes \I_M\Big)\circ (\cp_C\otimes \I_M)
\cr
=
&
(\I_C\otimes {\g}_M)\circ \Big( \hat\mp\big(\hat\mq({\eta}^C_{\O})\big)\otimes \I_M\Big)
\circ (\I_C\otimes u_\O\otimes \I_M)
\circ ( \jmath^{-1}_C\otimes \I_M)
\cr
=
&
(\I_C\otimes {\g}_M)\circ (\eta_\O^C\otimes \I_M)
\circ (\I_C\otimes u_\O\otimes \I_M)\circ (\jmath^{-1}_C\otimes \I_M)
\cr
=
&
\eta^C_M\circ (\I_C\otimes {\g}_M)
\circ (\I_C\otimes u_\O\otimes \I_M)\circ (\jmath^{-1}_C\otimes \I_M)
=\eta^C_M.
}
\end{itemize}

We are left to show that $\bm{\widebreve{\eta}}^C$ and $\bm{\widebreve{g}}^C$ are morphisms of dg-algebras.
Since they are inverse to each other, 
it suffices to show that $\bm{\widebreve{\eta}}^C$ is a morphism of dg-algebras. 
Clearly, $\bm{\widebreve{\eta}}^C$ is a $\Bbbk$-linear map of degree $0$.
Let $\big(M,\g_M\big)$ be a left dg-module over $\O$.
\begin{itemize}
\item
$\bm{\widebreve{\eta}}^C$ is a chain map, 
i.e. $\d^C\circ\bm{\widebreve{\eta}}^C=\bm{\widebreve{\eta}}^C\circ\rd_{C\!,\O}$. Indeed, for $\a \in \Hom(C,\O)$,
\[
\begin{aligned}
\d^C\big(\bm{\widebreve{\eta}}^C(\alpha)\big)_M
&=\rd_{C\otimes M,C\otimes M}\big((\I_C\otimes {\g}_M)\circ (\I_C\otimes \a\otimes \I_M)\circ (\cp_C\otimes \I_M)\big)\\
&=(\I_C\otimes {\g}_M)\circ \big(\I_C\otimes \rd_{C\!,\O}\a\otimes \I_M\big)\circ (\cp_C\otimes \I_M)\\
&=\bm{\widebreve{\eta}}^C(\rd_{C\!,\O}\alpha)_M.
\end{aligned}
\]
The $2$nd equality follows from the properties 
$\rd_{\O\otimes M,M}\g_M=0$ and $\rd_{C,C\otimes C}\cp_C=0$.

\item
$\bm{\widebreve{\eta}}^C$ sends the identity to the identity, i.e. $\bm{\widebreve{\eta}}^C(u_\O\circ \ep_C)=\I^C$:
$$
\bm{\widebreve{\eta}}^C(u_\O\circ \ep_C)_{M} 
:=  (\I_C\otimes {\g}_M)\circ \big(\I_C\otimes( u_\O\circ \ep_C) \otimes \I_M\big)\circ (\cp_C\otimes \I_M)
=\I_{C\otimes M}
=\I^C_M
.
$$

\item
$\bm{\widebreve{\eta}}^C$ preserves the binary operations, 
i.e. $\bm{\widebreve{\eta}}^C(\a_1\star_{C\!,\O} \a_2)
=\bm{\widebreve{\eta}}^C(\a_1)\circ \bm{\widebreve{\eta}}^C(\a_2)$
for all $\a_1,\a_2 \in \Hom(C,\O)$:
\eqalign{
\bm{\widebreve{\eta}}^C \big(\a_1 &\star_{C\!,\O} \a_2\big)_{M} 
:=
(\I_C\otimes {\g}_M)\circ \Big(\I_C\otimes \big(m_\O\circ (\a_1\otimes \a_2)\circ \cp_C\big)\otimes \I_M\Big)
\circ (\cp_C\otimes \I_M)
\cr
=
&
(\I_C\otimes {\g}_M)\circ (\I_C\otimes \a_1\otimes \I_M)\circ (\cp_C\otimes \I_M)\circ (\I_C\otimes {\g}_M)
\circ (\I_C\otimes \a_2\otimes \I_M)\circ (\cp_C\otimes \I_M)
\cr
=
&\bm{\widebreve{\eta}}^C\big(\a_1\big)_{M} \circ \bm{\widebreve{\eta}}^C\big(\a_2\big)_{M}.
}
The $2$nd equality is due to the coassociativity of $\cp_C$ and the action axiom of $\g_M$.
\end{itemize}
\qed
\end{proof}

In Lemma \ref{presheaftwo}, we showed that
$\bm{\CP}_{\!\!\O}(C)=\big(\HOM_\ccdgc(C,\O), u_\O\circ\ep_C, \star_{C\!,\O}\big)$
is a group for every ccdg-coalgebra $C$. 
The inverse of $g \in \HOM_\ccdgc(C,\O)$ is given by $g^{-1}:=\vs_\O\circ g$.
Remind that 
$\HOM_\ccdgc(C,\O)$ is the subset of $\Hom(C,\O)$ consisting the morphisms of ccdg-coalgebras:
$$
\HOM_\ccdgc(C,\O)
=\Big\{ g \in \Hom(C, \O)_0\Big|\rd_{C\!,\O} g =0, 
\;\cp_\O\circ g =(g\otimes g)\circ \cp_C,
\;   \ep_\O\circ g =\ep_C \Big\}.
$$

\begin{lemma}\label{ctanhc}
For  every ccdg-coalgebra $C$, the isomorphism in Lemma \ref{ctanha} gives an isomorphism
$\xymatrix{\bm{\widebreve{\eta}}^C: \bm{\CP}_{\!\!\O}(C)\ar@/^/[r] 
& \ar@/^/[l] \bm{\CP}^\otimes_{\!\!\bm{\o}}(C):\bm{\widebreve{g}}^C}$ 
of groups.
\end{lemma}

\begin{proof}
We only need to show two things: 
$\bm{\widebreve{g}}^C\left(Z_0\mathsf{End}^\otimes\big(C\!\otimes\! \bm{\o}\big)\right) \subset \HOM_{\ccdgc}(C, \O)$ 
and $\bm{\widebreve{\eta}}^C\left(\HOM_{\ccdgc}(C, \O)\right) 
\subset Z_0\mathsf{End}^\otimes\big(C\!\otimes\! \bm{\o}\big)$.

1. For ${\eta}^C \in Z_0\mathsf{End}^\otimes\big(C\!\otimes\! \bm{\o}\big)$
we have $\bm{\widebreve{g}}^C({\eta}^C)\in \HOM_{\ccdgc}(C, \O)$.
\begin{itemize}
\item
$\bm{\widebreve{g}}^C({\eta}^C)$ is of degree $0$ and $\rd_{C\!,\O}\bm{\widebreve{g}}^C({\eta}^C)= 0$: 
This is immediate since ${\eta}^C$ is of degree $0$ with $\d^C\eta^C=0$, and $\bm{\widebreve{g}}^C$ is a chain map 
by Lemma \ref{ctanha}.

\item
$\ep_\O\circ \bm{\widebreve{g}}^C({\eta}^C) =\ep_C$: 
Lemma \ref{modpr}\emph{(b)} states that $\ep_\O:\big(\O,m_\O\big)\to\big(\Bbbk,\g_\Bbbk\big)$ 
is a morphism of left dg-modules over $\O$. Since $\eta^C$ is a tensor natural transformation, the following diagram commutes:
\[
\xymatrixrowsep{1.3pc}
\xymatrixcolsep{3.5pc}
\xymatrix{
\ar[d]_-{{\eta}^C_{\O}}
C\otimes \O \ar[r]^{\I_C\otimes \ep_\O} &C\otimes \Bbbk
\ar[d]^-{{\eta}^C_\Bbbk=\I_{C\otimes \Bbbk}= \I^C_\Bbbk}
\cr
C\otimes \O \ar[r]^{\I_C\otimes \ep_\O} &C\otimes \Bbbk
,}
\quad\hbox{i.e.,}\quad
(\I_C\otimes \ep_\O)\circ  {\eta}^C_{\O}= \I_C\otimes \ep_\O.
\]
Therefore we have
\eqalign{
\ep_\O\circ \bm{\widebreve{g}}^C({\eta}^C)
&
=m_\Bbbk\circ(\ep_C\otimes\I_\Bbbk)\circ(\I_C\otimes\ep_\O)\circ\eta^C_\O\circ(\I_C\otimes u_\O)
\circ\jmath^{-1}_C\\
&=m_\Bbbk\circ(\ep_C\otimes\I_\Bbbk)\circ(\I_C\otimes\ep_\O)\circ(\I_C\otimes u_\O)\circ\jmath^{-1}_C
=\ep_C
.
}


\item
$\cp_\O\circ \bm{\widebreve{g}}^C({\eta}^C) = \left(\bm{\widebreve{g}}^C({\eta}^C)\otimes \bm{\widebreve{g}}^C({\eta}^C)\right)\circ \cp_C$:
Lemma \ref{modpr}\emph{(c)} states that 
$\cp_\O:\big(\O,m_\O\big)\to\big(\O\otimes\O,\g_{\O\otimes_{\cp_\O}\O}\big)$ 
is a morphism of left dg-modules over $\O$.
Since $\eta^C$ is a tensor natural transformation, the following diagram commutes:
\[
\xymatrixrowsep{1.3pc}
\xymatrixcolsep{2.7pc}
\xymatrix{
\ar[d]_-{{\eta}^C_{\O}}
C\otimes \O\ar[r]^-{\I_C\otimes \cp_\O} &C\otimes\O\otimes \O
\ar[d]^-{{\eta}^C_{\O\otimes_{\cp_\O\!} \O}=\eta^C_{\O}\otimes_{\cp_C}\eta^C_{\O}}
\cr
C\otimes \O\ar[r]^-{\I_C\otimes \cp_\O} &C\otimes\O\otimes \O
,}
\quad\hbox{i.e.,}\quad
(\I_C\otimes \cp_\O)\circ {\eta}^C_{\O} =\big( \eta^C_{\O}\otimes_{\cp_C}\eta^C_{\O}\big)\circ (\I_C\otimes \cp_\O).
\]
Therefore we obtain that
\eqalign{
\cp_\O\circ \bm{\widebreve{g}}^C({\eta}^C)
&=\imath_{\O\otimes\O}\circ(\ep_C\otimes\I_{\O\otimes\O})\circ(\I_C\otimes \cp_\O)
\circ\eta^C_\O\circ(\I_C\otimes u_\O)\circ\jmath^{-1}_C\\
&=\imath_{\O\otimes\O}\circ(\ep_C\otimes\I_{\O\otimes\O})\circ\big(\eta^C_\O\otimes_{\cp_C}\eta^C_\O\big)
\circ(\I_C\otimes \cp_\O)\circ(\I_C\otimes u_\O)\circ\jmath^{-1}_C\\
&=\left(\bm{\widebreve{g}}^C(\eta^C) \otimes \bm{\widebreve{g}}^C(\eta^C)\right)\circ \cp_C
.
}
\end{itemize}

2. 
For $g \in \HOM_{\ccdgc}(C,\O)$, we have 
$\bm{\widebreve{\eta}}^C(g) \in Z_0\mathsf{End}^\otimes\big(C\!\otimes\! \bm{\o}\big)$.

\begin{itemize}

\item $\bm{\widebreve{\eta}}^C(g)$ is of degree $0$ and satisfies $\d^C \bm{\widebreve{\eta}}^C(g)=0$: 
This is immediate, since $g$ is of degree $0$ with $\rd_{C\!,\O}g=0$, 
and $\bm{\widebreve{\eta}}^C$ is a chain map by Lemma \ref{ctanha}.

\item $\bm{\widebreve{\eta}}^C(g)_{\Bbbk}= \I^C_\Bbbk$: Using $\ep_\O\circ g=\ep_C$, we have
\[
\begin{aligned}
\bm{\widebreve{\eta}}^C(g)_{\Bbbk}
&=(\I_C\otimes m_\Bbbk)\circ(\I_C\otimes\ep_\O\otimes\I_\Bbbk)\circ(\I_C\otimes g\otimes \I_\Bbbk)\circ(\cp_C\otimes\I_\Bbbk)\\
&=(\I_C\otimes m_\Bbbk)\circ(\I_C\otimes\ep_\O\otimes\I_\Bbbk)\circ(\cp_C\otimes\I_\Bbbk)=\I^C_\Bbbk.
\end{aligned}
\]

\item
$\bm{\widebreve{\eta}}^C(g)_{M\otimes_{\cp_\O}\! M^\pr}
=\bm{\widebreve{\eta}}^C(g)_M\otimes_{\cp_C}\! \bm{\widebreve{\eta}}^C(g)_{M^\pr}$
for every left dg-modules $\big(M,\g_M\big)$ and $\big(M^\pr,\g_{M^\pr}\big)$ over $\O$:
This  is equivalent to the condition 
$\check{\mq}\Big(
\bm{\widebreve{\eta}}^C(g)_M\otimes_{\cp_C}\! \bm{\widebreve{\eta}}^C(g)_{M^\pr}
\Big)=\check{\mq}\big(\bm{\widebreve{\eta}}^C(g)_{M\otimes_{\cp_\O}\! M^\pr}\big)$.
Using $\cp_\O\circ g=(g\otimes g)\circ\cp_C$, we have
\eqalign{
\check{\mq}\Big(
\bm{\widebreve{\eta}}^C(g)_M\otimes_{\cp_C}\! &\bm{\widebreve{\eta}}^C(g)_{M^\pr}
\Big)
:=
\Big(
\check{\mq}\big(\bm{\widebreve{\eta}}^C(g)_M\big)\otimes
\check{\mq}\big(\bm{\widebreve{\eta}}^C(g)_{M^\pr}\big)
\Big)
\circ(\I_C\otimes \t\otimes \I_{M^\pr})\circ(\cp_C\otimes\I_{M\otimes M^\pr})
\\
&=
(\g_M\otimes \g_{M^\pr})\circ(g\otimes \I_M\otimes g\otimes \I_{M^\pr})
\circ(\I_C\otimes \t\otimes \I_{M^\pr})\circ(\cp_C\otimes\I_{M\otimes M^\pr})
\\
&=
(\g_M\otimes\g_{M^\pr})\circ(\I_\O\otimes \t\otimes \I_{M^\pr})\circ(\cp_\O\otimes\I_{M\otimes M^\pr})\circ(g\otimes\I_{M\otimes M^\pr})
\\
&=\g_{M\otimes_{\cp_\O}M^\pr}\circ(g\otimes \I_{M\otimes M^\pr})
=\check{\mq}\big(\bm{\widebreve{\eta}}^C(g)_{M\otimes_{\cp_\O}\! M^\pr}\big).
}
\end{itemize}
\qed
\end{proof}

Now we finish the proof of Proposition \ref{homainpr}.

\begin{proof}[Proposition \ref{homainpr}]
We claim that the isomorphisms $\bm{\widebreve{\eta}}^C:\bm{\CE}_{\!\!\O}(C)\to\bm{\CE}_{\!\!\bm{\o}}(C)$ 
are natural in $C\in\ccdgc$. This will give us a natural isomorphism
\[
\bm{\widebreve{\eta}}:\bm{\CE}_{\!\!\O}\Longrightarrow\bm{\CE}_{\!\!\bm{\o}}
:\mathring{\category{ccdgC}}(\Bbbk) \rightsquigarrow \category{dgA}(\Bbbk),
\]
whose component at a ccdg-coalgebra $C$ is $\bm{\widebreve{\eta}}^C$. Then 
$\bm{\widebreve{g}}:=\{\bm{\widebreve{g}}^C\}$ is also a natural transformation, 
which is the inverse of $\bm{\widebreve{\eta}}$. 
Moreover, $\bm{\widebreve{\eta}}$ will canonically induce a natural isomorphism
$$
\bm{\widebreve{\eta}}:\bm{\CP}_\O\Longrightarrow\bm{\CP}^\otimes_{\!\!\bm{\o}}
:\mathring{\category{ccdgC}}(\Bbbk) \rightsquigarrow \category{Grp},
$$
with its inverse, again, $\bm{\widebreve{g}}$.
We need to show that the following diagram commutes for every morphism $f:C\to C^\pr$ of ccdg-coalgebras:
\[
\xymatrixrowsep{1.3pc}
\xymatrix{
\bm{\CE}_{\!\!\O}(C^\pr) \ar[r]^-{\bm{\widebreve{\eta}}^{C^\pr}} \ar[d]_-{\bm{\CE}_{\!\!\O}(f)}&
\bm{\CE}_{\!\!\bm{\o}}(C^\pr) \ar[d]^-{\bm{\CE}_{\!\!\bm{\o}}(f)}\\
\bm{\CE}_{\!\!\O}(C) \ar[r]^-{\bm{\widebreve{\eta}}^C}&
\bm{\CE}_{\!\!\bm{\o}}(C),
}
\quad\text{i.e.,}\quad
\bm{\CE}_{\!\!\bm{\o}}(f)\circ\bm{\widebreve{\eta}}^{C^\pr}=\bm{\widebreve{\eta}}^C\circ\bm{\CE}_{\!\!\O}(f).
\]
Let $g^\pr:C^\pr\to \O$ be a linear map and $\big(M,\g_M\big)$ be a left dg-module over $\O$. Then
\eqalign{
\check{\mq}\Big(
\bm{\CE}_{\!\!\bm{\o}}(f)\big(\bm{\widebreve{\eta}}^{C^\pr}(g^\pr)\big)_M
\Big)
&=
\check{\mq}\big(\bm{\widebreve{\eta}}^{C^\pr}(g^\pr)_M\big)
\circ(f\otimes \I_M)\\
&=\g_M\circ(g^\pr\otimes \I_M)\circ(f\otimes\I_M)\\
&=\g_M\circ\big((g^\pr\circ f)\otimes \I_M\big)
=\check{\mq}\Big(\bm{\widebreve{\eta}}^{C^\pr}(g^\pr\circ f)_M\Big).
}
Therefore $\big(\bm{\CE}_{\!\!\bm{\o}}(f)\circ\bm{\widebreve{\eta}}^{C^\pr}\big)(g^\pr)
=\big(\bm{\widebreve{\eta}}^C\circ\bm{\CE}_{\!\!\O}(f)\big)(g^\pr)$
holds for all $g^\pr:C^\pr\to \O$.
\qed
\end{proof}

We end this paper with the proof of Theorem \ref{homainth}.

\begin{proof}[Theorem \ref{homainth}]
By Proposition \ref{homainpr} and the definitions of 
$\bm{\mP}_\O$ and $\bm{\mP}_{\bm{\o}}^\otimes$, 
it is sufficient  to show that, for every ccdg-coalgebra $C$,
\begin{enumerate}[label=$({\alph*})$,leftmargin=.8cm]

\item $\bm{\widebreve{\eta}}^C$ sends a homotopy pair 
$\big(g(t),\chi(t)\big)$ on $\HOM_\ccdgc(C,\O)$
to a homotopy pair
$\Big(\bm{\widebreve{\eta}}^C\big(g(t)\big),\bm{\widebreve{\eta}}^C\big(\chi(t)\big)\Big)$
on $Z_0\mathsf{End}^\otimes\big(C\!\otimes\! \bm{\o}\big)$, and

\item $\bm{\widebreve{g}}^C$ sends a homotopy pair
$\big(\eta(t)^C,\l(t)^C\big)$ on $Z_0\mathsf{End}^\otimes\big(C\!\otimes\! \bm{\o}\big)$
to a homotopy pair $\Big(\bm{\widebreve{g}}^C\big(\eta(t)^C\big),\bm{\widebreve{g}}^C\big(\l(t)^C\big)\Big)$
on $\HOM_\ccdgc(C,\O)$.

\end{enumerate}
Then $\bm{\widebreve{\eta}}^C$ and $\bm{\widebreve{g}}^C$ will induce an isomorphism of groups 
$\bm{\mP}^\otimes_{\!\bm{\o}}(C)\cong \bm{\mP}_{\!\O}(C)$ for every ccdg-coalgebra $C$. 
Moreover, this isomorphism is natural in $C\in\ccdgc$ by Proposition \ref{homainpr} and Lemma \ref{costpdy}
so that we have natural isomorphism 
$$
\bm{\mP}_{\!\!\O}\cong\bm{\mP}^\otimes_{\!\!\bm{\o}}:\mathring{\mathit{ho}\category{ccdgC}}(\Bbbk) 
\rightsquigarrow \category{Grp}.
$$

We will prove $(a)$ only since the proof of $(b)$ is similar.
We need to check the pair $\Big(\bm{\widebreve{\eta}}^C\big(g(t)\big),\bm{\widebreve{\eta}}^C\big(\chi(t)\big)\Big)$
has the following properties.
\begin{enumerate}[label=(\arabic*),leftmargin=.8cm]
\item $\frac{d}{dt}\bm{\widebreve{\eta}}^C\big(g(t)\big)=\d^C\bm{\widebreve{\eta}}^C\big(\chi(t)\big)$,
\item $\bm{\widebreve{\eta}}^C\big(g(0)\big)\in Z_0\mathsf{End}^\otimes\big(C\!\otimes\! \bm{\o}\big)$,
\item $\bm{\widebreve{\eta}}^C\big(\chi(t)\big)_\Bbbk=0$, and
\item $\bm{\widebreve{\eta}}^C \big(\chi(t)\big)_M  \otimes_{\cp_C}\bm{\widebreve{\eta}}^C\big(g(t)\big)_{M^\pr}
+\bm{\widebreve{\eta}}^C\big(g(t)\big)_M\otimes_{\cp_C}\bm{\widebreve{\eta}}^C\big(\chi(t)\big)_{M^\pr}
=\bm{\widebreve{\eta}}^C\big(\chi(t)\big)_{M\otimes_{\cp_\O}M^\pr}$,
\end{enumerate}
where the last equality should hold for all left dg-modules $\big(M,\g_M\big)$ and $\big(M^\pr,\g_{M^\pr}\big)$ over $\O$.
Property $(1)$ follows from the condition $\frac{d}{dt}g(t)=\rd_{C\!,\O}\chi(t)$,
since
$\frac{d}{dt}\bm{\widebreve{\eta}}^C\big(g(t)\big)
=\bm{\widebreve{\eta}}^C\left(\frac{d}{dt}g(t)\right)
=\bm{\widebreve{\eta}}^C\left(\rd_{C\!,\O}\chi(t)\right)
=\d^C\bm{\widebreve{\eta}}^C\big(\chi(t)\big)$.
Property $(2)$ follows from the condition $g(0)\in \HOM_\ccdgc(C,\O)$.
Property $(3)$ follows from the condition $\ep_\O\circ \chi(t)=0$, since
we have
$\bm{\widebreve{\eta}}^C\big(\chi(t)\big)_\Bbbk
=(\I_C\otimes m_\Bbbk)\circ\big(\I_C\otimes(\ep_{\O}\circ \chi(t))\otimes\I_\Bbbk\big)\circ(\cp_C\otimes\I_\Bbbk)$.
Finally we check that Property $(4)$ is a consequence of 
the condition $\big(\chi(t)\otimes g(t)+g(t)\otimes \chi(t)\big)\circ \cp_C=\cp_\O\circ \chi(t)$.
We note that Property $(4)$ is equivalent to the identity
\eqn{sdrwa}{
\check{\mq}
\Big(
\bm{\widebreve{\eta}}^C \big(\chi(t)\big)_M  \otimes_{\cp_C}\bm{\widebreve{\eta}}^C\big(g(t)\big)_{M^\pr}
+\bm{\widebreve{\eta}}^C\big(g(t)\big)_M\otimes_{\cp_C}\bm{\widebreve{\eta}}^C\big(\chi(t)\big)_{M^\pr}
\Big)=\check{\mq}\big(\bm{\widebreve{\eta}}^C\big(\chi(t)\big)_{M\otimes_{\cp_\O}M^\pr}\big),
}
which can be checked as follows.  We begin with the $1$st term in the LHS of \eq{sdrwa}:
\eqalign{
\check{\mq}
\Big(
\bm{\widebreve{\eta}}^C \big(\chi(t)\big)_M &  \otimes_{\cp_C}\bm{\widebreve{\eta}}^C\big(g(t)\big)_{M^\pr}
\Big)\\
&=
\Big(
\check{\mq}\big(\bm{\widebreve{\eta}}^C \big(\chi(t)\big)_M\big)
\otimes
\check{\mq}\big(\bm{\widebreve{\eta}}^C\big(g(t)\big)_{M^\pr}\big)
\Big)
\circ(\I_C\otimes\t\otimes\I_{M^\pr})\circ(\cp_C\otimes\I_{M\otimes M^\pr})\\
&=
(\g_M\otimes \g_{M^\pr})\circ\big(\chi(t)\otimes\I_M\otimes g(t)\otimes \I_{M^\pr}\big)
\circ(\I_C\otimes\t\otimes\I_{M^\pr})\circ(\cp_C\otimes\I_{M\otimes M^\pr})\\
&=
(\g_M\otimes \g_{M^\pr})\circ(\I_\O\otimes\t\otimes\I_{M^\pr})\circ
\big(\chi(t)\otimes g(t)\otimes\I_{M\otimes M^\pr}\big)\circ(\cp_C\otimes\I_{M\otimes M^\pr}).
}
After the similar calculation for the $2$nd term in the LHS of \eq{sdrwa},
we obtain that
\eqalign{
\check{\mq}
\Big(&
\bm{\widebreve{\eta}}^C \big(\chi(t)\big)_M  \otimes_{\cp_C}\bm{\widebreve{\eta}}^C\big(g(t)\big)_{M^\pr}
+
\bm{\widebreve{\eta}}^C\big(g(t)\big)_M\otimes_{\cp_C}\bm{\widebreve{\eta}}^C\big(\chi(t)\big)_{M^\pr}
\Big)\\
&=
(\g_M\otimes \g_{M^\pr})\circ(\I_\O\otimes\t\otimes\I_{M^\pr})\circ
\Big(\big(\chi(t)\otimes g(t)+g(t)\otimes \chi(t)\big)\otimes\I_{M\otimes M^\pr}\Big)
\circ(\cp_C\otimes\I_{M\otimes M^\pr})\\
&=
(\g_M\otimes \g_{M^\pr})\circ(\I_\O\otimes\t\otimes\I_{M^\pr})\circ
(\cp_\O\otimes\I_{M\otimes M^\pr})\circ \big(\chi(t)\otimes \I_{M\otimes M^\pr}\big)\\
&=
\g_{M\otimes_{\cp_\O}M^\pr}\circ\big(\chi(t)\otimes \I_{M\otimes M^\pr}\big)
=\check{\mq}\big(\bm{\widebreve{\eta}}^C\big(\chi(t)\big)_{M\otimes_{\cp_\O}M^\pr}\big).
}
\qed
\end{proof}


\begin{thebibliography}{AB}

\bibitem{Adams}
J.F.\  Adams, 
\emph{On the cobar construction}.
Proc.\ Natl.\ Acad.\ Sci.\ USA 42 (1956), 409--412.

\bibitem{Cartier}
P.\ Cartier,
\emph{A primer of Hopf algebras}. In: Frontiers in number theory, physics, and geometry. II., pp. 537--615,  Springer, Berlin, 2007.


\bibitem{Dieud}
J.\ Dieudonn\'{e},
\emph{Introduction to the theory of formal groups}. 
Pure and Applied Mathematics, 20. Marcel Dekker, Inc., New York, 1973.





\bibitem{Deligne90}
P.\ Deligne,
\emph{Cat\'egories tannakiennes}.  
The Grothendieck Festschrift. Vol.\ II,
 pp.\ 111--195,
Progr.\ Math.\ 87, Birkh\"auser Boston, Boston, MA, 1990. 



%
\bibitem{Keller}
B.\ Keller,  \emph{On differential graded categories}. 
International Congress of Mathematicians. Vol. II, pp 151--190, Eur.\ Math.\ Soc.\, Zürich, 2006. 



\bibitem{JLP}
J.\ Lee and J.-S.\ Park, 
\emph{Tannakian reconstruction of affine group dg-scheme and
a rational de Rham fundamental group dg-scheme of a space}.
In prepration.

\bibitem{Quillen}
D.G.\ Quillen, 
\emph{Rational homotopy theory}.
Ann.\ of Math.\ 90 (1969), 205--295.

\bibitem{Rivano}
N.\ Saavedra Rivano, 
\emph{Cat\'egories Tannakiennes},
Lecture Notes in Math., Vol.\ 265.
Springer-Verlag, Berlin. 1972.


\bibitem{Simpson}
C.T.\ Simpson, 
\emph{Higgs bundles and local systems}, 
Publ.\ Math.\ Inst.\ Hautes \'Etudues Sci.\   75  (1992) 5--95.


\bibitem{Sullivan}
D.\ Sullivan, 
\emph{Infinitesimal computations in topology}.
Publ.\ Math.\ Inst.\ Hautes \'Etudues  Sci.\  No.\ 47 (1978), 269--331.

\end{thebibliography}
\end{document}